\documentclass[a4paper,11pt]{amsart}

\usepackage{amsmath,amsthm,amsfonts,amssymb,graphics}
\usepackage{mathtools}
\usepackage[foot]{amsaddr}
\usepackage[colorlinks=true,linkcolor=black]{hyperref}

\usepackage{enumerate}
\usepackage{geometry}
\usepackage{dsfont}
\usepackage{caption,subcaption}
\usepackage{circledsteps}
\usepackage{tikz}
\usetikzlibrary{
patterns,
calc,
arrows.meta,
decorations.pathreplacing}

\usepackage[
  backend=biber,
  style=trad-abbrv,   
  maxnames=5,         
  giveninits=true,    
  doi=false,
  url=false,
  isbn=false,
  eprint=false,
  sorting=nyt,
  backref=true
]{biblatex}
\newtheorem{theorem}{Theorem}[section]
\newtheorem{lemma}[theorem]{Lemma}
\newtheorem{proposition}[theorem]{Proposition}
\newtheorem{corollary}[theorem]{Corollary}
\newtheorem*{assumption}{Assumption}
\newtheorem{question}[theorem]{Question}
\theoremstyle{remark}
\newtheorem{remark}[theorem]{Remark}
\newtheorem{definition}[theorem]{Definition}

\numberwithin{equation}{section}

\newcommand {\R} {\mathbb{R}}
\newcommand {\E} {\mathbb{E}}
\newcommand {\N} {\mathbb{N}}
\newcommand {\Z} {\mathbb{Z}}
\renewcommand{\P} {\mathbb{P}}
\newcommand {\Var} {\mathrm{Var}}
\newcommand {\Vol} {\mathrm{Vol}}
\newcommand {\EC} {\mathrm{EC}}
\newcommand {\diam} {\mathrm{diam}}
\newcommand {\Cov} {\mathrm{Cov}}
\newcommand {\dist} {\mathrm{dist}}

\newcommand{\ind}{\mathds{1}}

\newenvironment{namedassumption}[2]{
  \begin{assumption}[\textbf{#1}]\label{#2}}{\end{assumption}}
\newcommand{\aref}[1]{\hyperref[#1]{\textbf{(\nameref*{#1})}}}

\begin{document}
\title[]{Limit theorems for non-local functionals of smooth Gaussian fields via quasi-association}
\author{Michael McAuley\textsuperscript{1}}
\address{\textsuperscript{1}School of Mathematics and Statistics, Technological University Dublin}
\email{m.mcauley@cantab.net}
\begin{abstract}
Many classical objects of study related to the geometry/topology of smooth Gaussian fields (e.g.\ the volume, surface area or Euler characteristic of excursion sets) have a `locality' property which is crucial to their analysis. More recently, progress has been made in studying `non-local' quantities of such fields (e.g.\ the component/nodal count or the Betti numbers of excursion sets). In this work, we develop a new approach to analysing non-local functionals based on a form of topological quasi-association. We use this to establish a variety of limit theorems for approximately additive functionals on growing domains, including concentration bounds, a quantitative central limit theorem and the law of the iterated logarithm.
\end{abstract}
\date{\today}
\thanks{}

\maketitle
\section{Introduction}\label{s:intro}
\subsection{Motivation}
Smooth Gaussian fields are widely used across science for modelling spatial phenomena. For instance, such fields are used in cosmology to describe the cosmic microwave background radiation \cite{pc14} or the distribution of galaxies \cite{bbk86}, in quantum chaos to understand high-energy quantum wavefunctions \cite{ber77} and to model the noise process in medical imaging \cite{wor96}. Within mathematics, smooth Gaussian fields have been used to describe `typical' real algebraic varieties as a way of addressing Hilbert's 16th problem \cite{ll15}, whilst there has also been significant progress in recent years in understanding the percolation behaviour (i.e., large scale topology) of level sets of Gaussian fields \cite{bel23}.

In many of these applications, one is naturally interested in geometric/topological properties of the field (such as the volume or number of connected components of the level sets of the field). A rich theory has developed to describe the statistics of geometric/topological functionals of Gaussian fields. This includes tools like the Kac-Rice formula for computing moments \cite{bll22} and the Malliavin-Stein method for establishing quantitative distributional limit theorems \cite{np12}. Such tools rely heavily on the functionals being \emph{local}, in the sense that they are determined by summing infinitesimal contributions. More precisely, given a $C^2$-smooth Gaussian field $f$ defined on $\R^d$, one can describe a functional $\Phi$ as local if it can be expressed as
\[
    \Phi(D,f)=\int_D\varphi(f(x),\nabla f(x),\nabla^2 f(x))\;\mu(dx),\qquad\text{for compact }D\subset\R^d,
\]
for some real-valued function $\varphi$ and some measure $\mu$. For example, the volume of level sets can be expressed in this form by choosing $\varphi(f(x),\nabla f(x),\nabla^2f(x))=\ind_{f(x)=\ell}$ and $\mu$ to be $(d-1)$-dimensional Hausdorff measure. Many geometric functionals are local, whereas most topological functionals are non-local. (One notable exception is the Euler characteristic of excursion sets, which is local by the Poincar\'e-Hopf theorem.)

In this work, we will mainly be interested in non-local functionals which are somehow `approximately additive' over disjoint domains, since such functionals could be expected to behave analogously to local functionals. For example, the number of connected components of the level set of a Gaussian field would be additive over domains if the components intersecting the boundary could be neglected. Whilst progress has been made recently in studying such non-local functionals mathematically, there is still a large gap between what is known in the two cases. This is particularly dissatisfying, given that the known results all indicate that approximately additive, non-local functionals do in fact exhibit similar behaviour to their local counterparts (more detail will be given in the next subsection). It would therefore be desirable to have a unified framework which can analyse both classes of functional.

In an effort to address this gap, we develop a new method for studying non-local functionals of Gaussian fields on Euclidean space, which we use to establish a variety of limit theorems. Our approach uses a covariance formula for topological events, developed in \cite{bmr20}, to show \emph{quasi-association} of the functionals. This property has been widely used to establish limit theorems for sums of discrete random fields and many of the relevant arguments can be adapted to our setting. Whilst our method does not immediately apply to geometric functionals, it could in principle be generalised to this case. If so, this may lead to a unified framework for local and (approximately additive) non-local functionals, at least for weakly dependent Gaussian fields. We discuss this further in Section~\ref{ss:discussion}, after presenting our results.

\subsection{Related literature}
To provide context for our work, we will briefly summarise some of the literature on geometric/topological functionals of Gaussian fields, focusing on results which are most closely related to our own. We are interested in the following setting: let $f:\R^d\to\R$ be a stationary, Gaussian field which is $C^2$-smooth almost surely. We assume that $f$ is normalised so that $\E[f(x)]=0$ and $\Var[f(x)]=1$ at any point $x$, and we let $K(x-y):=\E[f(x)f(y)]$ denote the covariance function of $f$. Given $\ell\in\R$, we define the \emph{excursion sets} and \emph{level sets} of $f$ (at level $\ell$) respectively as
\[
    \{f\geq\ell\}:=\big\{x\in\R^d\;\big|\;f(x)\geq\ell\big\}\qquad\text{and}\qquad\{f=\ell\}:=\big\{x\in\R^d\;\big|\;f(x)=\ell\big\}.
\]
Topological functionals of $f$ are usually defined in terms of these sets (restricted to some compact domain).

\underline{Law of large numbers:} Nazarov and Sodin \cite{ns16} established one of the seminal results in this area: a law of large numbers for the `component count'. Specifically they showed that, under some minimal non-degeneracy and covariance decay assumptions on $f$, if $\Phi(\Lambda_R,\ell)$ denotes the number of connected components of $\{f=\ell\}$ in the domain $\Lambda_R:=[-R/2,R/2]^d$ then there exists some constant $c>0$ such that
\begin{equation}\label{e:ns_lln}
    \frac{\Phi(\Lambda_R,\ell)}{R^d}\to c
\end{equation}
almost surely and in $L^1$, as $R\to\infty$. (Although the result was stated only for the level set at $\ell=0$, the proof holds verbatim for excursion sets and arbitrary levels.) This result was subsequently generalised to the number of components with a given volume \cite{bw18} or a given diffeomorphism type \cite{sw19} as well as to the (higher order) Betti numbers of level sets \cite{wig21}. Further properties of the limiting constant in \eqref{e:ns_lln} have also been established \cite{kw18,bmm20,bmm20a}.

\underline{Variance asymptotics:} Subsequent work has studied the variance of the component count. The first significant result \cite{ns20} established a polynomial lower bound on the variance (with a small exponent) under minimal assumptions on the field. Whilst the result was proven for ensembles of Gaussian fields on the sphere, most parts of the proof should extend to our Euclidean setting. It was later shown \cite{bmm24b} that for fields with integrable covariance function, there exists $\sigma^2(\ell)\geq 0$ such that
\[
    \frac{\Var[\Phi(\Lambda_R,\ell)]}{R^d}\to\sigma^2(\ell)
\]
as $R\to\infty$. For fields with slower covariance decay (i.e., long-range dependence), precise asymptotics have not been established, but lower and upper bounds are known. Specifically if $K(x)\sim\lvert x\rvert^{-\alpha}$ for some $\alpha\in(0,d)$ as $\lvert x\rvert\to\infty$ and some additional regularity assumptions hold, then \cite{bmm22,bmm24b} there exist $c_2>c_1>0$ such that
\[
    c_1R^{2d-\alpha}\leq\Var[\Phi(\Lambda_R,\ell)]\leq c_2R^{2d-\alpha}
\]
for all $R\geq 1$. The upper bound here is valid for all $\ell$, while the lower bound requires some condition which should hold for generic $\ell$.

\underline{Distributional limits:} A central limit theorem for the component count has been established \cite{bmm24} under slightly stronger assumptions: if $f$ is $C^\infty$-smooth almost surely, $\sup_{\alpha\leq 2}\lvert \partial^\alpha K(x)\rvert\leq C(1+\lvert x\rvert)^{-3d-\epsilon}$ for all $x\in\R^d$ and some $C,\epsilon>0$ and some additional regularity conditions hold, then
\[
    \frac{\Phi(\Lambda_R,\ell)-\E[\Phi(\Lambda_R,\ell)]}{R^{d/2}}\overset{d}{\to}\sigma(\ell)Z
\]
where $Z\sim\mathcal{N}(0,1)$ and $\sigma(\ell)>0$. This result was proven using a martingale argument which has since been generalised to prove distributional limits for non-local functionals with a mixed geometric and topological character \cite{mca26} as well as to Betti numbers of excursion sets \cite{hlr24}.

More recently, the analogue of the component count for a discrete Gaussian field $f:\Z^d\to\R$ with long-range dependence has been studied using the Wiener chaos expansion \cite{mm25}. This work proved variance asymptotics and distributional limit theorems, both of which depend on the underlying level of the excursion/level set due to a cancellation phenomenon. In the context of local fuctionals, this phenomenon is known as `Berry cancellation' \cite{ber02} and has been thoroughly studied \cite{wig09,kkw13,npr19,mrw20}. Striking results related to Berry cancellation have recently been established for a variety of local and non-local functionals of continuous fields \cite{gmp25}.

\underline{Concentration bounds:} One of the earliest results for a non-local functional of a Gaussian field showed that the normalised component count for random Laplace eigenfunctions on the sphere concentrates exponentially about its mean \cite{ns09}. This was subsequently adapted to fields on Euclidean space and the torus \cite{pri20}, although both works rely heavily on specific properties of Laplace eigenfunctions. For general Gaussian fields, somewhat weaker lower concentration bounds were proven \cite{rv19,bmr20}. Specifically, if $K(x)=O(\lvert x\rvert^{-\beta})$ as $\lvert x\rvert\to\infty$ then for any $\epsilon,\delta>0$
\[
    \P\big(\overline{\Phi}(\Lambda_R,\ell)<-\epsilon R^d\big)=O(R^{2d-\beta+\delta})\qquad\text{as }R\to\infty,
\]
where $\overline{X}:=X-\E[X]$. While if $K(x)=O(e^{-c\lvert x\rvert^\alpha})$ for some $c,\alpha>0$ as $\lvert x\rvert\to\infty$ then for some $c^\prime>0$,
\[
    \P\big(\overline{\Phi}(\Lambda_R,\ell)<-\epsilon R^d\big)=O\big(\exp\big(-c^\prime R^{d\alpha/(d+\alpha)}\big)\big)\qquad\text{as }R\to\infty.
\]

\underline{Euler characteristic:} Although our main interest lies in non-local functionals, we also derive results for the Euler characteristic of Gaussian field excursion sets. We recall that the Euler characteristic is an integer-valued topological invariant defined on a wide class of sets in Euclidean space. In particular for a two-dimensional set, the Euler characteristic is equal to the number of connected components minus the number of `holes' and analogous interpretations exist in higher dimensions. For further details, and a robust theoretical presentation in the context of Gaussian fields, we refer the reader to \cite{at07}.

The Poincar\'e-Hopf theorem implies that the Euler characteristic of an excursion set is in fact a local functional, rendering it amenable to study using tools such as the Kac-Rice theorem and Wiener chaos expansion. The latter was used in \cite{el16} to derive a central limit theorem for the Euler characteristic of a stationary Gaussian excursion set assuming that the covariance function is integrable. Further results, including a quantitative central limit theorem, have been established for two-dimensional fields whose realisations are eigenfunctions of the Laplacian \cite{cmw16,cm18,bst24}.

\section{Main results}
We now describe our main results. Given a stationary Gaussian field $f$ with covariance function $K$, we define the \emph{spectral measure} $\nu$ to be the symmetric measure (i.e., $\nu(A)=\nu(-A)$) satisfying
\[
    \Cov[f(x)f(y)]=K(x-y)=\int_{\R^d}e^{it\cdot(x-y)}\;d\nu(t)
\]
for all $x,y\in\R^d$. We will require the following baseline assumption throughout our work:
\begin{namedassumption}{Basic}{a:basic}
    For $d\geq 2$, let $f:\R^d\to\R$ be a stationary Gaussian field with covariance kernel $K$ such that
    \begin{enumerate}
        \item $\E[f(0)]=0$ and $\Var[f(0)]=1$,
        \item $f$ is $C^4$-smooth almost surely,
        \item $\sup_{\lvert\alpha\rvert\leq 2}\lvert\partial^\alpha K(x)\rvert\to 0$ as $\lvert x\rvert\to\infty$,
        \item the support of the spectral measure of $f$ contains an open set.
    \end{enumerate}
\end{namedassumption}

\subsection{Topological functionals: qualitative limit theorems}
Our first class of results concerns the qualitative asymptotic behaviour of topological functionals on large domains. These will require (different levels of) control over the decay of the covariance function, which we now specify:

For $x\in\R^d$ we define $\widetilde{K}(x)=\sup_{\lvert y-x\rvert\leq 1}\lvert K(y)\rvert$.

\begin{namedassumption}{CovDecay-Weak}{a:cov_decay_weak}
    For some $\epsilon>0$, as $R\to\infty$
    \[
        \int_{\Lambda_R^2}\widetilde{K}(x-y)\;dxdy=O(R^{2d}(\log R)^{-1-\epsilon}).
    \]
\end{namedassumption}
\begin{namedassumption}{CovDecay-Int}{a:cov_decay_int}
    The function $\widetilde{K}$ is integrable.
\end{namedassumption}
Note that a sufficient condition for these assumptions would be that $\lvert K(x)\rvert\leq C(1+\lvert x\rvert)^{-\beta}$ for all $x\in \R^d$, some $C>0$ and some $\beta>0$ or $\beta>d$ respectively. We also remark that Assumption~\aref{a:cov_decay_int} implies that the spectral measure is absolutely continuous with respect to Lebesgue measure, and so verifies the fourth part of Assumption~\aref{a:basic}.

We now specify the class of topological functionals to which our results will apply. We say that a function $g:\R^d\to\R$ is \emph{Morse} if it is $C^2$ smooth and all of its critical points are non-degenerate (i.e., the Hessian matrix has full rank at each critical point). Given $\ell\in\R$, we define $\mathrm{Morse}(\R^d,\ell)$ to be the set of Morse functions for which $\ell$ is not a critical value.
\begin{definition}[Bounded excursion/level-set functional]
    Let $\Phi$ be a real-valued function defined on triples $(g,D,\ell)$ where $\ell\in\R$, $g\in\mathrm{Morse}(\R^d,\ell)$ and $D\subset\R^d$ is a `nice' compact set (e.g., a box). We say that $\Phi$ is a \emph{bounded excursion-set functional} if it can be expressed as
    \begin{equation}\label{e:es_functional}
        \Phi(D,g):=\Phi(D,g,\ell)=\sum_{E\in\mathrm{Comp}(\{g\geq\ell\},D)}\varphi(E)
    \end{equation}
    where $\mathrm{Comp}(\{g\geq\ell\},D)$ denotes the set of connected components of $\{g\geq\ell\}$ which intersect $D$ but not $\partial D$ and $\varphi$ is a bounded function depending only on the topology of $E$ (i.e., it is invariant under isotopies of $E$).

    We say that $\Phi$ is a \emph{bounded level-set functional} if it can be represented as above with $\{g=\ell\}$ replacing $\{g\geq\ell\}$ and we use the phrase bounded excursion/level-set functional to denote a functional of either type.
\end{definition}

We give a more rigorous definition in Section~\ref{s:quasi-association}. For now, we simply observe that the component count and the various `diffeomorphism counts' (i.e., the number of excursion/level set components of a particular diffeomorphism type inside a domain) are examples of bounded excursion/level-set functionals. However the higher order Betti numbers of the excursion/level-set do not fall into this class, because the contribution from an individual component may be unbounded.

We will consider one other example of a topological functional in our results: the Euler characteristic of the excursion set. In this case we will take $\Phi(D,g,\ell)$ to denote the Euler characteristic of $\{g\geq\ell\}\cap D$ (i.e., we consider the whole excursion set in $D$ and not just the components in the interior), to match the convention in previous literature. For a precise definition of the Euler characteristic in this setting, as well as a proof that it is well defined for Gaussian field excursion sets, we refer the reader to \cite[Chapters 6 and 11]{at07}. 
Throughout the remainder of this section, $\Phi$ will denote either an arbitrary excursion/level-set functional or the Euler characteristic.

Our first main result is a law of large numbers for topological functionals:

\begin{theorem}\label{t:lln}
    Let $f$ satisfy Assumption~\aref{a:basic} and $\ell\in\R$ and suppose that either:
    \begin{enumerate}
        \item Assumption~\aref{a:cov_decay_weak} holds, or
        \item $\Phi$ is the Euler characteristic.
    \end{enumerate}
    In the first case, there exists $c=c(\ell)\in\R$ such that, as $R\to\infty$
    \[
        \frac{\Phi(\Lambda_R,f)}{R^d}\longrightarrow c
    \]
    almost surely and in $L^2$. In the second case, the same conclusion holds with almost sure convergence restricted to integer values of $R$.
\end{theorem}

The result for the Euler characteristic follows fairly directly from the ergodic theorem (so is included for completeness rather than novelty). The statement for excursion/level-set functionals would follow (under slightly weaker conditions) from the ergodic argument of Nazarov-Sodin in \cite{ns16}, however we use a very different method of proof. Moreover our method could likely be generalised to deal with non-stationary fields (see Remark~\ref{r:lln_non-stationary}).

For fields with stronger covariance decay, we are able to characterise the asymptotic behaviour of the variance of topological functionals and establish distributional/almost sure central limit theorems:

\begin{theorem}\label{t:var+clt}
    Let $f$ satisfy Assumptions~\aref{a:basic} and~\aref{a:cov_decay_int}, $\ell\in\R$ and $Z\sim\mathcal{N}(0,1)$.
    \begin{enumerate}
        \item (Variance asymptotics) As $R\to\infty$
        \begin{equation}\label{e:sigma}
            \frac{\Var[\Phi(\Lambda_R,f)]}{R^d}\longrightarrow\sigma^2:=\int_{\R^d}K(x)\int_0^1\mu^t_{x,0}(d_x\Phi_\infty d_0^t\Phi_\infty)\;dtdx
        \end{equation}
        where the latter integrand is defined in Section~\ref{s:quasi-association}.
        \item (Distributional CLT) As $R\to\infty$
        \[
            \widetilde{\Phi}_R:=\frac{\Phi(\Lambda_R,f)-\E[\Phi(\Lambda_R,f)]}{R^{d/2}}\overset{d}{\longrightarrow}\sigma Z.
        \]
        \item (Almost sure CLT) For any Lipschitz function $F:\R\to\R$, as $R\to\infty$
        \[
            \frac{1}{\log(R)}\int_1^Rr^{-1}F\big(\widetilde{\Phi}_r\big)\;dr\overset{a.s.}{\longrightarrow}\E[F(\sigma Z)].
        \]
    \end{enumerate}
\end{theorem}

The expression for the limiting variance $\sigma^2$ given in \eqref{e:sigma} is in general quite difficult to analyse, although could plausibly be used to derive upper/lower bounds. In the case of the Euler characteristic, the expression simplifies considerably; see Remark~\ref{r:var_ec}.

Theorem~\ref{t:var+clt} imposes significantly weaker assumptions on the field than the central limit theorems for non-local functionals in \cite{bmm24,mca26,hlr24}.

The statement of the almost-sure CLT given above implies the more conventional statement that, with probability one,
\[
    \frac{1}{\log(R)}\int_1^R r^{-1}\delta_{\widetilde{\Phi}_r}\;dr
\]
converges weakly to the standard Gaussian measure as $R\to\infty$, where $\delta_x$ denotes the Dirac measure on the point $x$. This follows from a standard argument making use of the separability of $\R$. Theorem~\ref{t:var+clt} appears to be the first example of an almost-sure central limit theorem for a non-local functional (or the Euler characteristic) of a Gaussian field. An analogous limit theorem for local functionals has been proven in \cite{mrz25}.

\subsection{Topological functionals: quantitative limit theorems}\label{ss:quant_lim}
We can attain refined, quantitative versions of the previous limit theorems by imposing stronger assumptions on the underlying field.

We first require the covariance function to be smooth and decay either polynomially or exponentially:

\begin{namedassumption}{Smooth}{a:smooth}
    The covariance function $K$ is $C^\infty$-smooth (or equivalently, $f$ is $C^\infty$-smooth almost surely).
\end{namedassumption}

\begin{namedassumption}{CovDecay-Pol($\beta$)}{a:cov_decay_pol}
    For some $\beta>d$,
        \[
            \sup_{\lvert\alpha\rvert\leq 1}\lvert\partial^\alpha K(x)\rvert=O(\lvert x\rvert^{-\beta})\qquad\text{as }\lvert x\rvert\to\infty.
        \]
\end{namedassumption}

\begin{namedassumption}{CovDecay-Exp}{a:cov_decay_exp}
        For some $c>0$,
        \[
            \sup_{\lvert\alpha\rvert\leq 1}\lvert\partial^\alpha K(x)\rvert=O(e^{-c\lvert x\rvert})\qquad\text{as }\lvert x\rvert\to\infty.
        \]
\end{namedassumption}

In order to prove quantitative results for excursion/level-set functionals, we also require probabilistic control over excursion/level-set components with large diameter. This control is closely related to the \emph{percolation phase transition} for the field $f$, which we now describe. For a variety of smooth, stationary Gaussian fields $f$, it has been shown \cite{bg17,mv20,riv21,sev22,mrv23} that there exists a \emph{critical level} $\ell_c$ such that:
\begin{enumerate}
    \item for all $\ell>\ell_c$, $\{f\geq\ell\}$ almost surely contains only bounded connected components,
    \item for all $\ell<\ell_c$, $\{f\geq\ell\}$ almost surely contains an unbounded connected component.
\end{enumerate}
In the planar case, $\ell_c=0$, while in higher dimensions it is expected that $\ell_c>0$, and this has been proven assuming sufficiently fast covariance decay \cite{drrv23}.

For fields with integrable covariance function (and satisfying some additional regularity conditions), it is known that for each subcritical level (i.e., $\ell>\ell_c$), the probability of the `one-arm event' that there exists a component of $\{f\geq\ell\}$ intersecting both $\Lambda_1$ and $\partial\Lambda_R$ decays exponentially in $R$. At supercritical levels ($\ell<\ell_c$), it is expected that the same is true if one considers only \emph{bounded} components (known as a `truncated one-arm event') although this is unproven. The statement follows for $\ell<-\ell_c$ (the so-called `strongly supercritical' regime) using a simple argument based on subcritical decay and symmetry of the normal distribution. At the critical level in two dimensions $\ell=\ell_c=0$, the probability of the one-arm event is known to decay (at least) polynomially and this is expected to be the correct type of decay (albeit with unknown exponent). Critical one-arm decay in higher dimensions is a major open question.

We impose different assumptions to reflect the subcritical, supercritical and critical (in dimension two) cases respectively:

Given three sets $A,B,C\subset\R^d$, we write $A\overset{B}{\longleftrightarrow}C$ to denote the event that there exists a continuous path in $B$ which starts in $A$ and ends in $C$. We let $B(x,r)$ denote the Euclidean ball of radius $r$ centred at $x$ and $B(r):=B(0,r)$.

\begin{namedassumption}{ArmDecay-NonCrit}{a:arm_decay}
    The field $f$ and level $\ell\in\R$ satisfy one of the following:
    \begin{enumerate}
        \item (Subcritical) For some $\epsilon,c>0$, as $R\to\infty$
        \[
            \P\Big(\Lambda_1\overset{\{f\geq\ell-\epsilon\}}{\longleftrightarrow}\partial\Lambda_R\Big)=O(e^{-cR}).
        \]
        \item (Strongly supercritical) For some $\epsilon,c>0$, as $R\to\infty$
            \[
                \P\Big(\Lambda_1\overset{\{f\leq\ell+\epsilon\}}{\longleftrightarrow}\partial\Lambda_R\Big)=O(e^{-cR}).
            \]
    \end{enumerate}
\end{namedassumption}

\begin{namedassumption}{ArmDecay-Crit}{a:arm_decay_critical}
    The field $f$ and level $\ell$ satisfy the following:
    \begin{enumerate}
        \item the density of the spectral measure of $f$ is strictly positive at the origin, and
        \item there exists $C,\epsilon>0$ such that for all $r,R\geq 1$
    \[
        \P\Big(\partial B(r)\overset{\{f\geq\ell\}}{\longleftrightarrow}\partial B(R)\Big)\leq C\Big(\frac{r}{R}\Big)^\epsilon.
    \]
    \end{enumerate}
\end{namedassumption}

From our preceding description, it is apparent that Assumption~\aref{a:arm_decay} is valid (for many fields with integrable covariance function) whenever $\lvert\ell\rvert>\ell_c$ which covers the subcritical and strongly supercritical regimes. In the planar case, $\ell_c=0$ and so this will comprise all non-critical levels, however in higher dimensions it will not. In Remark~\ref{r:trunc_arm}, we describe a weaker (albeit more technical) assumption which could replace Assumption~\aref{a:arm_decay} in our arguments and conjecturally holds at all non-critical levels $\ell\neq\ell_c$. Assumption~\aref{a:arm_decay_critical} holds at criticality for (many) planar fields.

With these assumptions in hand, we may proceed to our results. Our first theorem controls the higher central moments of topological functionals:

\begin{theorem}[Centred moment bounds]\label{t:higher_moments}
    Let $f$ satisfy Assumptions~\aref{a:basic}, \aref{a:cov_decay_pol} and \aref{a:smooth}. Let $\ell\in\R$ and $n\in\N$ such that $n\geq2$.
    \begin{enumerate}
        \item If $\Phi$ is the Euler characteristic and $\beta>(n+1)d$, then as $R\to\infty$
        \[
            \E[\overline{\Phi}(\Lambda_R)^{2n}]=O\Big(R^{nd}\Big).
        \]
        \item If $d=2$ and $\beta>2(n+1)$, then for any $\epsilon>0$ as $R\to\infty$
        \[
            \E[\overline{\Phi}(\Lambda_R)^{2n}]=O\Big(R^{\frac{14}{5}n+\epsilon}\Big).
        \]
        \item If Assumption~\aref{a:arm_decay} holds and $\beta>2d(d^2-1)(n-1)$, then for any $\epsilon>0$ as $R\to\infty$
        \[
            \E[\overline{\Phi}(\Lambda_R)^{2n}]=O\Big(R^{nd+\epsilon}\Big).
        \]
        \item If Assumptions~\aref{a:arm_decay} and~\aref{a:cov_decay_exp} hold, then for any $m\in\N$ as $R\to\infty$
        \[
            \E[\overline{\Phi}(\Lambda_R)^{2n}]=O\Big(R^{nd}\log^{(m)}(R)\Big),
        \]
        where $\log^{(m)}(x)=\log\log\dots\log(x)$ denotes the $m$-fold composition of the logarithm.
    \end{enumerate}
\end{theorem}

We note that the bound in the first case is optimal in the sense of matching the order that would be obtained by taking a $d$-dimensional sum of independent, identically distributed random variables.

These bounds immediately yield concentration estimates, quantifying convergence in the law of large numbers, using Markov's inequality. We state only the strongest form of such estimates:

\begin{corollary}[Superpolynomial concentration]\label{c:concentration}
    In the setting of Theorem~\ref{t:higher_moments}, suppose that Assumption~\aref{a:cov_decay_pol} holds for all $\beta>d$, then for all $\epsilon>0$ and $N\in\N$
    \[
        \P\big(\lvert\overline{\Phi}(\Lambda_R)\rvert>\epsilon R^d\big)=O(R^{-N})\qquad\text{as }R\to\infty.
    \]
\end{corollary}

Assuming super-polynomial covariance decay, we are able to explicitly characterise the rate of convergence in the law of large numbers for the Euler characteristic:

\begin{theorem}[Law of the iterated logarithm]\label{t:ec_lil}
    Let $f$ satisfy Assumptions~\aref{a:basic}, \aref{a:smooth} and \aref{a:cov_decay_pol} for all $\beta>d$. Let $\Phi$ denote the Euler characteristic and for $\ell\in\R$ suppose that $\sigma$, as defined in Theorem~\ref{t:var+clt}, is positive. Then with probability one,
    \[
        \limsup_{\substack{n\to\infty\\n\in\N}}\frac{\overline{\Phi}(\Lambda_n,f)}{\sqrt{2n^d\log\log(n)}}=\sigma\qquad\text{and}\qquad\liminf_{\substack{n\to\infty\\n\in\N}}\frac{\overline{\Phi}(\Lambda_n,f)}{\sqrt{2n^d\log\log(n)}}=-\sigma.
    \]
\end{theorem}

We note that an explicit expression for $\E[\Phi(\Lambda_n,f)]$ is obtained in \cite{at07} and that a sufficient condition for positivity of $\sigma$ is given in \cite[Proposition~2.1]{el16}.

Our final result for topological functionals is a quantitative version of the central limit theorem. This is defined in terms of Kolmogorov distance, which we now recall: given two random variables $X$ and $Y$, let
\[
    d_\mathrm{Kol}(X,Y):=\sup_{x\in\R}\lvert \P(X\leq x)-\P(Y\leq x)\rvert.
\]

\begin{theorem}[Quantitative CLT]\label{t:qclt}
    Let $f$ satisfy Assumptions~\aref{a:basic} and \aref{a:cov_decay_pol} and let $\ell\in\R$. Suppose that $\sigma$ defined in Theorem~\ref{t:var+clt} is positive and that one of the following conditions holds:
    \begin{enumerate}
        \item $\Phi$ is the Euler characteristic,
        \item Assumption~\aref{a:arm_decay} holds at level $\ell$, or
        \item Assumption~\aref{a:arm_decay_critical} holds at level $\ell$,
    \end{enumerate}
    then there exists $\eta>0$ such that as $R\to\infty$,
    \begin{equation}\label{e:qclt}
        \Delta_R:=d_\mathrm{Kol}\big(\widetilde{\Phi}_R,\sigma Z\big)=O(R^{-\eta}),
    \end{equation}
    where $Z\sim\mathcal{N}(0,1)$. 

    Moreover, if Assumption~\aref{a:smooth} also holds and $\beta>2d(d^2-1)$, then we may set
    \[
        \eta=\begin{cases}
            \frac{d\beta-d^2}{2(2d+1)\beta-d(4d-1)} &\text{in case (1),}\\
            \frac{d}{2}\frac{\beta}{(2d+1)\beta+4d(d+1)}-\epsilon &\text{for any $\epsilon>0$ in case (2),}
        \end{cases}
    \]
    and in either case (1) or (2), if Assumption~\aref{a:cov_decay_exp} holds then
    \[
        \Delta_R=O\Big(R^{-\frac{d}{2(2d+1)}}\log(R)\Big).
    \]
\end{theorem}

Sufficient conditions for positivity of $\sigma(\ell)$ when $\Phi$ is an excursion/level-set functional can be found in \cite{bmm22,bmm24,hlr24}.

The exponent $\eta$ could be compared with the exponent $d/2$ that holds in the independent, identically distributed and additive case (i.e., for the classical Berry-Esseen theorem when summing over $d$-dimensions). It is not clear whether one would expect the rate $R^{-d/2}$ to hold generally (i.e., for all $\beta>d$) in our setting.

\subsection{Volume of the unbounded component}
We will consider one further example of a non-local functional, which is not purely topological. Assuming that the field $f$ undergoes a percolation phase transition (as described above), for every $\ell<\ell_c$ there exists an unbounded component of $\{f\geq\ell\}$ almost surely and we denote the union of all such bounded/unbounded components by $\{f\geq\ell\}_{<\infty}$ and $\{f\geq\ell\}_\infty$ respectively. (Under some regularity conditions on $f$, the unbounded component of the excursion set is known to be unique \cite{sev24}, but this fact will not be needed for our arguments.) Our functional of interest is the volume of the unbounded component (restricted to a compact domain) denoted by
\[
    \Vol_\infty(D,\ell):=\lvert D\cap\{f\geq\ell\}_\infty|=\int_D\ind_{x\in\{f\geq\ell\}_\infty}\;dx
\]
for compact $D\subset\R^d$, where $\lvert\cdot\rvert$ denotes $d$-dimensional Lebesgue measure. Observe that this functional is non-local since $\ind_{x\in\{f\geq\ell\}_\infty}$ is not a pointwise function of $f$ and its derivatives. This functional, together with two other geometric functionals of the unbounded component, were analysed in \cite{mca26}, wherein a law of large numbers and central limit theorem were proven under fairly strong assumptions on the field (including Assumption~\aref{a:cov_decay_pol} for $\beta>3d$). Here, we will need only the following, weaker conditions: 

\begin{namedassumption}{CovPos}{a:cov_pos}
    The covariance function $K$ is non-negative.
\end{namedassumption}

\begin{namedassumption}{ArmDecay-Trunc}{a:trunc_arm_decay}
    For any $\gamma>0$, as $R\to\infty$
        \[
            \P\Big(0\overset{\{f\geq\ell\}_{<\infty}}{\longleftrightarrow}\partial \Lambda_R\Big)=O(R^{-\gamma}).
        \]
\end{namedassumption}

Assumption~\aref{a:trunc_arm_decay} can easily be deduced using subcritical arm decay and symmetry of the normal distribution when $\ell<-\ell_c$ (see \cite[Proposition~3.4]{mca26}). So in particular, when $d=2$, this holds for all $\ell<\ell_c=0$ (under the conditions necessary for subcritical arm decay, see \cite{riv21}). Conjecturally, Assumption~\aref{a:trunc_arm_decay} should hold for all $\ell<\ell_c$.

Under these assumptions, we can obtain analogous versions of some of our previous results:

\begin{theorem}\label{t:vol}
    Let $f$ and $\ell<\ell_c$ satisfy Assumptions~\aref{a:basic}, \aref{a:cov_pos}, \aref{a:cov_decay_pol} and \aref{a:trunc_arm_decay}. Given $x\in\R^d$, let $E_x$ denote the event that $x\in\{f\geq\ell\}_\infty$ and define
    \begin{equation}\label{e:vol_sigma}
        \theta(\ell):=\P(E_0)\qquad\text{and}\qquad\sigma^2:=\int_{\R^d}\Cov[E_x,E_0]\;dx,
    \end{equation}
    then $0<\sigma^2<\infty$ and the following holds:
    \begin{enumerate}
        \item (Variance asymptotics and CLT) As $R\to\infty$
        \[
            \frac{\Var[\Vol_\infty(\Lambda_R,\ell)]}{R^d}\longrightarrow\sigma^2\qquad\text{and}\qquad\frac{\Vol_\infty(\Lambda_R,\ell)-\theta(\ell)R^d}{R^{d/2}}\overset{d}{\longrightarrow}\sigma Z
        \]
        where $Z\sim\mathcal{N}(0,1)$.
        \item (Quantitative CLT) For all $\epsilon>0$, as $n\to\infty$ (with $n\in\N$)
        \[
            \Delta_n:=d_\mathrm{Kol}\Big(\frac{\Vol_\infty(\Lambda_n,\ell)-\theta(\ell)n^d}{\sqrt{\Var[\Vol_\infty(\Lambda_n,\ell)]}},Z\Big)=O\Big(n^{-\frac{d}{2}+\frac{d^2}{\beta+d}+\epsilon}\Big).
        \]
        Moreover, if Assumption~\aref{a:cov_decay_exp} holds, then $\Delta_n=O(n^{-d/2}\log(n)^{2d})$ as $n\to\infty$.
    \end{enumerate}
\end{theorem}
We note that for fields with sufficiently fast correlation decay, the rate of convergence in the CLT approaches the conjecturally optimal rate of $n^{-d/2}$, up to a small polynomial/logarithmic factor.

\subsection{Method of proof}\label{ss:proof_outline} We now sketch our method of proof:

\underline{Quasi-association:} Our approach builds on ideas which originated in the study of associated random fields. We recall that a set of random variables $\mathcal{X}$ is said to be associated if, for any $n\in\N$, $X_1,\dots,X_n\in\mathcal{X}$ and any bounded, measurable, coordinate-wise non-decreasing functions $f,g:\R^n\to\R$
\[
    \Cov[f(X_1,\dots,X_n),g(X_1,\dots,X_n)]\geq 0.
\]
This property is also known as the FKG inequality, and is satisfied by a variety of statistical physics models. Newman \cite{new80} proved that for a stationary, associated (discrete) random field $(X_k)_{k\in\Z^d}$, the block sums $\sum_{k\in\Z^d\cap\Lambda_n}X_k$ satisfy a central limit theorem, as $n\to\infty$, provided that the following \emph{finite susceptibility} condition holds:
\[
    \sum_{k\in\Z^d}\Cov[X_k,X_0]<\infty.
\]
Newman's proof made use of the following inequality: for associated random variables $\{X_1,\dots,X_n\}$ and Lipschitz functions $F,G:\R^n\to\R$
\begin{equation}\label{e:quasi-association}
    \lvert\Cov[F(X_1,\dots,X_n),G(X_1,\dots,X_n)]\rvert\leq C\sum_{i,j=1}^n\|F\|_{\mathrm{Lip},i}\|G\|_{\mathrm{Lip},j}\Cov[X_i,X_j]
\end{equation}
where $C>0$ is an absolute constant and $\|F\|_{\mathrm{Lip},i}$ denotes the Lipschitz constant of $F$ with respect to its $i$-th argument. Subsequent work showed that \eqref{e:quasi-association} was sufficient to establish a variety of limit theorems (see \cite{bs07} for a textbook presentation), as well as being satisfied by systems exhibiting other forms of weak dependence, and so it became known as \emph{quasi-association}.

Our goal is to replicate the arguments for quasi-associated block sums in the context of non-local functionals of Gaussian fields. To do so, we must establish three properties of such functionals: (i) quasi-association, (ii) approximate additivity and (iii) geometric stabilisation (explained below).

\underline{Topological covariance formula:} In \cite{bmr20}, a general covariance formula for topological events related to smooth Gaussian fields was proven under minimal smoothness and regularity conditions on the field. The result can be roughly described as follows: given a $C^2$ field $f:\R^d\to\R$ and $t\in[0,1]$ we let
\[
    f^t=tf+\sqrt{1-t^2}\widetilde{f}    
\]
where $\widetilde{f}$ is an independent copy of $f$. For $\ell\in\R$ fixed and $x,y\in\R^d$, the measure $\mu^t_{x,y}$ is defined as
\[
    \mu^t_{x,y}(H)=\E\Big[\lvert\det\nabla^2 f(x)\rvert\lvert\det\nabla^2 f^t(y)\rvert H(f,f^t)\;\Big|\;A^t_{x,y}\Big]\phi^t_{x,y}
\]
where $\det$ denotes the determinant, $H$ is an appropriate real-valued Borel function,
\begin{equation}\label{e:conditioning_outline}
    A^t_{x,y}:=\big\{f(x)=f^t(y)=\ell,\nabla f(x)=\nabla f^t(y)=0\big\}
\end{equation}
and $\phi^t_{x,y}$ denotes the (Gaussian) density of $(f(x),f^t(y),\nabla f(x),\nabla f^t(y))$ evaluated at $(\ell,\ell,0,0)$. The covariance formula of \cite{bmr20} then says that for events $A_1$ and $A_2$, depending only on the topology of $\{f\geq\ell\}$ restricted to the compact sets $D_1,D_2\subset\R^d$ respectively,
\begin{equation}\label{e:cov_formula_outline}
    \Cov[A_1,A_2]=\int_{\mathcal{F}_1\times\mathcal{F}_2}K(x-y)\int_0^1\mu^t_{x,y}(d_x\ind_{A_1}d_y^t\ind_{A_2})\;dtdxdy
\end{equation}
where $\mathcal{F}_i$ denotes an appropriate `stratification' of $D_i$, which accounts for boundary effects, $d_x\ind_A$ denotes the change in $\ind_A$ when a small positive perturbation is added to $f$ at the critical point $x$ and $d_y^t\ind_A$ denotes the analogous change for the interpolated field $f^t$.  (We define these notions rigorously in Section~\ref{s:quasi-association}.) 

\underline{Topological quasi-association:} The first step in our work is to extend the covariance formula \eqref{e:cov_formula_outline} to apply to Lipschitz functions of topological functionals. Specifically, for Lipschitz $F,G:\R\to\R$, certain topological functionals $\Phi$ and suitable compact sets $D_1$, $D_2$, we show that
\begin{equation}\label{e:cov_formula_lip_outline}
    \Cov\big[F(\Phi(D_1)),G(\Phi(D_2))\big]=\int_{\mathcal{F}_1\times\mathcal{F}_2}K(x-y)\int_0^1\mu^t_{x,y}\big(d_xF(\Phi(D_1))d_y^tG(\Phi(D_2))\big)\;dtdxdy,
\end{equation}
where we have omitted the arguments $f$ and $f^t$ from the functional $\Phi$ for conciseness. The proof is based on approximating $F$ and $G$ by simple functions and applying a dominated convergence argument. A very similar argument was used in \cite{bmm24b} to deduce a covariance formula for the component count.

As a corollary of the covariance formula \eqref{e:cov_formula_lip_outline}, we establish a type of quasi-association (up to boundary effects) in the continuum. Specifically, using the same notation as above, we have
\begin{equation}\label{e:quasi-association_Gaussian}
    \big\lvert\Cov\big[F(\Phi(D_1)),G(\Phi(D_2))\big]\big\rvert\leq C\|F\|_\mathrm{Lip}\|G\|_\mathrm{Lip}\int_{D_1\times D_2}\widetilde{K}(x-y)\;dxdy
\end{equation}
where $C>0$ depends only on $\Phi$ and the distribution of $f$ and we recall that $\widetilde{K}(x):=\sup_{\lvert y-x\rvert\leq 1}\lvert K(y)\rvert$. The use of $\widetilde{K}$, rather than $K$, accounts for boundary effects.

\underline{Additivity:} We show that bounded excursion/level-set functionals are approximately additive on large domains, by using a three-scale argument. We first divide the macroscopic box $\Lambda_R$ into mesoscopic boxes of size $r$ which are separated by distance $\approx a$ where $1\ll a\ll r\ll R$ (see Figure~\ref{fig:box_stratification}). We let $\chi_R$ denote the set of points at the centres of the mesoscopic boxes and $U_R=\overline{\Lambda_R\setminus\cup_{x\in\chi_R}(x+\Lambda_r)}$ the closed complement of all mesoscopic boxes in $\Lambda_R$.

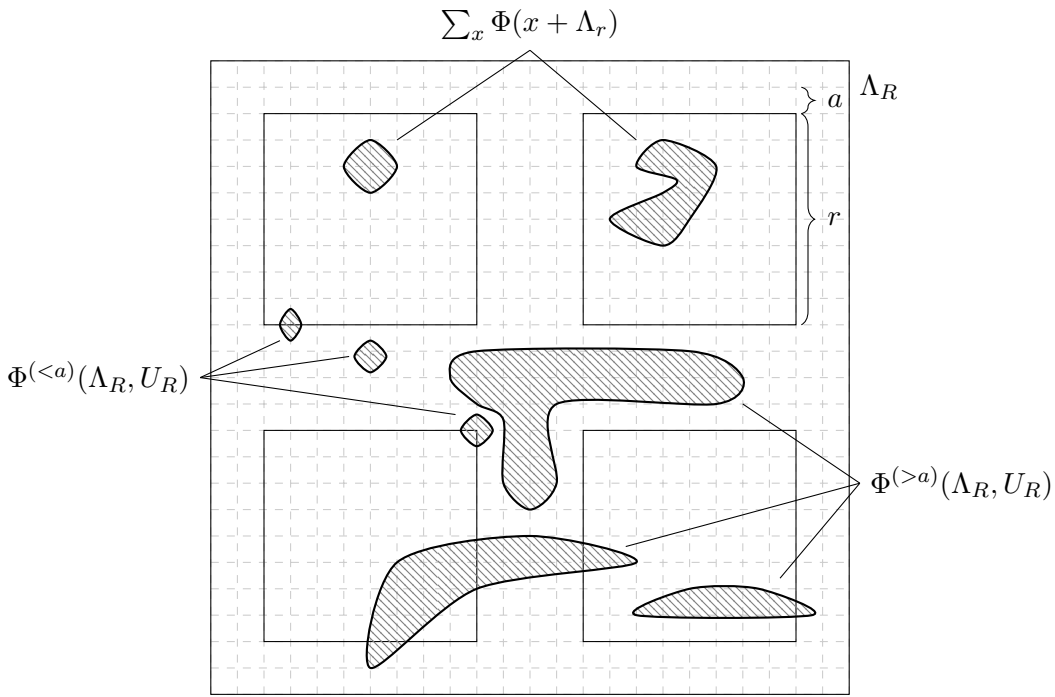
\begin{figure}[h]
    \centering
    \begin{tikzpicture}[brace/.style={decorate,decoration={brace, amplitude=5pt}}, scale=0.7]

    \draw[black!20,dashed,step=0.5] (0,0) grid (12,12);
    \draw (0,0) rectangle (12,12);
    \foreach \x in {1,7}
        \foreach \y in {1,7} {
        \draw (\x,\y) rectangle (\x+4,\y+4);
    }
    \draw[brace] (11.1,11) --(11.1,7) node[midway, right=6pt] {$r$};
    \draw[brace] (11.1,11.5) --(11.1,11) node[midway, right=6pt] {$a$};
    \node[right] at (12,11.5) {$\Lambda_R$};

    \draw [thick, pattern = north west lines, pattern color = gray] plot [smooth cycle, tension = 0.5] coordinates {(3,10.5) (2.5,10) (3,9.5) (3.5,10)};
    \draw [thick, pattern = north west lines, pattern color = gray] plot [smooth cycle, tension = 0.5] coordinates {(8,10) (8.5,10.5)(9.5,10) (9,9) (8.5,8.5) (7.5,9) (8.5,9.5) (8.75,9.75)};
    \node[above] at (6,12.2) {$\sum_x\Phi(x+\Lambda_r)$};
    \draw (6,12.2) -- (3.5, 10.5);
    \draw (6,12.2) -- (8,10.5);
    
    \draw [thick, pattern = north west lines, pattern color = gray] plot [smooth cycle, tension = 0.5] coordinates {(1.5,7.3) (1.3,7) (1.5,6.7)(1.7,7)};
    \draw [thick, pattern = north west lines, pattern color = gray] plot [smooth cycle, tension = 0.5] coordinates {(5,5.3)(4.7,5)(5,4.7)(5.3,5)};
    \draw [thick, pattern = north west lines, pattern color = gray] plot [smooth cycle, tension = 0.5] coordinates {(3,6.7)(2.7,6.4)(3,6.1)(3.3,6.4)};
    \node[left] at (-0.2,6) {$\Phi^{(<a)}(\Lambda_R,U_R)$};
    \draw (-0.2,6) -- (1.3,6.7);
    \draw (-0.2,6) -- (4.6,5.3);
    \draw (-0.2,6) -- (2.6,6.4);
    
    \draw [thick, pattern = north west lines, pattern color = gray] plot [smooth cycle, tension = 0.5] coordinates {(11.3,1.5) (10.3, 2) (9, 2) (8, 1.5)};
    \draw [thick, pattern = north west lines, pattern color = gray] plot [smooth cycle, tension = 0.5] coordinates {(3,0.5) (3.5,2.5) (6,3) (8,2.5) (5,2)};
    \draw [thick, pattern = north west lines, pattern color = gray] plot [smooth cycle, tension = 0.5] coordinates {(10,6) (9,6.5) (5,6.5) (4.5, 6) (5,5.5) (5.5,5.2) (5.5,4) (6, 3.5) (6.5, 4) (6.5,5.5) (9.5,5.5)};
    \node[right] at (12.2,4) {$\Phi^{(>a)}(\Lambda_R,U_R)$};
    \draw (12.2,4) -- (7.8,2.8);
    \draw (12.2,4) -- (10.7,2.2); 
    \draw (12.2,4) -- (10,5.5);

\end{tikzpicture}

    \caption{The box $\Lambda_R$ subdivided into separated mesoscopic $r$-boxes and stratified by hyperplanes at distance $a$. The shaded regions represent components of the excursion set $\{f\geq\ell\}$, labelled by the term of \eqref{e:top_decomp_outline} to which they contribute.}
    \label{fig:box_stratification}
\end{figure}

We then decompose
\begin{equation}\label{e:top_decomp_outline}
    \Phi(\Lambda_R)=\sum_{x\in\chi_R}\Phi(x+\Lambda_r)+\Phi^{(>a)}(\Lambda_R,U_R)+\Phi^{(<a)}(\Lambda_R,U_R)
\end{equation}
where $\Phi^{(>a)}(\Lambda_R,U_R)$ and $\Phi^{(<a)}(\Lambda_R,U_R)$ respectively denote the contributions from `large' and `small' components in $\Lambda_R$ which intersect $U_R$. The terms `large' and `small' are defined topologically, in terms of intersecting two strata at distance $a$. We then use quasi-association (i.e., \eqref{e:quasi-association_Gaussian} with $F$ and $G$ as the identity) to show that each of the last two terms has variance $o(R^d)$ as $R\to\infty$. For $\Phi^{(<a)}$, this holds because the contributing components are in a domain of volume $o(R^d)$ (and $\widetilde{K}$ is assumed integrable). For $\Phi^{(>a)}$, it follows from the fact that contributing components must be bounded but have diameter at least $a$ and this probability decays to zero (possibly very slowly) as $a\to\infty$.

Neglecting these terms of lower order variance, this expresses $\Phi_R$ as a sum of identically distributed terms. Correlations between these terms can be controlled by quasi-association, and decay since the mesoscopic boxes are separated on scale $a\gg 1$. Altogether this yields the desired additivity statement (for bounded excursion/level-set functionals).

\underline{Stabilisation:} In our setting, stabilisation intuitively refers to the property that for any $x\in\R^d$, the contribution to $\Phi(\Lambda_R)$ from the field on a neighbourhood of $x$ is roughly constant for all $R$ sufficiently large. The precise statement we need is that for a Morse function $g$ with a critical point at $x\in\R^d$, the topological derivative $d_x\Phi(\Lambda_R,g)$ converges to some constant $d_x\Phi_\infty(g)$ as $R\to\infty$. Applying this property to the covariance formula \eqref{e:cov_formula_lip_outline} with $D_1=D_2=\Lambda_R$ and $F=G$ as the identity, shows that the integrand is (approximately) independent of $R$, from which follows the convergence of $R^{-d}\Var[\Phi(\Lambda_R)]$ as $R\to\infty$. (Note that in the setting of a block sum of a discrete, quasi-associated field, this independence between the domain and the contribution of a single term to the variance is trivial.)

For a bounded excursion/level-set functional, stabilisation follows from a simple topological argument. Specifically, for a given $g$, a positive perturbation at the critical point $x$ can affect the topology by creating/destroying/merging excursion/level-set components. However for an excursion/level-set functional, the only changes which matter are those affecting \emph{bounded} components. For $R$ sufficiently large, all such changes must be captured within $\Lambda_R$, so that $d_x\Phi(\Lambda_R,g)$ is constant. This also reveals that the rate of stabilisation is closely related to the rate of truncated arm decay.

\underline{Qualitative vs quantitative limit theorems:} Combining the three properties just established allows us to adapt many classical arguments for quasi-associated block sums and so establish the previously stated qualitative limit theorems (for excursion/level-set functionals). For instance, we prove the law of large numbers using Etemadi's method and the distributional central limit theorem using Newman's argument \cite{new80} based on the characteristic function. Due to the variety of techniques involved for the various limit theorems, we do not attempt to outline them here.

The quantitative limit theorems for excursion/level-set functionals require somewhat more refined additivity and stabilisation estimates. Such estimates can be achieved by controlling the measure of certain truncated arm events under $\mu^t_{x,y}$. This, in turn, reduces to controlling arm events for $(f\mid A^t_{x,y})$, the field $f$ conditioned to have critical points at $x$ and $y$ at level $\ell$. To do so, we couple the conditioned and unconditioned fields (which are close, away from the critical points $x$ and $y$) and invoke the arm decay assumptions we have imposed on $f$.

\underline{Euler characteristic and volume functional:} In studying our remaining two functionals, we make use of their more bespoke properties.

The Euler characteristic satisfies an exact additivity property and by the Poincar\'e-Hopf theorem, it stabilises immediately. We are therefore able to prove stronger versions of our limit theorems for excursion/level-set functionals, without recourse to arm decay assumptions.

The volume of the unbounded component is not a topological functional, so the previous arguments can not be used directly. Instead we observe that the discrete random field $(V_w)_{w\in\Z^d}$, defined by $V_w=\Vol_\infty(w/2+[0,1/2]^d)$, is associated. This follows from the fact that $\Vol_\infty[\cdot]$ is a pointwise non-decreasing function of $f$ and a theorem of Pitt for positively correlated Gaussian variables \cite{pit82}. We can then apply standard limit theorems for discrete associated fields to $\sum_{w\in\Z^d\cap[-n,n)^d}V_w=\Vol_\infty(\Lambda_{n})$ provided that $(V_w)_{w\in\Z^d}$ satisfies the finite susceptibility condition. To control the correlation of the field, we use the covariance formula \eqref{e:cov_formula_outline} along with Assumption~\aref{a:trunc_arm_decay}.

\subsection{Discussion}\label{ss:discussion}
We briefly outline some possible generalisations of our results and directions for future study.

\underline{Optimal rate of convergence:} We recall that the optimal rate of convergence in the CLT for i.i.d.\ sums (in $d$-dimensions) is $n^{-d/2}$, courtesy of the Berry-Esseen theorem. The same rate has been verified for local functionals of smooth Gaussian fields with integrable covariance function \cite{npp11} (at least in the one-dimensional case) using the Malliavin-Stein method. While our result for the volume functional approaches this optimal rate for sufficiently fast decay of correlations, the best rate we obtain for topological functionals is $n^{-\frac{d}{2}\frac{1}{2d+1}}\log(n)$. However it is unclear whether one should expect convergence at the optimal rate in this case, since the functionals are not additive, and so boundary effects may come into play. Moreover, at the critical level it is conjectured that large bounded clusters occur more frequently, which may conceivably exacerbate boundary effects, leading to a slower rate of convergence than for other levels. If so, this would be a novel phenomenon when compared to the case of local functionals.

\begin{question}
    Can the rate of convergence in the CLT for topological functionals be improved to $n^{-d/2}$? For all fields with integrable covariance function, or under stronger decay assumptions?
    Is the rate of convergence slower at the critical level?
\end{question}

\underline{Further extensions:} The use of quasi-association for studying non-local functionals of Gaussian fields offers many possibilities for further generalisation. For instance, the covariance formula in \cite{bmr20} was actually proven on general manifolds, which may offer a route to establishing limit theorems for ensembles of Gaussian fields on manifolds with suitable stationary local limits (see \cite{ns16} for a law of large numbers in this setting). A quasi-association approach could also be applied to functionals of non-stationary fields: much of the theory in the discrete setting holds under assumptions of uniform non-degeneracy and correlation decay (e.g., \cite{cg84}). Moreover, quasi-association methods provide a roadmap to proving further limit theorems such as strong/weak invariance principles.

\underline{Unified approach to approximately additive functionals:} Whilst our primary motivation in this work is the study of non-local functionals, it would be valuable to develop methods that can be used to give a unified treatment of both local and (approximately additive) non-local functionals. One might ask whether it is possible to do so using a quasi-association approach. The proof of the covariance formula in \cite{bmr20} relies on a fundamental interpolation formula for Gaussian vectors, along with sophisticated arguments to relate this to changes in the topology of excursion sets. If this analysis could be generalised suitably, it may be possible to establish quasi-association for a much wider class of functionals, leading to such a desired unified treatment. This could also allow the study of non-local functionals which are not purely topological: for instance, the measure of the unbounded component of the level set $\{f=\ell\}$. Such a component is known to exists at certain levels \cite{drrv23} and the resulting functional is non-local, but does not have the monotonicity property that allows us to study the volume of the unbounded excursion component.

\subsection{Acknowledgements}
I would like to thank Stephen Muirhead for helpful comments on an earlier draft of this work, as well as suggestions for proving the quantitative CLT under Assumption~\aref{a:arm_decay_critical}.

\section{Covariance formula and quasi-association}\label{s:quasi-association}
In this section we describe the covariance formula from \cite{bmr20} which forms the foundation of our work, extend it to Lipschitz functions of topological functionals and thereby establish quasi-association of said functionals. Along the way, we formalise the topological concepts and preliminary results required for our later work.

\subsection{Covariance formula for events}\hfill

\underline{Stratified sets:} Throughout our analysis, we fix an orthonormal basis $e_1,\dots,e_d$ of $\R^d$ which we assume to be parallel to our coordinate axes. We define a \emph{regular affine subspace} to be a set formed by translating the span of some subset of these basis vectors, in other words a set of the form
\[
    \{x+a_1e_{i_1}+\dots +a_ne_{i_n}\;|\;a_1,\dots,a_n\in\R\}    
\]
for $x\in\R^d$, $n\in\{0,1,\dots,d\}$ and $i_1,\dots,i_n\in\{1,\dots,d\}$. We define a \emph{regular affine stratified set} to be a compact set $D\subset\R^d$ together with a finite partition $D=\sqcup_{F\in\mathcal{F}}F$ of open, connected subsets $F$ of regular affine subspaces of $\R^d$ such that, for any $F_1,F_2\in\mathcal{F}$ if $F_1\cap\overline{F}_2\neq\emptyset$ then $F_1\subseteq \overline{F}_2$. We refer to the elements of $\mathcal{F}$ as \emph{strata}. A natural example of a regular affine stratified set is the cube $\Lambda_R:=[-R/2,R/2]^d$ for $R>0$ together with its (open) faces of dimension $0,1,\dots,d$. We will call this the \emph{minimal} stratification of $\Lambda_R$, denoted $\mathcal{F}_R$, and often use it without explicit mention.

We equip each stratum $F\in\mathcal{F}$ with $v_F$, the Lebesgue measure of the minimal affine subspace containing $F$. If $F$ is a zero-dimensional stratum, then $v_F$ denotes the Dirac measure at the point in $F$. We use the following notation for integrating over all strata of an affine stratified set:
\[
    \int_\mathcal{F} h(x)\;dx:=\sum_{F\in\mathcal{F}}\int_F h(x)\;dv_F(x)
\]
where $h:D\to\R$ is integrable on each stratum.

For a given regular affine stratified set $(D,\mathcal{F})$, a \emph{stratified isotopy} is a continuous map $H:D\times[0,1]\to D$ such that for each $t\in[0,1]$, $H(\cdot,t):D\to D$ is a homeomorphism and for each $F\in\mathcal{F}$, $H(F,t)=F$. For $E\subseteq D$, we define the \emph{stratified isotopy class} of $E$, denoted $[E]_{\mathcal{F}}$, to be the set of images $H(E,1)$ where $H$ is a stratified isotopy of $(D,\mathcal{F})$.

\underline{Regular and Morse functions:} For a $C^2$-smooth function $g:D\to\R$, we say that $x\in F\in\mathcal{F}$ is a (non-degenerate) \emph{stratified critical point} if it is a (non-degenerate) critical point of $g|_F$. By convention, the singleton in a zero-dimensional stratum is considered to be a non-degenerate critical point. We say that $g:D\to\R$ is \emph{regular} (on $D$ at level $\ell$) if it is $C^2$-smooth on each stratum of $D$ and $\ell$ is not a critical value of $g$. We let $\mathrm{Reg}(D,\mathcal{F},\ell)$ denote the set of such functions. Finally, we let $\mathcal{E}_{\mathcal{F}}$ denote the set of stratified isotopy classes of regular excursion sets on $D$, that is
\[
    \mathcal{E}_{\mathcal{F}}=\big\{[\{g\geq 0\}]_\mathcal{F}\;\big|\;g\in\mathrm{Reg}(D,\mathcal{F},0)\big\}.
\]
As a consequence of Bulinskaya's lemma (see \cite[Lemma~11.2.10]{at07}), for any regular affine stratified set $D$ and level $\ell$, a Gaussian field $f$ satisfying Assumption~\aref{a:basic} will be regular on $D$ at level $\ell$ with probability one. Therefore $[\{f\geq\ell\}]_\mathcal{F}\in\mathcal{E}_\mathcal{F}$ almost surely.

We say that $g:D\to\R$ is \emph{Morse} if:
\begin{enumerate}
    \item $g$ is $C^2$-smooth on each stratum of $D$,
    \item all stratified critical points of $g$ are non-degenerate, and
    \item if $x$ is a stratified critical point of a stratum $F$ then $\nabla|_{F^\prime}g(x)\neq 0$ for any affine subspace $F^\prime$ that contains $F$ and has higher dimension than $F$.
\end{enumerate}
We let $\mathrm{Morse}(D,\mathcal{F},\ell)$ be the set of Morse function on $(D,\mathcal{F})$ for which $\ell$ is not a critical value and $\mathrm{Morse}_x(D,\mathcal{F},\ell)$ denote the set of Morse functions on $(D,\mathcal{F})$ for which $x$ is the unique stratified critical point at level $\ell$. (If the stratification is clear from context, we may omit the second argument from these expressions.)

\underline{Topological events:} It is shown in \cite[Corollary~5.8]{bmr20} that for any regular affine stratified set $(D,\mathcal{F})$, $\mathcal{E}_\mathcal{F}$ is countable and so we endow it with its maximal $\sigma$-algebra. The same result also states that, for a Gaussian field satisfying Assumption~\aref{a:basic}, the mapping from $f$ to $[\{f\geq\ell\}]_\mathcal{F}$ is measurable (for any $\ell\in\R$). We therefore define a \emph{topological event} to be an event $A$ which is measurable with respect to $[\{f\geq\ell\}]_\mathcal{F}$. Examples of topological events include crossing events (i.e., the event that there exists an excursion component intersecting two given boundary strata) of relevance to percolation theory or events involving the number of excursion set components.

\underline{Pivotal measures:} Let us fix two regular affine stratified sets $(D_1,\mathcal{F}_1)$ and $(D_2,\mathcal{F}_2)$, a Gaussian field $f$ satisfying Assumption~\aref{a:basic} and a level $\ell\in\R$. We let $\widetilde{f}$ be an independent Gaussian field with the same distribution as $f$ and, for $t\in[0,1]$, define the interpolated field
\[
    f^t=tf+\sqrt{1-t^2}\widetilde{f}.
\]
Observe that, for each $t$, $f^t$ has the same distribution as $f$. Given $x\in D_1$, $y\in D_2$ and $t\in[0,1)$, we define the event
\begin{equation}\label{e:conditioning}
    A^t_{x,y}=\big\{f(x)=f^t(y)=\ell,\nabla|_{F_x} f(x)=0, \nabla|_{F_y}f^t(y)=0\big\}
\end{equation}
where $F_x$ and $F_y$ denote the strata of $D_1$ and $D_2$ which contain $x$ and $y$ respectively. By convention, if $F$ is a zero-dimensional stratum, then $\nabla|_Ff\equiv 0$. We let $\phi^t_{x,y}$ denote the density of the Gaussian vector $(f(x),f^t(y),\nabla|_{F_x}f(x),\nabla|_{F_y}f^t(y))$ evaluated at $(\ell,\ell,0,0)$, which allows us to define the \emph{pivotal measure} as
\begin{equation}\label{e:piv_measure}
    \mu^t_{x,y}(H)=\E\Big[\big\lvert\det\nabla|_{F_x}^2f(x)\big\rvert\big\lvert\det\nabla|_{F_y}^2f^t(y)\big\rvert H(f,f^t)\Big|A^t_{x,y}\Big]\phi^t_{x,y}
\end{equation}
where $H$ is a real-valued, Borel measurable function and for a zero-dimensional stratum $F$, $\det\nabla^2|_Ff\equiv 1$ by convention. We remind the reader that these measures depend on the sets $D_1$ and $D_2$, the field $f$ and the level $\ell$, although this is suppressed from the notation for clarity. The term pivotal measure was used in a slightly different way in \cite{bmr20} to refer to the integral of $\mu^t_{x,y}$ over $t$, but we will find it more convenient to use this to refer directly to $\mu^t_{x,y}$.

For later use, we state a regularity property of $f$ and $f^t$ under the pivotal measure:

\begin{lemma}\label{l:conditioned_field_morse}
    Let $f$ satisfy Assumption~\aref{a:basic}, $\ell\in\R$ and $(D_1,\mathcal{F}_1)$ and $(D_2,\mathcal{F}_2)$ be regular affine stratified sets. For distinct points $x\in D_1$, $y\in D_2$ and $t\in[0,1)$, conditional on $A^t_{x,y}$, with probability one
    \[
        f\in\mathrm{Morse}_x(D_1,\mathcal{F}_1,\ell)\qquad\text{and}\qquad f^t\in\mathrm{Morse}_y(D_2,\mathcal{F}_2,\ell).
    \]
\end{lemma}
\begin{proof}
Since $(f,f^t)$ has the same distribution as $(f^t,f)$, it is sufficient to prove the first claim. In doing so, we make repeated use of the fact that the Gaussian vectors formed by evaluating $(f,\nabla f,\nabla^2 f,\widetilde{f},\nabla\widetilde{f},\nabla^2\widetilde{f})$ at any finite set of distinct points is non-degenerate in the sense of having a covariance matrix of full rank (more precisely this holds under the convention that $\nabla^2 f$ consists of the $d(d+1)/2$ elements on or below the diagonal of the Hessian matrix). This follows from the fact that the support of the spectral measure of $f$ contains an open set \cite[Lemma~13]{bmm24b}.

By Gaussian regression, $f$ conditioned on $A^t_{x,y}$ is equal in distribution to
\begin{equation}\label{e:conditioned_morse1}
    f(\cdot)+\Cov\big[f(\cdot),V^t_{x,y}\big]\Var\big[V^t_{x,y}\big]^{-1}((\ell,\ell,0,0)-V^t_{x,y})^T
\end{equation}
where the superscript $T$ denotes transposition and
\[
    V^t_{x,y}:=(f(x),f^t(y),\nabla|_{F_x}f(x),\nabla|_{F_y}f^t(y)).
\]
Since $f$ is $C^4$ by assumption, its covariance function is $C^8$ and from the representation in \eqref{e:conditioned_morse1}, $(f|A^t_{x,y})$ has a version which is $C^4$ almost surely. It also follows from evaluating the mean and variance of the gradient of \eqref{e:conditioned_morse1} that the conditioned field almost surely has a critical point at $x$ at level $\ell$.

For any point $u$ in a stratum $F\in \mathcal{F}_1$, the Gaussian vector $(\nabla|_Ff(u),\nabla|_F^2f(u)|A^t_{x,y})$ is non-degenerate. It then follows from \cite[Corollary~11.3.5]{at07} that for any $\epsilon>0$, $(f|A^t_{x,y})$ restricted to $D\setminus B(x,\epsilon)$ is Morse almost surely, where $B(x,\epsilon)$ denotes the open Euclidean ball around $x$ of radius $\epsilon$. (The reader checking definitions may find it helpful to consult \cite[Corollary~11.3.2]{at07} which has a much simpler statement but applies only to boxes rather than more general stratified sets.) Applying this result to a countable sequence $\epsilon_n\searrow 0$ shows that $(f|A^t_{x,y})$ satisfies the second and third properties of a Morse function for all stratified critical points other than $x$. Applying Bulinskaya's lemma (Lemma~11.2.10 of \cite{at07}) in a similar way to $F\setminus B(x,\epsilon)$ for each stratum $F\in\mathcal{F}_1$ shows that $x$ is the only stratified critical point of $f$ at level $\ell$ almost surely. 

Since $(\nabla^2 f(x)|A^t_{x,y})$ is a non-degenerate Gaussian vector, we see that with probability one
\[
    \big(\det\nabla|_{F_x}^2f(x)\;\big|\;A^t_{x,y}\big)\neq 0.
\]
This implies that $(f|A^t_{x,y})$ satisfies the second point in the definition of a Morse function. Finally, letting $F_x^\perp$ denote the maximal affine subspace orthogonal to $F_x$ and passing through $x$, we see that $(\nabla|_{F_x^\perp}f(x)|A^t_{x,y})$ is non-degenerate, so that each of its elements is non-zero almost surely. This verifies the third point in the definition of a Morse function, completing the proof.
\end{proof}

\underline{Topological derivatives:} Let us fix a regular affine stratified set $(D,\mathcal{F})$ and a level $\ell$. For $A\subset\mathcal{E}_\mathcal{F}$ and $g\in \mathrm{Reg}(D,\mathcal{F},\ell)$ we adopt the notation
\[
    \ind_A(g)=\begin{cases}
        1 &\text{if }[\{g\geq\ell\}]_\mathcal{F}\in A\\
        0 &\text{otherwise}.
    \end{cases}
\]
We wish to define a type of discrete derivative, which describes how this indicator function changes when $g$ undergoes a small, positive perturbation near a particular point $x\in D$. Observe that for a sufficiently small perturbation, the indicator function will not change unless $g$ has a critical point at $x$ at level $\ell$. Hence the derivative will be defined on $\mathrm{Morse}_x(D,\mathcal{F},\ell)$.

Now let $\rho:\R^d\to\R$ be smooth, compactly supported, non-negative, and strictly positive at the origin. For $\epsilon,\delta>0$, we set $\rho_{\delta,\epsilon,x}(y)=\delta\rho(\epsilon^{-1}(y-x))$. For $x\in D$, $A\subset\mathcal{E}_\mathcal{F}$ and $g\in \mathrm{Morse}_x(D,\mathcal{F},\ell)$ we define the \emph{topological derivative} of $\ind_A$ at $g$ as
\begin{equation}\label{e:top_derivative_events}
    d_x\ind_A(g):=\lim_{\epsilon\searrow 0}\lim_{\delta\searrow0}\Big(\ind_A(g+\rho_{\delta,\epsilon,x})-\ind_A(g-\rho_{\delta,\epsilon,x})\Big).
\end{equation}
We note that, since $x$ is a non-degenerate critical point of $g$, for $\epsilon$ and $\delta$ sufficiently small, $g\pm\rho_{\delta,\epsilon,x}\in\mathrm{Reg}(D,\mathcal{F},\ell)$, so the term inside the limit above is well defined. It follows from a standard Morse theory argument (see Lemma~5.4 and Remark~5.6 of \cite{bmr20} or Lemma 12 of \cite{bmm24b}) that for any choices of $\epsilon,\delta>0$ sufficiently small, the excursion sets $\{g+\rho_{\delta,\epsilon,x}\geq\ell\}$ (respectively $\{g-\rho_{\delta,\epsilon,x}\geq\ell\}$) can be mapped to one another by a stratified isotopy, and so the limit above is well defined. In the context of a Gaussian field $f$ and its interpolation $f^t$, we use the abbreviated notation
\[
    d_x\ind_A:=d_x\ind_A(f)\qquad\text{and}\qquad d_y^t\ind_A=d_y\ind_A(f^t).
\]
Observe that these expressions are well defined on the event $A^t_{x,y}$ by Lemma~\ref{l:conditioned_field_morse}.

We may now state the covariance formula for topological events:

\begin{theorem}[Theorem 2.14 of \cite{bmr20}]\label{t:cov_formula_events}
    Let $f$ satisfy Assumption~\aref{a:basic}, let $(D_1,\mathcal{F}_1)$ and $(D_2,\mathcal{F}_2)$ be regular affine stratified sets and let $A_1$ and $A_2$ be topological events (at level $\ell$) for $D_1$ and $D_2$ respectively, then
    \[
        \Cov[A_1,A_2]=\int_{\mathcal{F}_1\times\mathcal{F}_2}K(x-y)\int_0^1\mu^t_{x,y}(d_x\ind_{A_1}d_y^t\ind_{A_2})\;dtdxdy.
    \]
\end{theorem}

To gain some intuition for this formula, we can relate it to a classical interpolation formula for smooth functions of Gaussian vectors \cite{pit96,bgh01}:
\begin{proposition}
    Let $X$ be an $n$-dimensional centred Gaussian vector with covariance matrix $(K_{i,j})_{1\leq i,j\leq n}$ and let $g_1,g_2:\R^n\to\R$ be absolutely continuous such that $g_i(X)$ and $\|\nabla g_i(X)\|$ are square integrable for $i=1,2$. Then
    \[
        \Cov[g_1(X),g_2(X)]=\sum_{1\leq i,j\leq n}K_{i,j}\int_0^1\E\Big[\frac{dg_1(X)}{dX_i}\frac{dg_2(X^t)}{dX_j^t}\Big]\;dt
    \]
    where $X^t:=tX+\sqrt{1-t^2}\widetilde{X}$ and $\widetilde{X}$ is an independent copy of $X$.
\end{proposition}
The proof of this statement uses only calculus and some standard properties of Gaussian vectors. Comparing this result with Theorem~\ref{t:cov_formula_events} reiterates our interpretation of $d_x\ind_{A}$ as a type of derivative. The appearance of the measure $\mu^t_{x,y}$, rather than $\E$, can be compared to the Kac-Rice formula and reflects the fact that the topological derivative can only be non-zero if the field has a critical point at level $\ell$.

\subsection{Covariance formula for functionals}
We now extend the covariance formula to a general class of topological functionals.

\underline{Topological functionals:} Let $\Xi$ denote the set of triples $(D,\mathcal{F}, E)$ where $(D,\mathcal{F})$ is a regular affine stratified set and $E\in\mathcal{E}_\mathcal{F}$. We define a \emph{topological functional} to be any function $\Phi:\Xi\to\R$ such that $\Phi$ is zero whenever $E=[\emptyset]_\mathcal{F}$. Given such a functional, in a slight abuse of notation we will write $\Phi(D,g)=\Phi(D, g, \ell)$ to denote the value of the functional applied to a particular affine stratified set $(D,\mathcal{F})$ and the stratified isotopy class $[\{g\geq\ell\}]_\mathcal{F}$, assuming that $g\in\mathrm{Reg}(D,\mathcal{F},\ell)$.

We extend our previous definition of topological derivative to this setting, that is
\begin{equation}\label{e:top_derivative_functional}
    d_x\Phi(D,g):=\lim_{\epsilon\searrow 0}\lim_{\delta\searrow0}\Big(\Phi(D,g+\rho_{\delta,\epsilon,x})-\Phi(D,g-\rho_{\delta,\epsilon,x})\Big)
\end{equation}
where $g\in\mathrm{Morse}_x(D,\mathcal{F},\ell)$ and $\rho_{\delta,\epsilon,x}$ is defined as before. We say that a topological functional $\Phi$ is \emph{Lipschitz} if there exists $C>0$ such that $\lvert d_x\Phi(D,g)\rvert\leq C$ for all $D$ and $g$ and we define $\|\Phi\|_\mathrm{Lip}$ to be the smallest such $C$. Note that this class is closed under composition with Lipschitz functions; that is, if $F:\R\to\R$ is Lipschitz and $\Phi$ is a Lipschitz topological functional, then so is $F\circ\Phi$. Lipschitz topological functionals may be controlled in terms of the critical points of the underlying function:

\begin{lemma}\label{l:top_crit}
    Let $\Phi$ be a Lipschitz topological functional and $(D,\mathcal{F})$ a regular affine stratified set. For any $\ell\in\R$ and $g\in\mathrm{Morse}(D,\mathcal{F},\ell)$ for which all stratified critical points have distinct critical values,
    \[
        \lvert\Phi(D,g,\ell)\rvert\leq \|\Phi\|_\mathrm{Lip}\overline{N}_\mathrm{Crit}(D,g)
    \]
    where $\overline{N}_\mathrm{Crit}(D,g)$ denotes the number of stratified critical points of $g$ in $D$.
\end{lemma}
\begin{proof}
    We consider how the value of $\Phi$ changes when we change the value of $\ell$. Let $\ell_1<\ell_2<\dots<\ell_n$ be the stratified critical values of $g$ which are greater than $\ell=:\ell_0$ (note that $g$ has only finitely many stratified critical points since they are all non-degenerate and $D$ is compact). We argue that $\Phi(D,g,\ell^\prime)$ does not change as $\ell^\prime$ varies between consecutive critical values and that the change when passing through a critical value is at most $\|\Phi\|_{\mathrm{Lip}}$. Assuming these two claims, by the triangle inequality we deduce that
    \[
        \lvert \Phi(D,g,\ell)\rvert\leq\lvert\Phi(D,g,\ell_n+1)\rvert+n\|\Phi\|_\mathrm{Lip}.
    \]
    This is enough to conclude since $n\leq\overline{N}_\mathrm{Crit}(D,g)$ and $\ell_n$ is the maximum of $g$, so $\{g\geq\ell_n+1\}=\emptyset$ meaning that $\Phi(D,g,\ell_n+1)=0$ by our definition of a topological functional.

    It remains to prove the two claims. It is a standard result in (stratified) Morse theory that if $g$ has no stratified critical values in the interval $[a,b]$ then $\{g\geq a\}$ and $\{g\geq b\}$ lie in the same stratified isotopy class, proving the first claim. A rigorous statement and proof of this result can be found in \cite[Part I, Chapter 3.2]{gm88} (in a very general form) or in \cite[Lemma~12]{bmm24b} (in a setting very similar to the current one).
    
    The second claim follows from the definition of a Lipschitz topological functional if we can show for all $\epsilon,\delta,\gamma>0$ sufficiently small that $\{g+\gamma\geq\ell_i\}$ lies in the same stratified isotopy class as $\{g+\rho_{\delta,\epsilon,x}\geq\ell_i\}$ where $x$ is a stratified critical point of $g$ at level $\ell_i$. Once again this can be deduced from the Morse theory result \cite[Lemma~12]{bmm24b}, since interpolating between $g+\gamma$ and $g+\rho_{\delta,\epsilon,x}$ cannot create any stratified critical points.
\end{proof}

We will often combine this result with bounds on the moments of the number of critical points of a Gaussian field. (The fact that all stratified critical points of a Gaussian field have distinct critical values follows from Bulinskaya's lemma \cite[Lemma~11.2.10]{at07}.) Given a regular affine stratified set $(D,\mathcal{F})$ and a $C^2$ function $g:\R^d\to\R$, let $N_\mathrm{Crit}(F,g)$ denote the number of stratified critical points of $g$ belonging to the stratum $F\in\mathcal{F}$ (or equivalently, the number of critical points of $g|_F$).

\begin{lemma}\label{l:crit_moments}
    Let $f$ satisfy Assumption~\aref{a:basic} and $(D,\mathcal{F})$ be an affine stratified set.
    \begin{enumerate}
        \item For any $F\in\mathcal{F}$, $\E[N_{\mathrm{Crit}}(F,f)]\leq Cv_F(F)$ where $C>0$ depends only on $f$.
        \item If $f$ is $C^{k+1}$-smooth almost surely for $k\in\N$, then there exists $C>0$ depending only on $f$ such that for all $R\geq 1$
        \[
        \E\big[\big\lvert\overline{N}_\mathrm{Crit}(\Lambda_R,f)\big\rvert^k\big]\leq CR^{kd}
        \]
        where the stratification of $\Lambda_R$ consists of its open faces of all dimensions.
    \end{enumerate}
\end{lemma}

\begin{proof}
    By stationarity, we may assume that $F$ is contained in a subspace $S$ of $\R^d$. Then by the Kac-Rice theorem \cite[Corollary~11.2.2]{at07} and stationarity once more,
    \begin{align*}
        \E\big[N_\mathrm{Crit}(F,f)]&=\int_{F}\E[\lvert\det\nabla^2|_Ff(x)\rvert\;\big|\;\nabla|_Ff(x)=0\big]\phi_{\nabla|_Ff(x)}(0)\;dv_F(x)=cv_F(F)
    \end{align*}
    where $\phi_{X}$ denotes the density of the random vector $X$ and $c>0$ depends only on the distribution of $f$ and the basis vectors spanning $S$. By definition of a regular affine stratified space, these basis vectors must be a subset of the vectors parallel to the coordinate axes. Hence $c$ may take only a finite number of values for a given field $f$, proving the first point of the lemma.

    For $R\geq 1$, let $A_R:=\{x\in\Z^d\;|\;(x+\Lambda_1)\cap\Lambda_R\neq\emptyset\}$, and for $x\in\R^d$, let $N_x:=N_\mathrm{Crit}(x+\Lambda_1,f)$. It follows from \cite[Theorem~1.2]{gs23} that $N_x$ has finite $k$-th moment. Then by H\"older's inequality and stationarity
    \[
        \E[N_\mathrm{Crit}(\Lambda_R,f)^k]\leq\E\Big[\Big(\sum_{x\in A_R}N_x\Big)^k\Big]=\sum_{x_1,\dots,x_k\in A_R}\E[N_{x_1}\dots N_{x_k}]\leq \sum_{x_1,\dots,x_k\in A_R}\E[N_0^k]\leq cR^{kd}
    \]
    where $c>0$ depends only on $f$. Applying a near identical argument to the lower dimensional faces of $\Lambda_R$ (which we take as our strata) verifies the second point of the lemma.
\end{proof}

Extending the covariance formula to topological functionals makes use of an approximation argument along with dominated convergence. This requires bounds on the pivotal measures which were established in \cite{bmm24b}. We denote the total variation of a measure $\mu$ by $\|\mu\|$.

\begin{proposition}[{\cite[Proposition~2]{bmm24b}}]\label{p:intensity_integrable}
    Let $f$ satisfy Assumption~\aref{a:basic} and $(D_1,\mathcal{F}_1)$, $(D_2,\mathcal{F}_2)$ be regular affine stratified sets with $F_1\in\mathcal{F}_1$ and $F_2\in\mathcal{F}_2$. There exists $c>0$, depending only on the distribution of $f$ and the level $\ell$, such that
    \begin{equation}\label{e:pivotal_bounds}
        \sup_{x\in F_1}\int_0^1\int_{\substack{y\in F_2\\\lvert y-x\rvert\leq 1}}\|\mu^t_{x,y}\|+\|\mu^t_{y,x}\|\;dtdy<c\qquad\text{and}\qquad\sup_{t\in[0,1)}\sup_{\substack{(x,y)\in F_1\times F_2\\\lvert x-y\rvert\geq 1}}\|\mu^t_{x,y}\|+\|\mu^t_{y,x}\|<c.
    \end{equation}
\end{proposition}
The proof of this result involves using the `divided difference' method to control the degeneracy of the Gaussian vector $(f(x),f^t(y),\nabla|_{F_1} f(x),\nabla|_{F_2} f^t(y))$ as $t\to 1$ and $x\to y$.

We may now state our covariance formula for topological functionals:

\begin{theorem}\label{t:cov_formula}
    Let $F,G:\R\to\R$ be Lipschitz, $(D_1,\mathcal{F}_1)$ and $(D_2,\mathcal{F}_2)$ be regular affine stratified sets, $\Phi,\Psi$ be Lipschitz topological functionals and let $f$ satisfy Assumption~\aref{a:basic}. Then denoting $\Phi_1=\Phi(D_1,f)$ and $\Psi_2=\Psi(D_2,f)$,
    \begin{equation}\label{e:cov_formula}
        \Cov[F(\Phi_1),G(\Psi_2)]=\int_{\mathcal{F}_1\times\mathcal{F}_2}K(x-y)\int_0^1\mu^t_{x,y}(d_xF(\Phi_1)d_y^tG(\Psi_2))\;dtdxdy.
    \end{equation}
\end{theorem}

\begin{proof}
    For any Borel sets $B_1,B_2\subseteq\R$, $\{\Phi_1\in B_1\}$ and $\{\Psi_2\in B_2\}$ are topological events to which we can apply the covariance formula Theorem~\ref{t:cov_formula_events}. Therefore by linearity, \eqref{e:cov_formula} holds whenever $F$ and $G$ are simple measurable functions.
    
    If $F$ and $G$ are bounded, continuous functions then we can choose two sequences $(F_n)$ and $(G_n)$ which are simple measurable, uniformly bounded and converge pointwise to $F$ and $G$ respectively. For any given $x\in D_1$, $y\in D_2$ with $x\neq y$ and $t\in[0,1)$, by Lemma~\ref{l:conditioned_field_morse} on the complement of a $\mu^t_{x,y}$-null set
    \[
        f\in\mathrm{Morse}_x(D_1,\mathcal{F}_1,\ell)\qquad\text{and}\qquad f^t\in\mathrm{Morse}_y(D_2,\mathcal{F}_2,\ell).
    \]
    Therefore, by the observations following \eqref{e:top_derivative_events}, there are two (random) values $c_1,c_2\in\R$ such that for all $\epsilon,\delta>0$ sufficiently small
    \[
        \Phi(D_1,f+\rho_{\delta,\epsilon,x})=c_1\qquad\text{and}\qquad\Phi(D_1,f-\rho_{\delta,\epsilon,x})=c_2.
    \]
    Then using the convergence of $(F_n)$ and the definition of the topological derivative
    \[
        \lim_{n\to\infty}d_xF_n(\Phi_1)=\lim_{n\to\infty}F_n(c_1)-F_n(c_2)=F(c_1)-F(c_2)=d_xF(\Phi_1).
    \]
    An analogous argument shows that $\lim_{n\to\infty}d_y^tG_n(\Psi_2)=d_y^tG(\Psi_2)$. Hence, since $\lvert d_xF_n(\Phi_1)\rvert\leq2\sup_n\|F_n\|_\infty$ and $\lvert d_xG_n(\Phi_1)\rvert\leq2\sup_n\|G_n\|_\infty$, by dominated convergence
    \[
        \lim_{n\to\infty}\mu^t_{x,y}(d_xF_n(\Phi_1)d_y^tG_n(\Psi_2))=\mu^t_{x,y}(d_xF(\Phi_1)d_y^tG(\Psi_2)).
    \]
    Then using the bounds $\lvert\mu^t_{x,y}(d_xF_n(\Phi_1)d_y^tG_n(\Psi_2))\rvert\leq 4\sup_n\|F_n\|_\infty\|G_n\|_\infty\|\mu^t_{x,y}\|$ and $\lvert K\rvert\leq 1$ along with \eqref{e:pivotal_bounds}, we can apply the dominated convergence theorem once more to conclude that as $n\to\infty$
    \begin{align*}
        \Cov[F_n(\Phi_1),G_n(\Psi_2)]\to \int_{\mathcal{F}_1\times\mathcal{F}_2}K(x-y)\int_0^1\mu^t_{x,y}(d_xF(\Phi_1)d_y^tG(\Psi_2))\;dtdxdy.
    \end{align*}
    Using the uniform boundedness of $F_n$ and $G_n$ once more, we see that
    \[
        \mathrm{Cov}[F_n(\Phi_1),G_n(\Psi_2)]\to\Cov[F(\Phi_1),G(\Psi_2)],
    \]
    which finally establishes \eqref{e:cov_formula} when $F$ and $G$ are bounded and continuous.

    Finally we consider the case that $F$ and $G$ are assumed only to be Lipschitz. The truncated functions
    \[
        F_n:=(F\wedge n)\vee (-n)\qquad\text{and}\qquad G_n:=(G\wedge n)\vee (-n)
    \]
    are uniformly Lipschitz and converge pointwise to $F$ and $G$ respectively as $n\to\infty$. We can extend the covariance formula using the same dominated convergence argument as before, making use of the two following facts: 1) $F_n(\Phi_1)$ and $G_n(\Psi_2)$ are dominated by $\lvert F(\Phi_1)\rvert$ and $\lvert G(\Psi_2)\rvert$ respectively, which are square-integrable by Lemma~\ref{l:crit_moments}, and 2)
    \[
        \lvert d_xF_n(\Phi_1)\rvert\leq \|F\|_\mathrm{Lip}\|\Phi\|_\mathrm{Lip}\qquad\text{and}\qquad \lvert d_y^tG_n(\Psi_2)\rvert\leq \|G\|_\mathrm{Lip}\|\Psi\|_\mathrm{Lip},
    \]
    by definition of the topological derivative and Lipschitz topological functions.    
\end{proof}

Quasi-association of Lipschitz topological functionals is a straightforward consequence of this covariance formula. For $A\subset\R^d$ and $c>0$, let $A^{+c}$ denote the set of points with Euclidean distance at most $c$ from $A$ and let $\lvert A\rvert$ denote the $d$-dimensional Lebesgue measure of $A$. Throughout this and subsequent proofs, we let $C$ denote a positive constant, the value of which may change between occurrences.

\begin{corollary}[Quasi-association of Lipschitz topological functionals]\label{c:quasi_association}
    Let $(D_1,\mathcal{F}_1)$ and $(D_2,\mathcal{F}_2)$ be regular affine stratified sets and suppose that there exists $m>0$ such that for $i=1,2$ and any $x\in\R^d$, $B(x,1)$ intersects at most $m$ strata of $D_i$. Under the assumptions of Theorem~\ref{t:cov_formula}, there exists $C>0$ depending only on $\ell$ and the distribution of $f$ such that
    \begin{equation}\label{e:quasi_assoc_top}
        \big\lvert \Cov[F(\Phi_1),G(\Psi_2)]\big\rvert\leq Cm^2 \|F\|_{\mathrm{Lip}}\|G\|_{\mathrm{Lip}}\|\Phi\|_\mathrm{Lip}\|\Psi\|_{\mathrm{Lip}}\int_{D_1^{+2}\times D_2^{+2}} \widetilde{K}(x-y)\;dxdy
    \end{equation}
    where we recall that $\widetilde{K}(u):=\sup_{\lvert v-u\rvert\leq 1}\lvert K(v)\rvert$.
\end{corollary}
\begin{proof}
    By Theorem~\ref{t:cov_formula} and the fact that $F$, $G$, $\Phi$ and $\Psi$ are Lipschitz
    \[
        \big\lvert \Cov[F(\Phi_1),G(\Psi_2)]\big\rvert\leq \|F\|_{\mathrm{Lip}}\|G\|_{\mathrm{Lip}}\|\Phi\|_\mathrm{Lip}\|\Psi\|_{\mathrm{Lip}}\sum_{\substack{F_1\in\mathcal{F}_1\\F_2\in\mathcal{F}_2}}\int_{F_1\times F_2}\lvert K(x-y)\rvert\int_0^1\|\mu^t_{x,y}\|\;dtdv_{F_1}(x)dv_{F_2}(y).
    \]
    We consider separately the contribution to this integral when $\lvert x-y\rvert$ is greater than or less than one. In what follows, we repeatedly make use of the fact that, for a subset $F$ of an affine subspace of $\R^d$ and a continuous function $h:\R^d\to\R$
    \[
        \int_F \lvert h(u)\rvert\;dv_F(u)\leq c_d\int_{F^{+1/2}}\sup_{v\in B(u,1/2)}\lvert h(v)\rvert\;du
    \]
    for some constant $c_d>0$ depending only on the dimension $d$.
    
    By the second of our bounds on the pivotal measure (in Proposition~\ref{p:intensity_integrable}), the previous observation and our assumption on $D_1$ and $D_2$,
    \begin{equation}\label{e:qa_proof1}
    \begin{aligned}
        \sum_{\substack{F_1\in\mathcal{F}_1\\F_2\in\mathcal{F}_2}}\int_{F_1\times F_2}\lvert K(x-y)\rvert\ind_{\lvert x-y\rvert\geq 1}\int_0^1\|\mu^t_{x,y}\|\;dtdv_{F_1}(x)dv_{F_2}(y)&\leq C\sum_{\substack{F_1\in\mathcal{F}_1\\F_2\in\mathcal{F}_2}}\int_{F_1^{+1}\times F_2^{+1}}\widetilde{K}(x-y)\;dxdy\\
        &\leq Cm^2\int_{D_1^{+1}\times D_2^{+1}}\widetilde{K}(x-y)\;dxdy.
    \end{aligned}
    \end{equation}
    By the first bound in Proposition~\ref{p:intensity_integrable}, our assumption on $D_2$ and the trivial bound $\lvert K\rvert\leq 1$, for any $x\in D_1$
    \[
        \sum_{F_2\in\mathcal{F}_2}\int_{F_2}\lvert K(x-y)\rvert\ind_{\lvert x-y\rvert\leq 1}\int_0^1\|\mu^t_{x,y}\|\;dtdv_{F_2}(y)\leq Cm\ind_{x\in D_2^{+1}}.
    \]
    It follows that
    \begin{equation}\label{e:qa_proof2}
    \begin{aligned}
        \sum_{\substack{F_1\in\mathcal{F}_1\\F_2\in\mathcal{F}_2}}\int_{F_1\times F_2}\lvert K(x-y)\rvert\ind_{\lvert x-y\rvert\leq 1}\int_0^1\|\mu^t_{x,y}\|\;dtdv_{F_1}(x)dv_{F_2}(y)&\leq Cm\sum_{F_1\in\mathcal{F}_1}\int_{F_1^{+1/2}}\ind_{x\in D_2^{+3/2}}\;dx\\
        &\leq Cm^2\lvert D_1^{+3/2}\cap D_2^{+3/2}\rvert.
    \end{aligned}
    \end{equation}
    Finally, using the fact that $\widetilde{K}(u)=1$ for $\lvert u\rvert\leq 1$ we have
    \[
        \int_{D_1^{+2}\times D_2^{+2}}\widetilde{K}(x-y)\;dxdy\geq \int_{D_1^{+3/2}\cap D_2^{+3/2}}\int_{B(x,1/2)}\widetilde{K}(x-y)\;dydx\geq c_d\lvert D_1^{+3/2}\cap D_2^{+3/2}\rvert
    \]
    for some $c_d>0$ depending only on the dimension $d$. Combining this with \eqref{e:qa_proof1} and \eqref{e:qa_proof2} completes the proof of quasi-association
\end{proof}

\subsection{Excursion/level-set functionals and the Euler characteristic}
To complete this section, we discuss the specific topological functionals which are of interest for our main results:

\underline{Bounded excursion/level-set functionals:} Given a regular affine stratified set $(D,\mathcal{F})$ and $A\subset\R^d$ we let $\mathrm{Comp}(A,D)$ denote the set of connected components of $A$ that intersect $D$ but not $\partial D$. Let $\Theta$ denote the space of triples $(D,\mathcal{F},[E]_\mathcal{F})$ where $(D,\mathcal{F})$ is a regular affine stratified set and $E\in\mathrm{Comp}(\{g\geq 0\},D)$ for some $g\in\mathrm{Reg}(D,\mathcal{F},0)$. We say that $\varphi:\Theta\to\R$ is an \emph{excursion component functional} if 
\[
    \varphi(D_1,\mathcal{F}_1,[E]_{\mathcal{F}_1})=\varphi(D_2,\mathcal{F}_2,[E]_{\mathcal{F}_2})
\]
for all $(D_1,\mathcal{F}_1,[E]_{\mathcal{F}_1}),(D_2,\mathcal{F}_2,[E]_{\mathcal{F}_2})\in\Theta$. Intuitively, this means that $\varphi$ depends only on the topology of the excursion component $E$ and not on the (regular affine stratified) set in which it is contained. The functional can therefore be expressed using the shortened notation $\varphi([E])$ to denote the value $\varphi(D,\mathcal{F},[E]_{\mathcal{F}})$ for any suitable $D$ containing the excursion component $E$.

We say that a topological functional $\Phi$ is an \emph{excursion-set functional} if it can be expressed in the form
\begin{equation}\label{e:es_functional2}
    \Phi(D,g,\ell)=\sum_{E\in\mathrm{Comp}(\{g\geq\ell\},D)}\varphi(D,\mathcal{F},[E]_{\mathcal{F}})
\end{equation}
where $\varphi$ is an excursion component functional. We say that $\Phi$ is a \emph{bounded excursion-set functional} if $\varphi$ is bounded in this representation.  We define a \emph{(bounded) level-set functional} analogously if we replace $\{f\geq\ell\}$ by $\{f=\ell\}$.

As a simple, but important, example of a bounded excursion set functional, taking $\varphi\equiv 1$ yields the component count (i.e., the number of connected components of the excursion set in the domain). Taking $\varphi$ to be $1$ for excursion components with a particular topology (say, doubly connected) and $0$ otherwise results in $\Phi$ being the number of excursion components with that topology. In general, the Betti numbers of the excursion set cannot be represented as a bounded excursion-set functional, since they can take arbitrarily large absolute value on a single component.

The reason we restrict $\varphi$ to be bounded, is that this ensures the change in the functional when passing through a critical point is also bounded, resulting in a Lipschitz topological functional:

\begin{lemma}\label{l:exc_lipschitz}
    If $\Phi$ is a bounded excursion/level-set functional, then it is also a Lipschitz topological functional.
\end{lemma}
\begin{proof}
    It is immediate that $\Phi$ is a topological functional (in particular if $\{g\geq\ell\}=\emptyset$ then $\Phi(D,g,\ell)=0$ by \eqref{e:es_functional2}). We need only verify that the topological derivative of $\Phi$ is uniformly bounded. Let $(D,\mathcal{F})$ be a regular affine stratified set, $x\in D$ and $g\in\mathrm{Morse}_x(D,\mathcal{F},\ell)$. Once again, it is a standard result of stratified Morse theory that for $\delta,\epsilon>0$ sufficiently small, the number of components of $\{g+\rho_{\delta,\epsilon,x}\geq\ell\}$ and $\{g-\rho_{\delta,\epsilon,x}\geq\ell\}$ inside $D$ differ by at most one (and the same is true for the respective level sets). This statement was proven in \cite[Lemma~2]{bmm24b} for a stratified box, and the proof holds verbatim for a regular affine stratified set. It therefore follows from the definition of the topological derivative in \eqref{e:top_derivative_functional} that $\lvert d_x\Phi(D,g,\ell)\rvert$ is bounded above by three times the supremum of $\lvert \varphi\rvert$, which is bounded by assumption.
\end{proof}

We record a preliminary result regarding this class of functionals, which controls the contribution from components intersecting a given stratum. For a regular affine stratified set $(D,\mathcal{F})$, $g\in\mathrm{Reg}(D,\mathcal{F},\ell)$, a bounded excursion/level-set functional $\Phi$ and distinct strata $F_1,\dots,F_n\in\mathcal{F}$ we define $\Phi(D,g,\ell,\overline{\cup_{i=1}^nF_i})$ to be the restriction of the sum in \eqref{e:es_functional2} to components which intersect $\overline{\cup_{i=1}^nF_i}$.

\begin{lemma}\label{l:boundary_contribution}
    Given $(D,\mathcal{F})$, $g$, $\ell$, $\Phi$ and $F_1,\dots,F_n\in\mathcal{F}$ as above, there exists $C>0$ depending only on $\Phi$ such that
    \[
        \lvert\Phi(D,g,\ell,\overline{\cup_{i=1}^nF_i})\rvert\leq C\overline{N}_\mathrm{Crit}(\overline{\cup_{i=1}^nF_i},g).
    \]
\end{lemma}
\begin{proof}
    First suppose that $\Phi$ is an excursion set functional. Let $\overline{F}:=\overline{\cup_{i=1}^nF_i}$. By definition of a stratification, we can decompose this as $\overline{F}=\cup_{i=1}^mF_i$ for strata $F_1,\dots,F_m\in\mathcal{F}$ where $m\geq n$. If $E$ is an element of $\mathrm{Comp}(\{g\geq\ell\},D)$ which intersects $\overline{F}$, then $g|_{E\cap\overline{F}}$ attains a maximum at some point. This point must be contained in some stratum $F_i$ for $i\in\{1,\dots,m\}$ and must therefore be a stratified critical point. Clearly such point are distinct for different excursion components. From the definition of an excursion set functional in \eqref{e:es_functional2}, we have
    \[
        \lvert \Phi(D,g,\ell)\rvert\leq \sup\lvert\varphi\rvert \overline{N}_\mathrm{Crit}(\overline{F},g)
    \]
    which proves the statement since $\varphi$ is bounded by assumption.

    Suppose now that $\Phi$ is a level-set functional. Since $g$ has no critical points at level $\ell$, each level set component $L$ which is inside $D$ must be a closed hypersurface that surrounds some component $E$ of $\{g>\ell\}$ or of $\{g<\ell\}$ (i.e., $E$ is contained in the bounded component of the complement of $L$). Clearly this mapping from $L$ to $E$ is injective. Repeating the argument for excursion set components, we can find either a maximum or minimum of $g|_{\overline{E}\cap\overline{F}}$ which must be a stratified critical point and yields the required bound on $\Phi$.
\end{proof}

\underline{Euler characteristic:} We recall that the Euler characteristic is an integer-valued topological invariant defined for sufficiently nice sets. For instance, for suitable planar sets, it is equal to the number of connected components of the set minus the number of `holes' in the set and analogous interpretations are possible in higher dimensions (see \cite[Chapter~6]{at07}). We consider the case in which $\Phi(D,\mathcal{F},E)$ is the Euler characteristic of $E$. In contrast with the definition of an excursion/level-set functional, this definition includes components which intersect the boundary of $D$.

It follows from the Poincar\'e-Hopf theorem that the Euler characteristic of an excursion set may be defined as a signed sum over the number of critical points of the (Morse) function defining the excursion set. To avoid unnecessary definitions, we will not specify exactly how the sign for boundary critical points is determined.

\begin{lemma}\label{l:ec_decomp}
    Let $(D,\mathcal{F})$ be a regular affine stratified set, $g\in\mathrm{Reg}(D,\mathcal{F},\ell)\cap\mathrm{Morse}(D,\mathcal{F},\ell)$ and $\Phi$ be the Euler characteristic, then
    \begin{equation}\label{e:ec_decomp}
        \Phi(D,g,\ell)=\sum_{i=0}^d(-1)^{d-i}N_\mathrm{Crit}^{(i)}(D,g,\ell)+\sum_{F\in\mathcal{F}^{(<d)}}\sum_{x\in\overline{N}_\mathrm{Crit}(F,g)}\alpha(x,g(x),\nabla g(x),\nabla^2g(x),\mathcal{F})
    \end{equation}
    where $N_\mathrm{Crit}^{(i)}(D,g,\ell)$ denotes the number of critical points $x$ of $g$ in $D$ such that $g(x)\geq\ell$ and $\nabla^2 g(x)$ has exactly $i$ negative eigenvalues, $\mathcal{F}^{(<d)}$ denotes the subset of all strata in $\mathcal{F}$ with dimension less than $d$ and $\alpha$ is a function taking values in $\{-1,0,1\}$.
\end{lemma}
\begin{proof}
    This immediately follows from \cite[Corollary~9.3.5]{at07} (see also (9.4.1)-(9.4.5) of \cite{at07} for the particular case that $D$ is a box).
\end{proof}

This decomposition ensures that our topological functional is Lipschitz:

\begin{lemma}\label{l:ec}
    The Euler characteristic is a Lipschitz topological functional.
\end{lemma}
\begin{proof}
    Let $(D,\mathcal{F})$ be a regular affine stratified set and $\overline{g}\in\mathrm{Reg}(D,\mathcal{F},0)$. It is possible to find a $g\in\mathrm{Morse}(D,\mathcal{F},0)\cap\mathrm{Reg}(D,\mathcal{F},0)$ which has the same excursion set as $\overline{g}$ in $D$ and so by Lemma~\ref{l:ec_decomp}, the Euler characteristic of $\{g\geq 0\}\cap D$ is well defined. The Euler characteristic is well known to be invariant under homeomorphism, so in particular is uniquely defined for any element of the stratified isotopy class of $\{g\geq 0\}$. Since the Euler characteristic of the empty set is zero, we have shown that it is a topological function.

    Given $x\in D$, $\ell\in\R$ and $g\in\mathrm{Morse}_x(D,\mathcal{F},\ell)$, we require a uniform bound on $\lvert d_x\Phi(D,g)\rvert$ as defined in \eqref{e:top_derivative_functional}. By Lemma~\ref{l:ec_decomp}, $\Phi(D,g\pm\rho_{\delta,\epsilon,x},\ell)$ may be decomposed in terms of the stratified critical points of the two functions. For $\epsilon,\delta>0$ sufficiently small, we see that each such stratified critical point, apart from $x$, will contribute in the same way for each function. Hence the topological derivative has absolute value at most one, and so $\Phi$ is Lipschitz.
\end{proof}

\section{Additivity and stabilisation}\label{s:additivity+stabilisation}
In this section, we formalise the additivity and stabilisation properties of topological functionals described in Section~\ref{ss:proof_outline}.

\subsection{Additivity}\label{ss:additivity}
We begin by proving an approximate additivity result, which says that the value of a topological functional on a macroscopic box (of scale $R$) is well approximated by summing its value over mesoscopic boxes (of scale $r\ll R$) which are well separated (at scale $a\gg 1$). We therefore define sequences $(a_R)_{R\geq 1}$ and $(r_R)_{R\geq 1}$ such that
\begin{equation}\label{e:scales}
    \begin{aligned}
    \frac{r_R}{a_R}\in 2\N\qquad\text{and}\qquad\frac{R}{r_R+4a_R}\in 2\N\qquad\text{for each }R\geq 20,\\
    a_R,r_R\to\infty\qquad\text{and}\qquad \frac{a_R}{r_R},\frac{r_R}{R}\to 0\qquad\text{as }R\to\infty.
    \end{aligned}
\end{equation}
Below we will write $a=a_R$ and $r=r_R$ for notational clarity. These conditions allow us to tile $\Lambda_R$ with mesoscopic cubes of side length $r+4a$ which in turn can be tiled by microscopic cubes of side length $a$ (although we emphasise that $a$ is large, so is only microscopic with respect to $R$), see Figure~\ref{fig:box_stratification}. We define
\[
    \chi_R:=\Lambda_R\cap(r+4a)(1/2+\Z)^d\qquad\text{and}\qquad U_R=\overline{\Lambda_R\setminus\bigcup_{x\in\chi_R}x+\Lambda_r}
\]
to be the set of centres of mesoscopic cubes and the (closed) region that is left when removing all mesoscopic cubes respectively. Below, we abbreviate $\Phi(D)=\Phi(D,f,\ell)$ whenever the field $f$ and level $\ell$ are clear from context.

\begin{proposition}[Additivity]\label{p:additivity}
    Let $f$ satisfy Assumptions~\aref{a:basic} and \aref{a:cov_decay_int} and suppose that $(a_R)_{R\geq 1}$ and $(r_R)_{R\geq 1}$ satisfy \eqref{e:scales}. If $\Phi$ is a bounded excursion/level-set functional, then as $R\to\infty$
    \[
        \Var\Big[\Phi(\Lambda_R)-\sum_{x\in\chi_R}\Phi(x+\Lambda_r)\Big]=o(R^d).
    \]
    If $\Phi$ is the Euler characteristic, then the left-hand side is $O(ar^{-1}R^d)=o(R^d)$ as $R\to\infty$.
\end{proposition}

The proof of this result for the Euler characteristic is relatively straightforward, since the latter is a local functional and therefore (exactly) additive. To avoid a precise (and somewhat irrelevant) definition of the general sets on which such additivity holds, we will give only the statement that we need for Gaussian excursion sets.

\begin{lemma}\label{l:ec_additivity}
    Let $f$ be a Gaussian field satisfying Assumption~\aref{a:basic}, $\Phi$ be the Euler characteristic and $\ell\in\R$. Suppose that $D_1,D_2,D_1\cup D_2,D_1\cap D_2\subset\R^d$ can each be expressed as a finite union of closed faces (of any dimension) of boxes in $\R^d$, then with probability one
    \[
        \Phi(D_1\cup D_2)=\Phi(D_1)+\Phi(D_2)-\Phi(D_1\cap D_2).
    \]
    In particular, in the setting of Proposition~\ref{p:additivity}, with probability one
    \begin{equation}\label{e:ec_additivity}
        \Phi(\Lambda_R)=\sum_{x\in\chi_R}\Phi(x+\Lambda_r)+\Phi(U_R)-\sum_{x\in\chi_R}\Phi(x+\partial\Lambda_r).
    \end{equation}
\end{lemma}
\begin{proof}
    \cite[Theorem~6.1.1]{at07} states that for $A$, $B$, $A\cup B$ and $A\cap B$ belonging to a class of sets known as `basic complexes', $\EC(A\cup B)=\EC(A)+\EC(B)-\EC(A\cap B)$ where we temporarily denote the Euler characteristic by $\EC$. It is shown in \cite[Theorem 6.2.2]{at07} that the excursion sets of `suitably regular' functions on boxes are basic complexes. The proof of this result generalises trivially to finite unions of closed faces of boxes (and in particular to the domains $x+\partial\Lambda_r$ and $U_R$). Finally it follows from \cite[Theorem 11.3.3]{at07} that a Gaussian field satisfying Assumption~\aref{a:basic} is suitably regular on any regular affine stratified set with probability one. Combining these observations proves the result.
\end{proof}

\begin{proof}[Proof of Proposition~\ref{p:additivity} (Euler characteristic)]
    Let $\Phi$ be the Euler characteristic, by quasi-association (Corollary~\ref{c:quasi_association}) applied to $U_R$ equipped with its minimal stratification,
    \[
        \Var[\Phi(U_R)]\leq C\int_{(U_R^{+2})^2}\widetilde{K}(x-y)\;dxdy\leq C\lvert U_R^{+2}\rvert\int_{\R^d}\widetilde{K}(x)\;dx=O(ar^{-1}R^d)
    \]
    where $C>0$ is independent of $R$ and the final equality can be justified by dividing $U_R$ into $(R/(r+4a))^d$ regions which are translations of $\overline{\Lambda_{r+4a}\setminus\Lambda_{r}}$, each of which has $d$-dimensional Lebesgue measure at most $Cr^{d-1}a$. An identical argument yields the same upper bound for the variance of $\sum_{x\in\chi_R}\Phi(x+\partial\Lambda_r)$. Combining these estimates with \eqref{e:ec_additivity} completes the proof.
\end{proof}

The proof of the additivity estimate for excursion/level-set functionals is more involved. The key input that we need is control over contributions from `large' components. Since the size of a component is not a topological property, to define large components, we refine our stratification by tiling $\Lambda_R$ using cubes of side-length $a$, and consider components which cross such cubes.

Given $a,R\geq 1$ such that $2a$ divides $R$, we define $\mathcal{F}^{(a)}_R$ to be the stratification of $\Lambda_R$ formed by the union of the minimal stratifications of $x+[0,a]^d$ for all $x\in a\Z^d\cap[-R/2,R/2)^d$. In other words, we divide $\Lambda_R$ into $(R/a)^d$ cubes of side length $a$ and include all of their respective open faces of different dimensions in the stratification.

We define the set of \emph{$a$-planes} to be the collection of $(d-1)$-dimensional hyperplanes passing through some point in $a\Z^d$ which are normal to one of the coordinate axes (see Figure~\ref{fig:a-planes}). Note that all of the lower dimensional strata of $\mathcal{F}^{(a)}_R$ are contained in some $a$-plane. Given an excursion/level-set functional $\Phi$, we define $\Phi^{(>a)}$ to be the contribution to $\Phi$ from components that intersect two parallel $a$-planes. That is, $\Phi^{(>a)}$ is defined as in \eqref{e:es_functional2} except the sum is restricted to components of $\{g\geq\ell\}$ (or $\{g=\ell\}$) which are contained in $D$ and intersect at least two $a$-planes which are parallel to one another. We define $\Phi^{(<a)}=\Phi-\Phi^{(>a)}$, which gives the contribution from components which do not intersect parallel $a$-planes.

\begin{figure}[h]
    \centering
    \begin{tikzpicture}[brace/.style={decorate,decoration={brace, amplitude=5pt}}, scale=0.5]

    \draw[black!20,dashed,step=3] (-2,-2) grid (14,14);
    \draw (0,0) rectangle (12,12);
    \node[right] at (12,11) {$\Lambda_R$};
    \node[above right] at (6,6) {$0$};
    \fill (6,6) circle (3pt);
    
    \draw[brace] (-0.1,9) --(-0.1,12) node[midway, left=6pt] {$a$};
    \draw[brace] (0,12.1) --(3,12.1) node[midway, above=6pt] {$a$};

    \draw [thick, pattern = north west lines, pattern color = gray] plot [smooth cycle, tension = 0.5] coordinates {(9.5,8) (10.5,8.5) (11,8) (10,7)};
    \draw [thick, pattern = north west lines, pattern color = gray] plot [smooth cycle, tension = 0.5] coordinates {(9,6.5) (8.5,6) (9,5.5) (9.5,6)};
    \draw [thick, pattern = north west lines, pattern color = gray] plot [smooth cycle, tension = 0.5] coordinates {(7.5,4.5)(7,4)(7,3)(7,2)(8,1)(9,0.5)(10,1.25)(9,2)(8,2.5)};
    \node[right] at (12.5,5) {$\Phi^{(<a)}$};
    \draw (12.5,5) -- (10.5, 7.2);
    \draw (12.5,5) -- (9.4,5.6);
    \draw (12.5,5) -- (8.2,2.8);

    \draw [thick, pattern = north west lines, pattern color = gray] plot [smooth cycle, tension = 0.5] coordinates {(1,10) (2,10.5) (6,10.5) (6.5, 10) (6,9.5) (5.5,9.2) (5.5,8) (5, 7.5) (4.5, 8) (4.5,9.5) (1.5,9.5)};
    \draw [thick, pattern = north west lines, pattern color = gray] plot [smooth cycle, tension = 0.5] coordinates {(1,7)(1,5)(2,4)(3.5,3.5)(4.5,2.5)(5,3.5)(3,4.5)(2,5)};
    \node[left] at (-0.2,7) {$\Phi^{(>a)}$};
    \draw (-0.2,7) -- (2,9.2);
    \draw (-0.2,7) -- (0.8,6.5);
    \end{tikzpicture}

    \caption{The box $\Lambda_R$ is subdivided by $a$-planes (denoted by dashed lines). The shaded regions represent excursion set components, which are labelled according to whether they contribute to $\Phi^{(>a)}$ or $\Phi^{(<a)}$.}
    \label{fig:a-planes}
\end{figure}
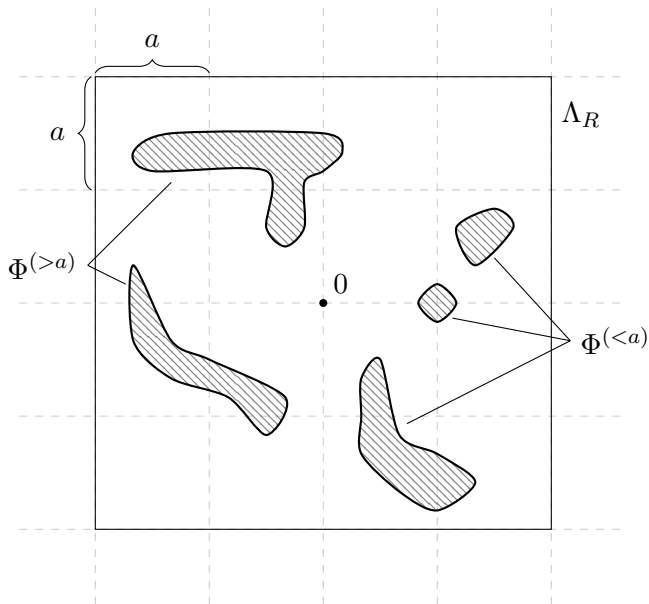

The key insight that we use to establish additivity, is that changes in $\Phi^{(>a)}$ (when perturbing the underlying field) imply the occurrence of a truncated arm event at scale $a$:

\begin{lemma}\label{l:arm_event}
    Let $a,R\geq 1$ such that $2a$ divides $R$, $x\in\Lambda_R$ and $g\in \mathrm{Morse}_x(\Lambda_R,\mathcal{F}^{(a)}_R,\ell)$. Let $\Phi$ be an excursion/level-set functional. If $d_x\Phi^{(>a)}(\Lambda_R,g)\neq 0$ then there exists a component of $\{g\geq\ell\}\backslash B(x,1)$ which intersects both $\partial B(x,1)$ and $\partial B(x,a/2)$ but does not intersect $\partial\Lambda_R$.
\end{lemma}
\begin{proof}
    Let $C$ denote the component of $\{g\geq\ell\}$ which contains $x$, then $C\setminus B(x,1)$ has finitely many components, which we denote $E_1,\dots,E_n, F_1,\dots,F_m$ where the $E_i$ all intersect $\partial\Lambda_R$ and the $F_i$ do not. (Of course, $n$ or $m$ may be zero.) Suppose that the arm event in the statement of the lemma fails, then $\cup_{i=1}^mF_i\subset B(x,a/2)$ which does not intersect parallel $a$-planes. Hence for any $\epsilon\in(0,1)$ the set
    \[
        \bigcup_{i=1}^mF_i\cup\bigcup_{j=1}^nE_j\setminus B(x,\epsilon)
    \]
    does not contribute to $\Phi^{(>a)}$ because each component either touches the boundary (i.e., intersects $\partial\Lambda_R$) or is `too small' (i.e., does not intersect parallel $a$-planes). Similarly    
    \[
        \bigcup_{i=1}^mF_i\cup\bigcup_{j=1}^nE_j\cup B(x,\epsilon)
    \]
    does not contribute to $\Phi^{(>a)}$ because it either intersects $\partial\Lambda_R$ (if $n>0$) or does not intersect parallel $a$-planes (if $n=0$). Combining these observations with the definition of topological derivative, we see that $d_x\Phi^{(>a)}(\Lambda_R,g)=0$, which completes the proof of the lemma.
\end{proof}

With this estimate, we can now prove the stated additivity result for excursion/level-set functionals:

\begin{proof}[Proof of Proposition~\ref{p:additivity} (Excursion/level-set functionals)]
    We define
    \[
        U_{R,+a}=\overline{\Lambda_R\setminus\bigcup_{x\in\chi_R}x+\Lambda_{r-2a}},
    \]
    which is simply the union of $U_R$ and all of the microscopic cubes in $\Lambda_R$ which touch $\partial U_R$. Recall that $\Phi$ is defined in \eqref{e:es_functional2} as the sum of $\varphi(E)$ where $E$ is taken to be a component of $\{f\geq\ell\}$ (or $\{f=\ell\}$) contained in $\Lambda_R$. We can partition such components into the following categories: 1) components contained in $x+\Lambda_r$ for some $x\in\chi_R$, 2) components contained in $\Lambda_R$ which intersect $U_R$ and also intersect two parallel $a$-planes, 3) components contained in $\Lambda_R$ which intersect $ U_R$ but do not intersect two parallel $a$-planes. Observe that components in the final case must be contained in $U_{R,+a}$. We express this partition as
    \begin{equation}\label{e:additivity1}
        \Phi(\Lambda_R)=\sum_{x\in\chi_R}\Phi(x+\Lambda_r) + \underbrace{\Phi^{(>a)}(\Lambda_R, U_R)}_{=:A_R} + \underbrace{\Phi^{(<a)}(U_{R,+a}, U_R)}_{=:B_R} 
    \end{equation}
    where we recall that $\Phi^{(>a)}$ and $\Phi^{(<a)}$ denote contributions from components which do or do not intersect parallel $a$-planes respectively and $\Phi(D,F)$ denotes the contribution from components contained in $D$ which intersect $F$. To prove the proposition, for excursion/level-set functionals, it is enough to bound the variance of $A_R$ and $B_R$.

    We note that the expressions defining $A_R$ and $B_R$ are Lipschitz topological functionals: to view them as topological functionals, we must define them to be identically zero for regular affine stratified sets which are not of the form $(\Lambda_R,\mathcal{F}_R^{(a)})$, while the Lipschitz property follows from the proof of Lemma~\ref{l:exc_lipschitz}. Hence by quasi-association (Corollary~\ref{c:quasi_association}) and Assumption~\aref{a:cov_decay_int}
    \begin{equation}\label{e:additivity2}
        \Var[B_R]\leq C\int_{(U_{R,+a}^{+2})^{2}}\widetilde{K}(x-y)\;dxdy\leq C\lvert (U_{R,+a})^{+2}\rvert\int_{\R^d}\widetilde{K}(x)\;dx=O(ar^{-1}R^d).
    \end{equation}
    Turning to $A_R$, by the covariance formula (Theorem~\ref{t:cov_formula})
    \begin{equation}\label{e:additivity5}
        \Var[A_R]=\int_{(\mathcal{F}^{(a)}_R)^2}K(x-y)\int_0^1\mu^t_{x,y}\big(d_x\Phi^{(>a)}(\Lambda_R, U_R,f)d_y^t\Phi^{(>a)}(\Lambda_R,U_R,f^t)\big)\;dtdxdy.
    \end{equation}
    We recall that $\mathcal{F}^{(a)}_R$ can be thought of as the stratification induced by dividing $\Lambda_R$ into $(R/a)^d$ cubes of side length $a$. In particular for each $k\in\{0,1,\dots, d\}$ this stratification contains at most $c_d(R/a)^d$ strata of dimension $k$, each of which has $k$-dimensional measure at most $c_da^k$. Using the fact that $\Phi^{(>a)}$ has the same Lipschitz norm as $\Phi$ and the pivotal measure bound from Proposition~\ref{p:intensity_integrable}, by \eqref{e:additivity5} we see that
    \begin{equation}\label{e:additivity6}
    \begin{aligned}
        \Var[A_R]=&\int_{\Lambda_{R}^2}K(x-y)\int_0^1\mu^t_{x,y}\big(d_x\Phi^{(>a)}(\Lambda_R, U_R,f)d_y^t\Phi^{(>a)}(\Lambda_R,U_R,f^t)\big)\;dtdxdy+O(R^d/a)
    \end{aligned}
    \end{equation}
    where the implicit constant may depend on $K$ and $\Phi$. By Lemma~\ref{l:arm_event} the integral term above is bounded in absolute value by a constant times
    \begin{equation}\label{e:additivity7}
    \begin{aligned}
        \int_{\Lambda_{R}^2}\lvert K(x-y)\rvert&\int_0^1\mu^t_{x,y}(\ind_{\mathrm{TArm}(x,a/2)})\;dtdxdy\leq R^d\int_{\R^d}\lvert K(x)\rvert\int_0^1\mu^t_{x,0}(\ind_{\mathrm{TArm}_x(a/2)})\;dtdx
    \end{aligned}
    \end{equation}
    where $\mathrm{TArm}_x(a/2)$ denotes the `truncated arm' event that there exists a bounded component of $\{f\geq\ell\}\setminus B(x,1)$ joining $\partial B(x,1)$ to $\partial B(x,a/2)$. As $a\to\infty$, $\ind_{\mathrm{TArm}_x(a/2)}\to 0$ $\mu^t_{x,0}$-almost surely by definition so by dominated convergence
    \[
        \mu^t_{x,0}(\ind_{\mathrm{TArm}_x(a/2)})\to 0\qquad\text{as }R\to\infty.
    \]
    Applying the dominated convergence theorem once more, along with Proposition~\ref{p:intensity_integrable}, shows that \eqref{e:additivity7} is $o(R^d)$ as $R\to\infty$. Combined with \eqref{e:additivity1} and \eqref{e:additivity2}, this completes the proof for bounded excursion/level-set functionals.
\end{proof}

\subsection{Stabilisation}\label{ss:stabilisation} We now show that our topological functionals stabilise, in the sense that the effect of a small perturbation at a critical point of the field causes the same change in the value of the functional for all sufficiently large domains. This property can be formalised for excursion/level-set functionals as follows:

\begin{lemma}\label{l:stabilisation}
    Let $g:\R^d\to\R$ be $C^2$ with a non-degenerate critical point $x$ at level $\ell$ and no other critical points at level $\ell$ and let $\Phi$ be an excursion/level-set functional. There exists $R_0\geq 1$ and $d_x\Phi_\infty(g)\in\R$ such that \[
        d_x\Phi(\Lambda_R,g)=d_x\Phi_\infty(g)
    \]
    for all $R\geq R_0$ such that $g\in\mathrm{Morse}_x(\Lambda_R,\mathcal{F}_R,\ell)$. Moreover, if $d_x\Phi(\Lambda_R,g)\neq d_x\Phi_\infty(g)$ for some $R\geq 1$ then there exists a bounded component of $\{g\geq\ell\}\backslash B(x,1)$ which intersects $\partial B(x,1)$ and $\partial\Lambda_R$.
\end{lemma}
\begin{proof}
    Let $C$ denote the component of $\{g\geq\ell\}\cup B(x,1)$ which contains $x$, then $C\setminus B(x,1)$ has finitely many components, which we denote $E_1,\dots,E_n, F_1,\dots,F_m$ where the $E_i$ are all bounded and the $F_i$ are unbounded. We choose $R_0$ such that $\cup_{i=1}^nE_i\subset\Lambda_{R_0-1}$, in particular if $n=0$ we may choose $R_0=2$. Then for any $C^2$ function $h$ supported on $B(x,1/2)$ and $R\geq R_0$ such that $g\in\mathrm{Morse}_x(\Lambda_R,\mathcal{F}_R,\ell)$,
    \[
        \Phi(\Lambda_R,g+h,\ell)=\sum_{\substack{E\in\mathrm{Comp}(\{g+h\geq\ell\},\Lambda_R)\\E\subset C}}\varphi([E])+\sum_{\substack{E\in\mathrm{Comp}(\{g+h\geq\ell\},\Lambda_R)\\ E\cap C=\emptyset}}\varphi([E]).
    \]
    The first sum here is independent of $R\geq R_0$, since any bounded component of $\{g+h\geq\ell\}$ contained in $C$ must be contained in $\Lambda_{R_0}$. The second sum is independent of $h$, since $\{g\geq\ell\}$ and $\{g+h\geq\ell\}$ coincide on the complement of $B(x,1)$. Combining these observations, we see that the change in $\Phi(\Lambda_R,g+h,\ell)$ as we vary $h$ is independent of $R\geq R_0$. By definition of the topological derivative in \eqref{e:top_derivative_functional} this proves the statement of the lemma.
\end{proof}

It follows from the Poincar\'e-Hopf theorem that the Euler characteristic stabilises immediately:

\begin{lemma}\label{l:ec_stabilisation}
    Let $\Phi$ be the Euler characteristic and $(D_1,\mathcal{F}_1)$ and $(D_2,\mathcal{F}_2)$ be regular affine stratified sets. If $x$ is contained in $d$-dimenisional strata of both $D_1$ and $D_2$ and $g\in\mathrm{Morse}_x(D_1,\mathcal{F}_1,\ell)\cap\mathrm{Morse}_x(D_2,\mathcal{F}_2,\ell)$ then
    \[
        d_x\Phi(D_1,g)=d_x\Phi(D_2,g),
    \]
    and we denote this common value by $d_x\Phi_\infty(g)$.
\end{lemma}
\begin{proof}
    As argued in the proof of Lemma~\ref{l:ec}, the topological derivative $d_x\Phi(D_i,g)$ is equal to $+1$, $0$ or $-1$ depending on how $x$ contributes to the decomposition of the Euler characteristic in \eqref{e:ec_decomp}. The contribution associated with a critical point in a $d$-dimensional stratum is determined by its index (i.e., the number of negative eigenvalues of the Hessian at the critical point). Hence this sign is the same on both domains $D_i$ and therefore so is the topological derivative.
\end{proof}

We will make use of these stabilisation results when applying the covariance formula on a large domain, to argue that the integrand is (asymptotically) independent of the size of the domain. The precise statement we need is the following:

\begin{proposition}[Stabilisation]\label{p:stabilisation}
    Let $f$ satisfy Assumptions~\aref{a:basic} and~\aref{a:cov_decay_int} and let $\Phi$ be a bounded excursion/\allowbreak level-set functional, then as $R\to\infty$,
    \[
        \int_{\Lambda_{R}^2}K(x-y)\int_0^1\mu^t_{x,y}(d_x\Phi_Rd_y^t\Phi_R-d_x\Phi_\infty d_y^t\Phi_\infty)\;dtdxdy=o(R^d).
    \]
    If instead, $\Phi$ is the Euler characteristic, then the above expression is identically zero.
\end{proposition}
\begin{proof}
    The statement for the Euler characteristic is immediate from Lemma~\ref{l:ec_stabilisation}. Assuming that $\Phi$ is a bounded excursion/level-set functional, for each $R$, we choose $r=r_R$ such that $r_R\to\infty$ and $r_R/R\to 0$ as $R\to\infty$. Using integrability of $K$ and the bound on pivotal measures in Proposition~\ref{p:intensity_integrable}, we have
    \begin{equation}\label{e:var_limit2}
    \begin{aligned}
        \Big\lvert\int_{\substack{x,y\in\Lambda_R\\d(x,\partial\Lambda_R)<r\text{ or }d(y,\partial\Lambda_R)<r}}K(x-y)\int_0^1\mu^t_{x,y}(d_x\Phi_Rd_y^t\Phi_R-d_x\Phi_\infty d_y^t\Phi_\infty)&\;dtdxdy\Big\rvert\\
        &\leq C\|\Phi\|_\mathrm{Lip}^2 rR^{d-1}=o(R^d),
    \end{aligned}
    \end{equation}
    where $d(x,\partial\Lambda_R)$ denotes the Euclidean distance between $x$ and $\partial\Lambda_R$. By Lemma~\ref{l:stabilisation}
    \begin{equation}\label{e:var_limit3}
    \begin{aligned}
        \Big\lvert\int_{\Lambda_{R-r}^2}K(x-y)&\int_0^1\mu^t_{x,y}(d_x\Phi_Rd_y^t\Phi_R-d_x\Phi_\infty d_y^t\Phi_\infty)\;dtdxdy\Big\rvert\\
        &\leq \|\Phi\|_\mathrm{Lip}^2\int_{\Lambda_{R-r}^2}\lvert K(x-y)\rvert\int_0^1\mu^t_{x,y}(\ind_{\mathrm{TArm}_x(r)}+\ind_{\mathrm{TArm}_y^t(r)})\;dtdxdy\\
        &\leq \|\Phi\|_\mathrm{Lip}^2R^d\int_{\R^d}\lvert K(x)\rvert\int_0^1\mu^t_{x,0}(\ind_{\mathrm{TArm}_x(r)}+\ind_{\mathrm{TArm}_0^t(r)})\;dtdx
    \end{aligned}
    \end{equation}
    where we recall that $\mathrm{TArm}_x(r)$ denotes the `truncated arm' event that there exists a bounded component of $\{f\geq\ell\}\setminus B(x,1)$ joining $\partial B(x,1)$ to $\partial B(x,r)$ and $\mathrm{TArm}^t_y(r)$ denotes the corresponding event for $f^t$ at $y$. As $r\to\infty$, $\ind_{\mathrm{TArm}_x(r)}\to 0$ $\mu^t_{x,0}$-almost surely by definition (as does $\ind_{\mathrm{TArm}_0^t(r)}$) so by dominated convergence
    \[
        \mu^t_{x,0}(\ind_{\mathrm{TArm}_x(r)}+\ind_{\mathrm{TArm}_0^t(r)})\to 0\qquad\text{as }R\to\infty.
    \]
    Applying the dominated convergence theorem once more, along with Proposition~\ref{p:intensity_integrable}, shows that \eqref{e:var_limit3} is $o(R^d)$ as $R\to\infty$.
\end{proof}

\subsection{Quantitative bounds}
In this subsection, we refine the additivity and stabilisation results for bounded excursion/level-set functionals under additional arm-decay assumptions. These refined estimates are only required for the higher moment bounds in Theorem~\ref{t:higher_moments} and the quantitative central limit theorem (Theorem~\ref{t:qclt}). Throughout this subsection, $\Phi$ denotes a bounded excursion/level-set functional evaluated at a level $\ell$. The estimates we prove are the following:

\begin{proposition}[Quantitative additivity]\label{p:q_additivity}
    Let $f$ satisfy Assumptions~\aref{a:basic} and \aref{a:cov_decay_pol} and suppose that $(a_R)_{R\geq 1}$ and $(r_R)_{R\geq 1}$ satisfy \eqref{e:scales}. Suppose also that one of the following holds:
    \begin{enumerate}
        \item Assumption~\aref{a:arm_decay_critical} holds and $a_R$ grows at least polynomially in $R$ (i.e., $a_R/R^\delta\to\infty$ for some $\delta>0$),
        \item Assumption~\aref{a:arm_decay} holds, or
        \item Assumptions~\aref{a:arm_decay} and \aref{a:cov_decay_exp} hold,
    \end{enumerate}
    then there exists $\epsilon>0$ such that as $R\to\infty$
    \[
        R^{-d}\Var\Big[\Phi(\Lambda_R)-\sum_{x\in\chi_R}\Phi(x+\Lambda_r)\Big]=\begin{cases}
            O(R^{-\epsilon}) &\text{in case (1)}\\
            O\Big(ar^{-1}+a^{-\frac{\beta}{2d+2}}\Big) &\text{in case (2),}\\
            O\Big(ar^{-1}+\exp(-\epsilon a)\Big) &\text{in case (3).}
        \end{cases}
    \]
\end{proposition}

\begin{proposition}[Quantitative stabilisation]\label{p:q_stabilisation}
    Let $f$ satisfy Assumptions~\aref{a:basic} and~\aref{a:cov_decay_pol}. If Assumption~\aref{a:arm_decay} holds, then as $R\to\infty$
    \[
    R^{-d}\int_{\Lambda_{R}^2}K(x-y)\int_0^1\mu^t_{x,y}(d_x\Phi_Rd_y^t\Phi_R-d_x\Phi_\infty d_y^t\Phi_\infty)\;dtdxdy=O\Big(R^{-\min\{1,\beta/(2d+2)\}}\Big).
    \]
    If Assumption~\aref{a:arm_decay} is replaced by Assumption~\aref{a:arm_decay_critical} then the left-hand side is $O(R^{-\epsilon})$ for some $\epsilon>0$.
\end{proposition}

To prove these results, we need to strengthen the control over the measure of truncated arm events (under $\mu^t_{x,y}$) used in Sections~\ref{ss:additivity} and~\ref{ss:stabilisation}. We now state the key estimate with which we do this.

Given $k\in\N$, $t\in[0,1)$ and $x\in\R^d$ we define
\begin{equation*}
    I_k(t,x):=\begin{cases}
        \overline{c} &\text{if }\lvert x\rvert\in[1,\infty),\\
        \overline{c}\lvert x\rvert^{2k}t &\text{if }\lvert x\rvert\in[\sqrt{t},1),\\
        \overline{c}t^{k+1} &\text{if }\lvert x\rvert\in[0,\sqrt{t}),
    \end{cases}
\end{equation*}
where $\overline{c}>0$ is a constant depending only on the distribution of $f$, which will be specified later (in Lemma~\ref{l:dc_bound}).

\begin{proposition}[Conditional truncated arm decay]\label{p:arm_conditional}
    Let $f$ satisfy Assumptions~\aref{a:basic} and \aref{a:cov_decay_pol}. Suppose that either
    \begin{enumerate}
        \item Assumption~\aref{a:arm_decay_critical} holds,
        \item Assumption~\aref{a:arm_decay} holds, or
        \item Assumptions~\aref{a:arm_decay} and \aref{a:cov_decay_exp} hold.
    \end{enumerate}
    then there exists $C,c,\epsilon>0$ such that for any regular affine stratified sets $(D_1,\mathcal{F}_1)$, $(D_2,\mathcal{F}_2)$, $x_1\in D_1$, $x_2\in D_2$, $t\in[0,1)$ and $R\geq 1$,
    \begin{equation}\label{e:arm_conditional}
        \mu^t_{x,y}(\ind_{\mathrm{TArm}_x(R)})\leq \hat{I}_{d_{xy},R}(1-t,x-y)
    \end{equation}
    where
    \[
         \hat{I}_{k,R}(s,z):=C I_{k}(s,z)^{-1/2}\Big(h_1(R)+\exp\big(-cI_{k}(s,z)^2/h_2(R)^2\big)\Big),
    \]
    $d_{xy}$ is the minimum of the dimensions of the strata containing $x$ and $y$, and
    \[
        (h_1(R),h_2(R))=\begin{cases}
            (R^{-\epsilon},R^{-\epsilon}) &\text{in case (1),}\\
            (e^{-cR},R^{-\beta}) &\text{in case (2),}\\
            (e^{-cR}, e^{-cR}) &\text{in case (3).}
        \end{cases}
    \]
\end{proposition}

The first term in the definition of $\hat{I}_{k,R}$ corresponds to the Gaussian density in the definition of $\mu^t_{x,y}$ (which diverges as $x\to y$ and $t\to 1$) while the following terms inside parentheses are related to arm decay when conditioning on critical points at $x$ and $y$.

With this estimate in hand, the quantitative additivity and stabilisation results reduce to evaluating some elementary (but slightly messy) integrals:

\begin{proof}[Proof of Propositions~\ref{p:q_additivity} and~\ref{p:q_stabilisation} (assuming Proposition~\ref{p:arm_conditional})]
    We first claim that there exists $C>0$ such that for all $R\geq 20$, $u\geq 1$
    \begin{equation}\label{e:q_additivity1}
        \sup_{x\in\Lambda_R}\int_{\mathcal{F}_R^{(a)}}\lvert K(x-y)\rvert\int_0^1\mu^t_{x,y}(\ind_{\mathrm{TArm}_x(u)})\;dtdy\leq C\Big(h_1(u)+(h_2(u))^{\frac{1}{2d+2}}\Big).
    \end{equation}
    To verify the claim, we invoke Proposition~\ref{p:arm_conditional} and control the integral of the right-hand side of \eqref{e:arm_conditional} as follows.

    We fix $x\in\Lambda_R$ and first note that, since $I_k(t,z)$ is constant for $\lvert z\rvert\geq 1$, by \eqref{e:arm_conditional} and integrability of $\widetilde{K}$
    \[
        \int_{\mathcal{F}_R^{(a)}\setminus B(x,1)}\lvert K(x-y)\rvert\int_0^1\mu^t_{x,y}(\ind_{\mathrm{TArm}_x(u)})\;dtdy\leq C h_1(u)\int_{\mathcal{F}_R^{(a)}}\lvert K(x-y)\rvert\;dy\leq C h_1(u),
    \]
    where, as before, the value of $C>0$ may change from line to line. Given a stratum $F\in\mathcal{F}_R^{(a)}$, we let $x_F$ denote the projection of $x$ onto the minimal affine subspace containing $F$. Since $I_k(s,z)$ is non-decreasing in $\lvert z\rvert$, for any $s,y,u$ we have $\hat{I}_{k,u}(s,x-y)\leq\hat{I}_{k,u}(s,x_F-y)$. Then by the trivial bound $\lvert K\rvert\leq 1$ and the substitution $y\mapsto y+x_F$,
    \begin{align*}
        \int_{\mathcal{F}_R^{(a)}\cap B(x,1)}\lvert K(x-y)\rvert\int_0^1\mu^t(\ind_{\mathrm{TArm}_x(u)})\;dtdy&\leq\sum_{F\in\mathcal{F}_R^{(a)}}\int_{F\cap B(x_F,1)}\int_0^1\hat{I}_{d_F,u}(1-t,x_F-y)\;dtdv_F(y)\\
        &\leq \sum_{F\in\mathcal{F}^\prime}\int_{F\cap B(0,1)}\int_0^1\hat{I}_{d_F,u}(s,y)\;dsdv_F(y)
    \end{align*}
    where $\mathcal{F}^\prime$ denotes the set of regular affine subspaces of $\R^d$ passing through the origin. Since the number of such subspaces depends only on $d$, we need only verify our desired upper bound for each $F$ individually.

    From the definition of $I_{k}(t,x)$
    \[
        \int_{F\cap B(0,1)}\int_0^1I_{d_F}(s,y)^{-1/2}\;dsdv_F(y)\leq C\int_{F\cap B(0,1)}\int_0^{\lvert y\rvert^2}\lvert y\rvert^{-d_F}s^{-1/2}\;ds+\int_{\lvert y\rvert^2}^1s^{-\frac{d_F+1}{2}}\;ds\;dv_{F}(y).
    \]
    By evaluating the inner integrals explicitly, we see that this expression is finite (and independent of $R$ and $u$).

    Denoting $\tilde{I}_{k,u}(s,y)=I_k(s,y)^{-1/2}\exp(-cI_k(s,y)^{2}/h_2(u)^2)$ and using the substitution $s\mapsto h_2(u)^{-\frac{1}{d_F+1}}s$,
    \begin{align*}
        \int_{F\cap B(0,1)}\int_{\lvert y\rvert^2}^1\tilde{I}_{d_F,u}(s,y)\;dsdy&\leq C\int_0^1\int_{F\cap B(0,\sqrt{s})}\;dy s^{-(d_F+1)/2}\exp(-c^\prime h_2(u)^{-2}s^{2d_F+2})\;ds\\
        &\leq C h_2(u)^{1/(2d_F+2)}\int_0^{\infty}s^{-1/2}\exp(-c^\prime s^{2d_F+2})\;ds\leq C h_2(u)^{\frac{1}{2d_F+2}}.
    \end{align*}
    By integrating over $F$ in ($d_F$-dimensional) spherical coordinates, we have
    \[
        \int_{F\cap B\big(0,h_2(u)^{\frac{1}{2d_F+2}}\big)}\int_0^{\lvert y\rvert^2}\tilde{I}_{d_F,u}(s,y)\;dsdy\leq C\int_{F\cap B\big(0,h_2(u)^{\frac{1}{2d_F+2}}\big)}\lvert y\rvert^{-d_F}\int_0^{\lvert y\rvert^2} s^{-1/2}\;dsdy\leq C h_2(u)^{\frac{1}{2d_F+2}}.
    \]
    Finally, letting $A(r_1,r_2):=B(0,r_2)\setminus B(0,r_1)$ and using the substitution $s\mapsto h_2(u)^{-1}\lvert y\rvert^{2d_F}s$,
    \begin{align*}
        \int_{F\cap A\big(h_2(u)^{\frac{1}{2d_F+2}},1\big)}&\int_0^{\lvert y\rvert^2}\tilde{I}_{d_F,u}(s,y)\;dsdy\\
        &=C\int_{F\cap A\big(h_2(u)^{\frac{1}{2d+2}},1\big)}\lvert y\rvert^{-d_F}\int_0^{\lvert y\rvert^2} s^{-1/2}\exp\big(-c^\prime h_2(u)^{-2}\lvert y\rvert^{4d_F}s^2\big)\;dsdy\\
        &\leq C h_2(u)^{\frac{1}{2}} \int_{F\cap A\big(h_2(u)^{\frac{1}{2d+2}},1\big)}\lvert y\rvert^{-2d_F}\;dy\int_0^\infty s^{-1/2}\exp(-c^\prime s^2)\;ds\leq Ch_2(u)^{\frac{1}{2d_F+2}}
    \end{align*}
    where the final bound follows from evaluating the integral over $y$ in spherical coordinates. Combining the last six equations verifies \eqref{e:q_additivity1}.

    We next observe that for a given $t\in[0,1)$, $(f,f^t)$ has the same distribution as $(f^t,f)$. Then by stationarity, we conclude that for any $w,z\in\R^d$ (viewed as elements of $d$-dimensional strata)
    \begin{equation}\label{e:q_additivity2}
        \mu^t_{w,z}(\ind_{\mathrm{TArm}^t_z(R)})=\mu^t_{z,w}(\ind_{\mathrm{TArm}_z(R)})=\mu^t_{0,w-z}(\ind_{\mathrm{TArm}_0(R)}).
    \end{equation}
    We may now prove the additivity and stabilisation estimates:
    
    \underline{Additivity:} From the proof of Proposition~\ref{p:additivity} (specifically \eqref{e:additivity1} and \eqref{e:additivity2}) we have
    \[
        \Var\Big[\Phi(\Lambda_R)-\sum_{x\in\chi_R}\Phi(x+\Lambda_r)\Big]\leq C\Big(\frac{a}{r}R^d+\Var\big[\Phi^{(>a)}(\Lambda_R, U_R)\big]\Big)
    \]
    By the covariance formula (Theorem~\ref{t:cov_formula}), Lemma~\ref{l:arm_event} and \eqref{e:q_additivity1}
    \begin{equation}\label{e:q_additivity}
    \begin{aligned}
        &\Var\big[\Phi^{(>a)}(\Lambda_R, U_R)\big]\\
        &\qquad=\int_{\big(\mathcal{F}_R^{(a)}\big)^2}K(x-y)\int_0^1\mu^t_{x,y}\big(d_x\Phi^{(>a)}(\Lambda_R,U_R)d_y^t\Phi^{(>a)}(\Lambda_R,U_R)\big)\;dtdxdy\\
        &\qquad\leq C\int_{\big(\mathcal{F}_R^{(a)}\big)^2}\lvert K(x-y)\rvert\int_0^1\mu^t_{x,y}(\ind_{\mathrm{TArm}_x(a/2)})\;dtdxdy\\
        &\qquad\leq C\sum_{F\in\mathcal{F}_R^{(a)}}v_{F}(F)\Big(h_1(a/2)+h_2(a/2)^{\frac{1}{2d+2}}\Big)\leq CR^d\Big(h_1(a/2)+h_2(a/2)^{\frac{1}{2d+2}}\Big)
    \end{aligned}
    \end{equation}
    where $h_1$ and $h_2$ are as specified in the conclusion of Proposition~\ref{p:arm_conditional}. This completes the proof of Proposition~\ref{p:q_additivity}.

    \underline{Stabilisation:} By Lemma~\ref{l:stabilisation} and \eqref{e:q_additivity2}
    \begin{align*}
        &\int_{\Lambda_{R}^2}K(x-y)\int_0^1\mu^t_{x,y}(d_x\Phi_Rd_y^t\Phi_R-d_x\Phi_\infty d_y^t\Phi_\infty)\;dtdxdy\\
        &\quad\leq\; C\int_{\Lambda_{R-1}^2}\lvert K(x-y)\rvert\int_0^1\mu^t_{x,y}(\ind_{\mathrm{TArm}_x(R-\|x\|_\infty)})\;dtdxdy+C\int_{\Lambda_{R}\setminus\Lambda_{R-1}\times\R^d}\lvert K(x-y)\rvert\int_0^1\|\mu^t_{x,y}\|\;dtdxdy,
    \end{align*}
    where $\|\cdot\|_\infty$ denotes the sup-norm on $\R^d$. The final expression here is $O(R^{d-1})$ as $R\to\infty$, by integrability of $K$ and the bounds on the pivotal measure (Proposition~\ref{p:intensity_integrable}). By \eqref{e:q_additivity1}, setting $\alpha=\beta/(2d+2)$ if Assumption~\aref{a:arm_decay} holds and $\alpha=\epsilon$ (as in the statement of Proposition~\ref{p:arm_conditional}) if Assumption~\aref{a:arm_decay_critical} holds, the penultimate expression is at most
    \begin{align*}
        C\int_{\Lambda_{R-1}}(R-\|x\|_\infty)^{-\alpha}\;dx&\leq C\int_0^{R-1}\int_{[0,x_1]^{d-1}}(R-x_1)^{-\alpha}\;dx_1\dots dx_d=C\int_0^{R-1}(R-x)^{-\alpha}x^{d-1}\;dx\\
        &\leq C\Big(R^{-\alpha}\int_0^{R/2}x^{d-1}\;dx+R^{d-1}\int_{R/2}^{R-1}(R-x)^{-\alpha}\;dx\Big)\leq CR^{d-\min\{1,\alpha\}}.
    \end{align*}
    Combining these two equations completes the proof of Proposition~\ref{p:q_stabilisation}.
\end{proof}

\begin{remark}\label{r:trunc_arm}
    In the above proof, we used the truncated arm decay bounds (for $\mu^t_{x,y}$) from Proposition~\ref{p:arm_conditional} to establish the additivity and stabilisation estimates needed for our quantitative limit theorems. In the remainder of this subsection, we use Assumption~\aref{a:arm_decay} to deduce the estimates in Proposition~\ref{p:arm_conditional}. As mentioned in the introduction, Assumption~\aref{a:arm_decay} does not hold at all non-critical levels for $d\geq 3$, rather it holds for $\lvert\ell\rvert>\ell_c$ which comprises the subcritical and strongly supercritical regimes. To extend our results to all non-critical levels, one could instead try to directly prove the bounds in Proposition~\ref{p:arm_conditional} by studying the conditioned field $(f|A^t_{x,y})$ (i.e., the field under the measure $\mu^t_{x,y}$). This might naturally be expected to exhibit exponential truncated arm decay at all non-critical levels, although no such result has currently been proven (even for the unconditioned field $f$).
\end{remark}

We now turn to the proof of Proposition~\ref{p:arm_conditional}. We need to bound the measure of arm events under $\mu^t_{x,y}$, which essentially reduces to controlling arm events for the field $f$ conditioned on
\[
    A^t_{x,y}=\big\{f(x)=f^t(y)=\ell,\nabla|_{F_x} f(x)=0, \nabla|_{F_y} f^t(y)=0\big\}.
\]
Our strategy is to couple the conditioned and unconditioned fields, which intuitively should be similar far away from $x$ and $y$, and show that an arm event for the conditioned field implies a slightly different arm event for the unconditioned field at a nearby level, with high probability. The precise nature of the arm event will depend on which assumption we invoke; the two cases of Assumption~\aref{a:arm_decay} correspond to the subcritical and (strongly) supercritical regimes, whilst Assumption~\aref{a:arm_decay_critical} is of most relevance in the critical regime (for planar fields).

Our next two lemmas help to establish arm events for the conditioned field in the supercritical and critical regimes respectively. For two sets $A,B\subset\R^d$, we say that $A$ surrounds $B$ if $B$ is contained in a bounded component of $\R^d\setminus A$. We recall that for a function $g$ and level $\ell$, $\{g\geq\ell\}_{<\infty}$ denotes the union of all bounded components of $\{g\geq\ell\}$. Extending our previous notation, let $\mathrm{Morse}_x(\R^d,\ell)$ denote the set of $C^2$-smooth functions $g:\R^d\to\R$ with a non-degenerate critical point $x$ at level $\ell$ and no other critical points at level $\ell$.

\begin{lemma}\label{l:arm_event_supercrit}
    Let $\ell\in\R$, $x,y\in\R^d$, $g\in\mathrm{Morse}_x(\R^d,\ell)$ and $R\geq 10$. If $B(x,1)\overset{\{g\geq\ell\}_{<\infty}}{\longleftrightarrow}\partial B(x,R)$, then there exists a connected component of $\{g\leq\ell\}\setminus(B(x,R/20)\cup B(y,R/20))$ which has diameter at least $R/20$ and intersects or surrounds either $\partial B(x,R/20)$ or $\partial B(y,R/20)$.
\end{lemma}

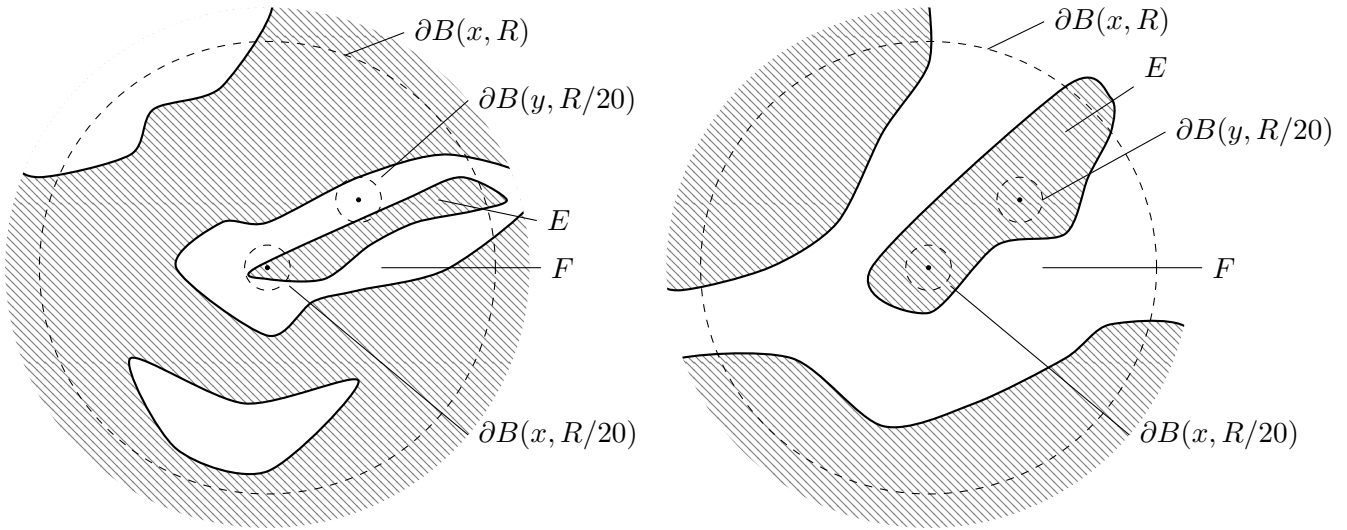
\begin{figure}[h]
    \centering
    \begin{tikzpicture}[scale=0.15]


    \begin{scope}
    \fill [pattern = north west lines, pattern color = gray] (0,0) circle (23);
    \clip (0,0) circle (23);
    \draw [thick, fill = white] plot [smooth cycle, tension = 0.5] coordinates {(8,8) (0,4) (-4,4)(-8,0) (0,-6) (4,-3) (8,-2) (16,0) (24,7) (16,10)};
    \draw [thick, fill = white] plot [smooth cycle, tension = 0.5] coordinates {(0,24) (-4,16) (-10,14) (-12, 10)(-20,8) (-22,10) (-18,18)};
    \draw [thick, fill = white] plot [smooth cycle, tension = 0.5] coordinates {(-12,-8) (-8,-16)(0,-18)(8,-10)(-2,-12)};
    \draw [thick, pattern = north west lines, pattern color = gray] plot [smooth cycle, tension = 0.5] coordinates {(-1,0) (0,-1) (5,-1) (9,2) (13,4) (18, 5) (21,6) (17,8) (12,6)};
    \end{scope}

    \draw[dashed] (0,0) circle (20);
    \node[right] at (60:24) {$\partial B(x,R)$};
    \draw (60:24)--(70:20);
    \draw[dashed] (0,0) circle (2);
    \node[right] at (-40:23) {$\partial B(x,R/20)$};
    \draw (-40:23) -- (-40:2.5);
    \draw[fill] (0,0) circle (5pt);
    
    \draw[dashed] (8,6) circle (2);
    \node[right] at (40:23) {$\partial B(y,R/20)$};
    \draw (40:23) -- (10,8);
    \draw[fill] (8,6) circle (5pt);

    \node[right] at (10:24) {$E$};
    \draw (10:24)--(15,6);
    \node[right] at (0:24) {$F$};
    \draw (0:24)--(10,0);

    \begin{scope}[shift = {(58,0)}]
    
    \begin{scope}
    \clip (0,0) circle (23);
    \draw [thick, pattern = north west lines, pattern color = gray] plot [smooth cycle, tension = 0.5] coordinates {(-5,0)(0,-4)(6,2) (12,3) (14,8) (16,12) (16,15) (12,16)};
    \draw [thick, pattern = north west lines, pattern color = gray] plot [smooth cycle, tension = 0.5] coordinates {(0,24) (0,18) (-4,12)(-8,4) (-14,0) (-22,-2)(-24,0) (160:24) (120:24) (90:24)};
    \draw [thick, pattern = north west lines, pattern color = gray] plot [smooth cycle, tension = 0.5] coordinates {(200:24) (-12,-8)(-4,-14) (4,-12) (12,-8) (16,-5) (23,-6) (-50:24)(-80:24)(-120:24)(-150:24)};
    \end{scope}

    \draw[dashed] (0,0) circle (20);
    \node[right] at (65:24) {$\partial B(x,R)$};
    \draw (65:24)--(75:20);
    \draw[dashed] (0,0) circle (2);
    \node[right] at (-40:23) {$\partial B(x,R/20)$};
    \draw (-40:23) -- (-40:2.5);
    \draw[fill] (0,0) circle (5pt);
    
    \draw[dashed] (8,6) circle (2);
    \node[right] at (30:24) {$\partial B(y,R/20)$};
    \draw (30:24) -- (10,6);
    \draw[fill] (8,6) circle (5pt);

    \node[above right] at (18,16) {$E$};
    \draw (18,16)--(12,12);
    \node[right] at (0:24) {$F$};
    \draw (0:24)--(10,0);
    \end{scope}

\end{tikzpicture}
    \caption{Two possible configurations of arm events illustrating Lemma~\ref{l:arm_event_supercrit}. The shaded region represents the excursion set $\{g\geq\ell\}$.}
    \label{fig:arm_event}
\end{figure}

\begin{proof}
    Let $E$ denote a bounded component of $\{g\geq\ell\}$ which intersects $B(x,1)$ and has diameter at least $R/2-1$. Since $E$ is bounded, it must be surrounded by a component $F$ of $\{g<\ell\}$ (see Figure~\ref{fig:arm_event}). Hence $F$ must either intersect or surround $B(x,1)$ and must have diameter at least $R/2-1$. We may therefore choose $u\in F$ such that $\partial B(u,R/2-1)$ intersects $F$. Then for some $a_1,a_2>0$ we have
    \[
        B(x,R/20)\subset A(u,a_1,a_1+R/10)\qquad\text{and}\qquad B(y,R/20)\subset A(u,a_2,a_2+R/10)
    \]
    where $A(u,r_1,r_2)$ denotes the closed annulus centred at $u$ with inner radius $r_1$ and outer radius $r_2$. By considering the different possible relative magnitudes of $a_1$ and $a_2$, it follows easily that $F$ must contains a path of diameter at least $R/20$ which does not intersect these two annuli. Let $F^\prime$ denote the component of $F\setminus(B(x,R/20)\cup B(y,R/20))$ which contains this path. If $F$ does not intersect $B(x,R/20)\cup B(y,R/20)$, then $F^\prime=F$ which must therefore surround $B(x,R/20)$. Otherwise $F$ must intersect $\partial B(x,R/20)\cup\partial B(y,R/20)$. In either case $F$ satisfies the requirements of the statement.
\end{proof}

\begin{lemma}\label{l:arm_critical}
    Let $f$ satisfy Assumptions~\aref{a:basic}, \aref{a:cov_decay_pol} and \aref{a:arm_decay_critical} for some $\epsilon>0$ at level $\ell$, then there exists $C>0$ such that for any $\delta>0$ and $R>r\geq 1$
    \[
        \P\Big(B(0,r)\overset{\{f\geq\ell-\delta\}}{\longleftrightarrow}\partial B(0,R)\Big)\leq C\big((r/R)^{\epsilon}+\delta R^{d/2}\big)
    \]
\end{lemma}
\begin{proof}
    This result follows from using the Cameron-Martin theorem to bound the change in probability of arm events for $f$ and $f+\delta$. The argument was given in \cite[Section~3.3]{mv20} for planar fields, so we will only outline the steps, highlighting the changes for $d\geq 3$.

    Let $A(f):=\{B(0,r)\overset{\{f\geq\ell\}}{\longleftrightarrow}\partial B(0,R)\}$ and $h$ denote a measurable function such that $h|_{B(0,R)}\geq 1$. By monotonicity of $A$
    \[
        \P(A(f+\delta))\leq \P(A(f+\delta h))\leq \P(A(f))+d_\mathrm{TV}(f,f+\delta h)
    \]
    where $d_\mathrm{TV}$ denotes the total variation distance. Letting $H$ denote the reproducing kernel Hilbert space of $f$, if $h\in H$ then by \cite[Proposition~3.8]{bmm22} (or \cite[Corollary~3.10]{mv20} in the planar case)
    \[
        d_{\mathrm{TV}}(f,f+\delta h)\leq \delta\|h\|_H.
    \]
    Letting $\rho:\R^d\to\R$ denote the spectral density of $f$, one can show that
    \[
        \|h\|_H^2\leq \frac{\sup_{x\in\Omega}\lvert\mathcal{F}^{-1}[h](x)\rvert^2\lvert\Omega\rvert}{\inf_{x\in\Omega}\rho^2(x)}
    \]
    where $\mathcal{F}[\cdot]$ denotes the standard Fourier transform and $\Omega$ is the support of $\mathcal{F}^{-1}[h]$. We now set
    \[
        h(t)=C_1\Big(\frac{R}{2c_0}\Big)^d\mathcal{F}[\ind_{\Lambda_{c_0/R}}](t)=C_1\prod_{i=1}^d\frac{\sin(c_0 t_i/R)}{c_0 t_i/R}
    \]
    for constants $c_0,C_1>0$. By choosing $c_0$ sufficiently small and then $C_1$ sufficiently large, we can ensure that $h\geq 1$ on $B(0,R)$. It is immediate from the definition that
    \[
        \sup_\Omega\lvert \mathcal{F}^{-1}[h]\rvert^2=C_1^2(R/(2c_0))^{2d}\qquad\text{and}\qquad \lvert\Omega\rvert=(c_0/R)^d.
    \]
    By Assumption~\aref{a:arm_decay_critical}, for all $R$ sufficiently large, $\inf_\Omega\rho^2$ is bounded below by a positive constant. Therefore combining all of the previous equations we have
    \[
        \P(A(f+\delta))\leq\P(A(f))+C\delta R^{d/2}
    \]
    for some $C>0$ depending only on the distribution of $f$. The probability on the right is bounded by Assumption~\aref{a:arm_decay_critical}, allowing us to conclude.
\end{proof}

Our next result controls the degeneracy imposed by the event $A^t_{x,y}$ as $t\to 1$ and $x\to y$. This degeneracy results in a bigger difference between the conditioned and unconditioned fields.

\begin{lemma}\label{l:dc_bound}
    Let $f$ satisfy Assumption~\aref{a:basic} and $k\in\N$, then there exists $C,\overline{c}>0$ such that the following holds: for any regular affine stratified sets $(D_1,\mathcal{F}_1)$, $(D_2,\mathcal{F}_2)$, distinct $x\in D_1$, $y\in D_2$, and $t\in[0,1)$
    \[
        \det\Cov\big[f(x),f^t(y),\nabla|_{F_1} f(x),\nabla|_{F_2} f^t(y)\big]\geq I_{d_{xy}}(1-t,\lvert x-y\rvert)
    \]
    and
    \[
        \E\big[\lvert\det\nabla|_{F_1}^2 f(x)\rvert^k\lvert\det\nabla|_{F_2}^2 f^t(y)\rvert^k\;\big|\;A^t_{x,y}\big]<C(1+\ell^{2dk}),
    \]
    where $d_{xy}$ denotes the minimum of the dimensions of the strata containing $x$ and $y$ and we recall that
    \[
        I_j(s,a):=\begin{cases}
            \overline{c} &\text{if }\lvert a\rvert\in[1,\infty),\\
            \overline{c}\lvert a\rvert^{2j}s &\text{if }\lvert a\rvert\in[\sqrt{s},1),\\
            \overline{c}s^{j+1} &\text{if }\lvert a\rvert\in[0,\sqrt{s}),
        \end{cases}
        \quad\text{and}\quad A^t_{x,y}:=\big\{f(x)=f^t(y)=\ell,\nabla|_{F_1} f(x)=0, \nabla|_{F_2} f^t(y)=0\big\}.
    \]
\end{lemma}
\begin{proof}
    The two statements are precisely Lemma 6 and Lemma 7 respectively of \cite{bmm24b}. 
\end{proof}

With these preliminaries, we can now control the conditional truncated arm decay:

\begin{proof}[Proof of Proposition~\ref{p:arm_conditional}]
    For $t\in[0,1)$ and $x,y\in\R^d$, we let $\widetilde{f}^t_{x,y}$ denote the field $f$ conditioned on the event $A^t_{x,y}$. Applying H\"older's inequality to the definition of the pivotal measure in \eqref{e:piv_measure},  using Lemma~\ref{l:dc_bound} and the fact that $\phi^t_{x,y}$ is a Gaussian density we have for any $k\geq 2$
    \begin{align*}
        \mu^t_{x,y}(\ind_{\mathrm{TArm}_x(R)})&\leq \E\big[\lvert\det\nabla|_{F_1}^2f(x)\rvert^2\lvert\det\nabla|_{F_2}^2 f^t(y)\rvert^k\;\big|\;A^t_{x,y}\big]^{1/k}\P\big(\widetilde{f}^t_{x,y}\in\mathrm{TArm}_x(R)\big)^{1-1/k}\phi^t_{x,y}\\
        &\leq CI_{d_{xy}}(1-t,x-y)^{-1/2}\P\big(\widetilde{f}^t_{x,y}\in\mathrm{TArm}_x(R)\big)^{1-1/k}.
    \end{align*}
    Since $k$ can be taken arbitrarily large, the statement of the proposition will follow if we can show that
    \begin{equation}\label{e:arm_decay0}
        \P\big(\widetilde{f}^t_{x,y}\in\mathrm{TArm}_x(R)\big)\leq C\Big(h_1(R)+\exp\big(-cI_{d_{xy}}(1-t,x-y)^2/h_2(R)^2\big)\Big).
    \end{equation}
    We do this separately for our different possible arm/covariance decay assumptions.

    \underline{Critical:} Suppose first that Assumption~\aref{a:arm_decay_critical} holds. By the union bound, for any $\delta>0$
    \begin{equation}\label{e:arm_decay1}
    \begin{aligned}
        \P\big(\widetilde{f}^t_{x,y}\in\mathrm{TArm}_x(R)\big)&\leq \P\Big(\partial B(x,r_1)\overset{\{f\geq\ell-\delta\}}{\longleftrightarrow}\partial B(x,r_2)\Big)\\
        &\qquad+\P\Big(\sup_{w\;:\;\lvert w-x\rvert,\lvert w-y\rvert\geq r_3}\lvert\widetilde{f}^t_{x,y}(w)-f(w)\rvert\geq\delta\Big)
    \end{aligned}
    \end{equation}
    where
    \[
        (r_1,r_2)=\begin{cases}
            (R^{1-2\alpha},R^{1-\alpha}) &\text{if }\lvert x-y\rvert\geq 2R^{1-\alpha},\\
            (3R^{1-\alpha},R) &\text{if }\lvert x-y\rvert<2R^{1-\alpha},
        \end{cases}
        \qquad\text{and}\qquad r_3=R^{1-2\alpha}.
    \]
    for some small parameter $\alpha\in(0,1/2)$. By Gaussian regression,
    \[
        \widetilde{f}_{x,y}^t(w)-f(w)=\Cov[f(w),V^t_{x,y}]\Cov[V^t_{x,y}]^{-1}(v_\ell-V^t_{x,y})^T
    \]
    where $X^T$ denotes the transpose of $X$,
    \[
        V^t_{x,y}:=(f(x),f^t(y),\nabla|_{F_x} f(x),\nabla|_{F_y} f^t(y))\qquad\text{and}\qquad v_\ell:=(\ell,\ell,0,0).
    \]
    By Lemma~\ref{l:dc_bound}, $\|\Cov[V^t_{x,y}]^{-1}\|_\infty\leq CI_{d_{xy}}(1-t,x-y)^{-1}$ for $C>0$ depending only the distribution of $f$. For any $w$ at distance at least $r_3$ from both $x$ and $y$
    \[
        \|\Cov[f(w),V^t_{x,y}]\|_\infty\leq C h(r_3):=\begin{cases}
            Cr_3^{-\beta} &\text{if Assumption~\aref{a:cov_decay_pol} holds,}\\
            Ce^{-cr_3} &\text{if Assumption~\aref{a:cov_decay_exp} holds.}
        \end{cases}
    \]
    Combining these last three observations with the standard Gaussian tail inequality applied to the vector $V^t_{x,y}$, there exists $C,c>0$ (depending on $\ell$, and the distribution of $f$) such that for $R\geq 1$
    \begin{equation}\label{e:arm_decay2}
        \P\Big(\sup_{w\;:\;\lvert w-x\rvert,\lvert w-y\rvert\geq r_3}\lvert\widetilde{f}^t_{x,y}(w)-f(w)\rvert\geq\delta\Big)\leq C\exp(-c\delta^2I_{d_{xy}}(1-t,x-y)^2/h(r_3)^2).
    \end{equation}
    Applying Lemma~\ref{l:arm_critical} to the first term on the right-hand side of \eqref{e:arm_decay1} and using \eqref{e:arm_decay2} for the second term, we have
    \begin{align*}
        \P\big(\widetilde{f}^t_{x,y}\in\mathrm{TArm}_x(R)\big)\leq C\Big(R^{-\alpha \epsilon}+\delta R^{d/2}+\exp(-c\delta^2R^{2\beta(1-2\alpha)}I_{d_{xy}}(1-t,x-y)^2)\Big).
    \end{align*}
    Setting $\delta=R^{-d}$, choosing $\alpha>0$ sufficiently small and recalling that $\beta>d$, verifies \eqref{e:arm_decay0} for some $h_1(R)=h_2(R)=R^{-\epsilon^\prime}$ where $\epsilon^\prime>0$.

    \underline{Non-critical:} Suppose now that the first case of Assumption~\aref{a:arm_decay} holds (i.e., the subcritical regime). Using the union bound once more, \eqref{e:arm_decay1} holds with
    \[
        (r_1,r_2)=\begin{cases}
            (R/8,R/4) &\text{if }\lvert x-y\rvert\geq R/2\\
            (3R/4,R) &\text{if }\lvert x-y\rvert<R/2
        \end{cases}
        \qquad\text{and}\qquad r_3=R/8.
    \]
    By the union bound and fixing $\delta>0$ sufficiently small, the first term on the right-hand side of \eqref{e:arm_decay1} is at most $Ce^{-cR}$ while the second term is bounded as in \eqref{e:arm_decay2}. This verifies \eqref{e:arm_decay0} with $h_1(R)=e^{-cR}$ and $h_2=h$ as above.

    \underline{Supercritical:} We now suppose that the second case of Assumption~\aref{a:arm_decay} holds. By Lemma~\ref{l:arm_event_supercrit}, if $\widetilde{f}^t_{x,y}\in\mathrm{TArm}_x(R)$ then there must exist a component of $\{\widetilde{f}^t_{x,y}\leq \ell\}\setminus (B(x,R/20)\cup B(y,R/20))$ of diameter at least $R/20$ which intersects or surrounds at least one of $\partial B(x,R/20)$ and $\partial B(y,R/20)$. Letting $\partial B(w,r)^+$ denote the points of $\Z^d$ within unit distance of $\partial B(w,r)$ and $\lfloor x\rfloor$ the integer part of $x$, by the union bound
    \begin{align*}
        \P\big(\widetilde{f}^t_{x,y}\in\mathrm{TArm}_x(R)\big)&\leq \P\Big(\sup_{w\;:\;\lvert w-x\rvert,\lvert w-y\rvert\geq R/20}\lvert\widetilde{f}^t_{x,y}(w)-f(w)\rvert\geq\delta\Big)\\
        &\quad +\sum_{v\in\{x,y\}}\sum_{j=\lfloor R/20\rfloor}^\infty\sum_{w\in\partial B(v,j)^+}\P\Big(w+\Lambda_1\overset{\{f\leq\ell+\delta\}}{\longleftrightarrow}w+\partial\Lambda_{\max\{j-R/20,R/20\}}\Big).
    \end{align*}
    By Assumption~\aref{a:arm_decay}, the sum on the right-hand side decays exponentially in $R$, while the first term on the right-hand side can be controlled as in the subcritical case. We conclude that \eqref{e:arm_decay0} also holds in the super-critical case.
\end{proof}

\section{Qualitative limit theorems}
In this section we use quasi-association and the additivity and stabilisation estimates to prove the previously stated qualitative limit theorems for topological functionals: the law of large numbers, variance asymptotics and central limit theorems. For clarity, we separate the proofs of the different parts of Theorem~\ref{t:var+clt}.

\subsection{Law of large numbers}
We establish the law of large numbers for excursion/level-set functionals by adapting Etemadi's argument for a sum of i.i.d.\ variables. A similar approach was applied to discrete quasi-associated fields in \cite[Chapter~4]{bs07}. We also make use of arguments from \cite{kw18} to show that the expectation of a topological functional in a box scales like the volume of the box.
\begin{proof}[Proof of Theorem~\ref{t:lln} (excursion/level-set functionals)]
    Let $\Phi$ be a bounded excursion/level-set functional. We abbreviate $\Phi_R:=\Phi(\Lambda_R,f)$ and recall that $\overline{\Phi}_R:=\Phi_R-\E[\Phi_R]$. By quasi-association (Corollary~\ref{c:quasi_association}), for any $R\geq 1$
    \begin{equation}\label{e:lln-1}
        \Var[R^{-d}\Phi_R]\leq R^{-2d}C\|\Phi\|_\mathrm{Lip}^2\int_{\Lambda_{R+2}^2}\widetilde{K}(x-y)\;dxdy=O((\log R)^{-1-\epsilon})
    \end{equation}
    as $R\to\infty$, where the final equality uses Assumption~\aref{a:cov_decay_weak}. Hence $R^{-d}\overline{\Phi}_R\to 0$ in $L^2$. Now let $\alpha>1$ and $R_n=\alpha^n$. By the previous equation, there exists $C>0$ such that for any $\delta>0$
    \[
        \sum_{n=1}^\infty\P(R_n^{-d}\lvert \overline{\Phi}_{R_n}\rvert\geq\delta)\leq \sum_{n=1}^\infty \delta^{-2}\Var[R_n^{-d}\Phi_{R_n}]\leq C\delta^{-2}\sum_{n=1}^\infty (n\log\alpha)^{-1-\epsilon}<\infty.
    \]
    Hence, by Borel-Cantelli, $\overline{\Phi}_{R_n}/R_n^d\to 0$ almost surely as $n\to\infty$.

    Let us now assume that the excursion/level component functional $\varphi$ (used to define $\Phi$ in \eqref{e:es_functional2}) is non-negative. It follows that $\Phi_R$ is non-decreasing in $R$ and so for $R_n\leq R\leq R_{n+1}$
    \[
        \frac{R_{n+1}^d}{R^d}\frac{\overline{\Phi}_{R_n}}{R_{n+1}^d}-\frac{\E[\Phi_{R_{n+1}}]-\E[\Phi_{R_n}]}{R_{n}^d}\leq\frac{\overline{\Phi}_R}{R^d}\leq \frac{R_{n+1}^d}{R^d}\frac{\overline{\Phi}_{R_{n+1}}}{R_{n+1}^d}+\frac{\E[\Phi_{R_{n+1}}]-\E[\Phi_{R_n}]}{R_n^d}.
    \]
    Using the previously established convergence, with probability one
    \begin{equation}\label{e:lln0}
        \limsup_{R\to\infty}\Big\lvert\frac{\overline{\Phi}_R}{R^d}\Big\rvert\leq \limsup_{n\to\infty}\frac{\E[\Phi_{R_{n+1}}-\Phi_{R_n}]}{R_n^d}.
    \end{equation}
    Since $\Phi(D,f)$ is defined by summing contributions from excursion/level-set components inside $D$ (see \eqref{e:es_functional2}), it follows that
    \[
        \Phi_{R_{n+1}}=\Phi_{R_n}+\Phi(\Lambda_{R_{n+1}},f,\overline{\Lambda_{R_{n+1}}\setminus\Lambda_{R_n}})
    \]
    where we recall that $\Phi(D,g,F)$ was defined (prior to Lemma~\ref{l:boundary_contribution}) as the contribution from components in $D$ which intersect $F$. Bounding the final term on the right-hand side in terms of critical points (Lemma~~\ref{l:boundary_contribution}) we see that \eqref{e:lln0} is bounded above by
    \[
        C\limsup_{n\to\infty}R_n^{-d}\E[\overline{N}_\mathrm{Crit}(\overline{\Lambda_{R_n+1}\setminus\Lambda_{R_n}},f)]\leq C\limsup_{n\to\infty}R_n^{-d}(R_{n+1}^d-R_n^d+R_{n+1}^{d-1}+R_{n+1}^{d-2}+\dots+1)\leq C(\alpha^d-1)
    \]
    where the first inequality uses our bound on the expected number of stratified critical points (Lemma~\ref{l:crit_moments}) applied to the minimal stratification of $\overline{\Lambda_{R_{n+1}}\setminus\Lambda_{R_n}}$. Since $\alpha>1$ was arbitrary, this establishes that $R^{-d}\overline{\Phi}_R\to 0$ almost surely under the assumption that $\varphi\geq 0$. For a general bounded excursion/level-set functional we can simply decompose $\varphi$ into positive and negative parts and apply the proven result.

    It remains to show that $R^{-d}\E[\Phi_R]$ converges to a constant as $R\to\infty$. For $1\leq r\leq R$ we define
    \[
        A(r,R)=\{x\in r\Z^d\;|\;x+\Lambda_r\subseteq\Lambda_R\}\qquad\text{and}\qquad D(r,R)=\overline{\Lambda_R\setminus\bigcup_{x\in A(r,R)}x+\Lambda_r}.
    \]
    Each excursion/level-set component contained in $\Lambda_R$ must either 1) be contained in $x+\Lambda_r$ for some $x\in A(r,R)$, 2) be contained in $D(r,R)$ or 3) intersect $x+\partial\Lambda_R$ for some $x\in A(r,R)$. Hence by the definition of an excursion/level-set functional, we have the decomposition
    \begin{equation}\label{e:lln1}
        \Phi(\Lambda_R)=\sum_{x\in A(r,R)}\Phi(x+\Lambda_r) + \Phi(D(r,R)) + \Phi\Big(\Lambda_R, \bigcup_{x\in A(r,R)}x+\partial\Lambda_r\Big)
    \end{equation}
    where we have dropped $f$ as an argument from the topological functionals for clarity. Using our bounds in terms of critical points once more (Lemmas~\ref{l:boundary_contribution} and~\ref{l:crit_moments})
    \[
        \E\Big[\Phi\Big(\Lambda_R, \bigcup_{x\in A(r,R)}x+\partial\Lambda_r\Big)\Big]\leq C\sum_{x\in A(r,R)}\E[\overline{N}_\mathrm{Crit}(x+\partial\Lambda_r,f)]\leq CR^dr^{-1}
    \]
    since $A(r,R)$ has at most $C(R/r)^d$ points and $\partial\Lambda_r$ is composed of a constant number of strata, each of which has volume at most $r^{d-1}$. Similarly, using Lemmas~\ref{l:top_crit} and~\ref{l:crit_moments}
    \[
        \E[D(r,R)]\leq CR^{d-1}r.
    \]
    Taking expectations of \eqref{e:lln1} and using stationarity of $f$, we therefore have
    \begin{equation}\label{e:lln2}
        \big\lvert\E[\Phi(\Lambda_R)]-\lvert A(r,R)\rvert\E[\Phi(\Lambda_r)]\big\rvert\leq C(R^dr^{-1}+R^{d-1}r),
    \end{equation}
    where $\lvert A(r,R)\rvert$ denotes the number of points in $A(r,R)$. Given $\epsilon>0$, we now choose $r$ such that
    \[
        \frac{C}{r}<\epsilon\qquad\text{and}\qquad\frac{\E[\Phi(\Lambda_r)]}{r^d}\geq \limsup_{R\to\infty}\frac{\E[\Phi(\Lambda_R)]}{R^d}-\epsilon.
    \]
    Using this choice of $r$ in \eqref{e:lln2}, we then have
    \[
        \liminf_{R\to\infty}\frac{\E[\Phi(\Lambda_R)]}{R^d}\geq \liminf_{R\to\infty}\Big(\frac{r}{R}\Big)^d\lvert A(r,R)\rvert\frac{\E[\Phi(\Lambda_r)]}{r^d}-C\Big(\frac{1}{r}+\frac{r}{R}\Big)\geq \limsup_{R\to\infty}\frac{\E[\Phi(\Lambda_R)]}{R^d}-2\epsilon.
    \]
    Since $\epsilon>0$ was arbitrary, this completes the proof of convergence.
\end{proof}
\begin{remark}\label{r:lln_non-stationary}
    The argument given above to show that $R^{-d}\overline{\Phi}_R\to 0$ almost surely and in $L^2$ could likely be extended to deal with suitable non-stationary fields. Stationarity was used in two ways: 1) to show that the expected number of (stratified) critical points in a domain is bounded by a constant times the volume of the domain and 2) to justify the bound on pivotal measures in Proposition~\ref{p:intensity_integrable} (proven in \cite{bmm24b}). The former condition could instead be justified by imposing a uniform bound on the first few derivatives of the covariance function. The latter condition could likely be derived under some uniform non-degeneracy assumptions by repeating the arguments given in \cite{bmm24b}. Since the corresponding proofs are reasonably long, we do not pursue this issue here in the interest of conciseness.
\end{remark}

The previous argument for almost sure convergence does not immediately apply to the Euler characteristic: the difficulty arises from accounting for boundary effects in the monotonicity part of the proof. We instead give a short direct proof based on the ergodic theorem.

\begin{proof}[Proof of Theorem~\ref{t:lln} (Euler characteristic)]
    Letting $\Phi$ denote the Euler characteristic, by the critical point decomposition in Lemma~\ref{l:ec_decomp}, we have
    \[
        \lvert\Phi_R-\Psi_R\rvert\leq\overline{N}_\mathrm{Crit}(\partial\Lambda_R,f)\qquad\text{where}\qquad\Psi_R:=\sum_{i=0}^d(-1)^{d-i}N_\mathrm{Crit}^{(i)}(\Lambda_R,f,\ell).
    \]
    Using stationarity and the bounds on critical point moments in Lemma~\ref{l:crit_moments}, we see that
    \[
        \frac{\E[\Phi_R]}{R^d}\to\E[\Psi_1]\qquad\text{and}\qquad\frac{\Phi_R-\Psi_R}{R^d}\overset{L^2}{\longrightarrow}0
    \]
    as $R\to\infty$. Using the same bound once more,
    \[
        \sum_{n=1}^\infty\frac{\E[(\Phi_n-\Psi_n)^2]}{n^{2d}}\leq C\sum_{n=1}^\infty\frac{n^{2(d-1)}}{n^{2d}}<\infty,
    \]
    and so by Borel-Cantelli, $n^{-d}\lvert\Phi_n-\Psi_n\rvert\to 0$ almost surely as $n\to\infty$. It remains to show that $R^{-d}\Psi_R$ converges almost surely and in $L^2$ to its mean as $R\to\infty$. This follows directly from the ergodic theorem \cite[Theorem~25.6]{kal21} since $\Psi_R$ is a stationary, square integrable functional of $f$ (which is ergodic by the assumption that its covariance decays to zero).
\end{proof}

\subsection{Variance asymptotics}\label{ss:variance}
We now use the covariance formula (Theorem~\ref{t:cov_formula}) to study the asymptotic variance of topological functionals. Specifically, we use a stabilisation argument to prove convergence of the variance normalised by the volume for fields with integrable correlations. The same argument was used in \cite{bmm24b} to study the component count.

\begin{proof}[Proof of Theorem~\ref{t:var+clt} (variance asymptotics)]
    Let $\Phi$ be a bounded excursion/level-set functional or the Euler characteristic and recall that $\mathcal{F}_R$ denotes the minimal stratification of $\Lambda_R$. By the covariance formula for Lipschitz topological functionals (Theorem~\ref{t:cov_formula})
    \begin{equation}\label{e:var_limit1}
        \Var[\Phi_R]=\sum_{F_1,F_2\in\mathcal{F}_R}\underbrace{\int_{F_1\times F_2}K(x-y)\int_0^1\mu^t_{x,y}(d_x\Phi_Rd_y^t\Phi_R)\;dtdv_{F_1}(x)dv_{F_2}(y)}_{:=J(F_1,F_2)}.
    \end{equation}
    We first show that the contributions from boundary strata are negligible. Since $\Phi$ is Lipschitz, by the pivotal measure bound (Proposition~\ref{p:intensity_integrable}) and integrability of $\widetilde{K}$ (Assumption~\aref{a:cov_decay_int}), for a boundary stratum $F_1$ and arbitrary $F_2$
    \[
        \lvert J(F_1,F_2)\rvert\leq C\|\Phi\|^2_\mathrm{Lip}\int_{F_1^{+1}}1+\int_{F_2^{+1}}\widetilde{K}(x-y)\;dydx\leq Cv_{F_1}(F_1).
    \]
    Swapping the order of integration yields the upper bound $Cv_{F_2}(F_2)$ whenever $F_2$ is a boundary stratum. We conclude that the contribution to \eqref{e:var_limit1} from all terms involving a stratum of dimension less than $d$ is $O(R^{d-1})$ as $R\to\infty$. 
    
    Recalling the definition of $\sigma$ in \eqref{e:sigma} and using stationarity, we have
    \begin{align*}
        J(\Lambda_R,\Lambda_R)-\sigma^2R^d&=\underbrace{\int_{\Lambda_{R}^2}K(x-y)\int_0^1\mu^t_{x,y}(d_x\Phi_Rd_y^t\Phi_R-d_x\Phi_\infty d_y^t\Phi_\infty)\;dtdxdy}_{=:\Circled{\mathrm{A}}}\\
        &\qquad-\underbrace{\int_{\Lambda_R\times\R^d\setminus\Lambda_{R+1}}K(x-y)\int_0^1\mu^t_{x,y}(d_x\Phi_\infty d_y^t\Phi_\infty)\;dtdxdy}_{=:\Circled{\mathrm{B}}}\\
        &\qquad-\underbrace{\int_{\Lambda_R\times\Lambda_{R+1}\setminus\Lambda_{R}}K(x-y)\int_0^1\mu^t_{x,y}(d_x\Phi_\infty d_y^t\Phi_\infty)\;dtdxdy.}_{=:\Circled{\mathrm{C}}}
    \end{align*}
    The (first) statement of the theorem will follow if we can show that each of the final three terms are $o(R^d)$ as $R\to\infty$. Proposition~\ref{p:stabilisation} states that this holds for $\Circled{\mathrm{A}}$. By the fact that $\Phi$ is Lipschitz, and the pivotal measure bound in Proposition~\ref{p:intensity_integrable},
    \[
        \lvert\Circled{\mathrm{C}}\rvert\leq C\|\Phi\|_\mathrm{Lip}^2\int_{\Lambda_{R+1}\setminus\Lambda_R}1+\int_{\R^d}\lvert K(x-y)\rvert\;dxdy=O(R^{d-1}).
    \]
    By the same reasoning, and recalling that $\|\cdot\|_\infty$ denotes the sup-norm on $\R^d$,
    \begin{align*}
        \lvert\Circled{\mathrm{B}}\rvert\leq C\|\Phi\|_\mathrm{Lip}^2\int_{\R^{2d}}\lvert K(x-y)\rvert\ind_{\| x\|_\infty<R<\| y\|_\infty}\;dxdy&\leq C\|\Phi\|_\mathrm{Lip}^2\int_{\R^{2d}}\lvert K(u)\rvert \ind_{R-\|u\|_\infty<\|x\|_\infty<R}\;dxdu\\
        &\leq C\|\Phi\|_\mathrm{Lip}^2\int_{\R^{d}}\lvert K(u)\rvert R^{d-1}\min\{R,\|u\|_\infty\}\;du.
    \end{align*}
    Since $K$ is integrable, a dominated convergence argument shows that the final expression is $o(R^d)$.
\end{proof}

\begin{remark}\label{r:var_ec}
    The formula for the asymptotic variance $\sigma^2$ given in \eqref{e:sigma} involves topological derivatives, which are challenging to evaluate explicitly for non-local functionals. However in the case of the Euler characteristic, the topological derivative simplifies substantially: by the critical point decomposition in Lemma~\ref{l:ec_decomp}, for $x$ in the interior of $\Lambda_R$ and $g\in\mathrm{Morse}_x(\Lambda_R,\mathcal{F}_R,\ell)$
    \[
        d_x\Phi(\Lambda_R,g,\ell)=(-1)^{d-\mathrm{index}\nabla^2 g(x)}
    \]
    where $\mathrm{index}\nabla^2g(x)$ denotes the number of negative eigenvalues of $\nabla^2 g(x)$. Then by definition of the pivotal measure \eqref{e:piv_measure} and the fact that $\lvert\det\nabla^2 f(x)\rvert=(-1)^{\mathrm{index}\nabla^2 f(x)}\det\nabla^2 f(x)$ we have
    \[
        \sigma^2=\int_{\R^d}K(x)\int_0^1\E\big[\det\big(\nabla^2f(x)\nabla^2f^t(0)\big)\big|A^t_{x,0}\big]\phi^t_{x,0}\;dtdx.
    \]
\end{remark}

\subsection{Central limit theorems}
We end this section by establishing the distributional and almost sure CLTs stated in Theorem~\ref{t:var+clt}. The distributional CLT essentially follows from Newman's argument, originally given for discrete associated fields \cite{new80}.

\begin{proof}[Proof of Theorem~\ref{t:var+clt} (distributional CLT)]
    For a random variable $X$ and $t\in\R$, we let $\varphi_t(X)=\E[e^{itX}]$ denote its characteristic function and we recall that $\overline{X}=X-\E[X]$ and $\widetilde{\Phi}_R=R^{-d/2}\overline{\Phi}_R$. Letting $e_R:=\Phi(\Lambda_R)-\sum_{x\in\chi_R}\Phi(x+\Lambda_r)$ and applying the inequality $\lvert e^{iu}-e^{iv}\rvert\leq \lvert u-v\rvert$ for $u,v\in\R$, we have
    \begin{equation}\label{e:clt1}
        \Big\lvert\varphi_t(\widetilde{\Phi}_R)-\varphi_t\Big(R^{-d/2}\Big(\sum_{x\in\chi_R}\overline{\Phi}(x+\Lambda_r)\Big)\Big)\Big\rvert\leq\frac{\lvert t\rvert}{R^{d/2}}\E\big[\lvert e_R-\E[e_R]\rvert\big]\leq\lvert t\rvert\Big(\frac{\Var[e_R]}{R^d}\Big)^{1/2}\to 0,
    \end{equation}
    as $R\to\infty$, by Proposition~\ref{p:additivity}.

    We now choose some ordering $\prec$ (e.g., lexicographic) on points in $\chi_R$. By repeated use of the triangle inequality, the bound $\lvert e^{it}\rvert\leq 1$ and quasi-association (Corollary~\ref{c:quasi_association})
    \begin{equation}\label{e:clt2}
    \begin{aligned}
        \Big\lvert&\varphi_t\Big(R^{-d/2}\sum_{x\in\chi_R}\overline{\Phi}(x+\Lambda_r)\Big)-\prod_{x\in\chi_R}\varphi_t(R^{-d/2}\overline{\Phi}(x+\Lambda_r))\Big\rvert\\
        &\qquad\leq\sum_{x\in\chi_R}\Big\lvert\Cov\big[\exp(itR^{-d/2}\overline{\Phi}(x+\Lambda_r)),\prod_{z\prec x}\exp(itR^{-d/2}\overline{\Phi}(z+\Lambda_r))\big]\Big\rvert\\
        &\qquad\leq C\|\Phi\|_\mathrm{Lip}^2\frac{\lvert t\rvert^2}{R^d}\sum_{x\in\chi_R}\int_{\Lambda_{r+2}\times(\R^d\setminus\Lambda_{r+4a-2})}\widetilde{K}(x-y)\;dxdy\\
        &\qquad\leq C\|\Phi\|_\mathrm{Lip}^2\lvert t\rvert^2\int_{\R^d\setminus\Lambda_a}\widetilde{K}(y)\;dy\to 0\qquad\text{as }R\to\infty.
    \end{aligned}
    \end{equation}

    Combining \eqref{e:clt1} and \eqref{e:clt2}, the theorem will follow if we can show that, as $R\to\infty$
    \begin{equation}\label{e:clt3}
        \prod_{x\in\chi_R}\E[\exp(itR^{-d/2}\overline{\Phi}(x+\Lambda_r))]\to\exp(-\sigma^2t^2/2).
    \end{equation}
    Using the elementary inequality
    \[
        e^{iu}=1+iu-u^2/2+\mathcal{R}(u)\qquad\text{for }u\in\R\text{ where }\lvert\mathcal{R}(u)\lvert\leq\lvert u\rvert^3,
    \]
    stationarity of $f$ and the fact that $\overline{\Phi}$ is centred, the left hand side of \eqref{e:clt3} can be expressed as
    \[
        \Big(1-(r/R)^{d}\frac{t^2\Var[\Phi_r]}{r^{d}}+\E\big[\mathcal{R}\big(tR^{-d/2}\overline{\Phi}_r\big)\big]\Big)^{(R/(r+4a))^d}.
    \]
    If $\alpha_n$ is a sequence in $\mathbb{C}$ converging to $\alpha$, then, by a standard analytic argument, $(1+\alpha_n/n)^n\to e^\alpha$. Applying this to the above expression, recalling that $a=o(r)$, to establish \eqref{e:clt3} it is sufficient that
    \[
        \frac{\Var[\Phi_r]}{r^d}\to\sigma^2\qquad\text{and}\qquad(R/r)^d\E\big[\mathcal{R}\big(tR^{-d/2}\overline{\Phi}_r\big)\big]\to 0
    \]
    as $R\to\infty$. The former statement holds by the first part of Theorem~\ref{t:var+clt}. Turning to the latter statement; using our bounds in terms of the number of critical points (Lemma~\ref{l:top_crit} and Lemma~\ref{l:crit_moments})
    \[
        (R/r)^d\E\big[\mathcal{R}\big(tR^{-d/2}\overline{\Phi}_r\big)\big]\leq \lvert t\rvert^3r^{-d}R^{-d/2}\E\big[\lvert\overline{\Phi}_r\rvert^3\big]\leq C\lvert t\rvert^3r^{2d}R^{-d/2}.
    \]
    Choosing $r<R^{1/5}$ ensures that this tends to zero as $R\to\infty$, verifying the second statement above and so completing the proof of the theorem.
\end{proof}

Our proof of the almost sure CLT adapts arguments from \cite{ps95,vro00} for sums of stationary, associated random variables. We also refer the reader to \cite{mrz25} which establishes an almost sure CLT for local functionals of smooth Gaussian fields using a very different method of proof.

\begin{proof}[Proof of Theorem~\ref{t:var+clt} (almost sure CLT)]
    Let $F:\R\to\R$ be Lipschitz continuous and define
    \[
        S_R=\frac{1}{\log R}\int_1^Rr^{-1}F(\widetilde{\Phi}_r)\;dr.
    \]
    For $r<s$, by quasi-association (Corollary~\ref{c:quasi_association})
    \[
        \lvert\Cov[F(\widetilde{\Phi}_r),F(\widetilde{\Phi}_s)]\rvert\leq C(rs)^{-d/2}\|F\|_\mathrm{Lip}^2\|\Phi\|_\mathrm{Lip}^2\int_{\Lambda_{r+2}\times\Lambda_{s+2}}\widetilde{K}(x-y)\;dxdy\leq C r^{d/2}s^{-d/2}.
    \]
    Hence
    \begin{align*}
        \Var[S_R]\leq \frac{2C}{(\log R)^2}\int_1^R\int_1^sr^{d/2-1}s^{-d/2-1}\;drds\leq \frac{C}{\log R}.
    \end{align*}
    We now define the sequence $R_n:=e^{n^2}$ so that, by Chebyshev's inequality, for any $\epsilon>0$
    \[
        \sum_{n=1}^\infty\P(\lvert \overline{S}_{R_n}\rvert>\epsilon)\leq \epsilon^{-2}\sum_{n=1}^\infty\Var[S_{R_n}]\leq C\epsilon^{-2}\sum_{n=1}^\infty\frac{1}{n^2}<\infty
    \]
    and hence, by Borel-Cantelli, $S_{R_n}-\E[S_{R_n}]\to0$ almost surely as $n\to\infty$.

    We next observe that $\lvert F(\widetilde{\Phi}_r)\rvert\leq\lvert F(0)\rvert+\|F\|_\mathrm{Lip}\lvert\widetilde{\Phi}_r\rvert$ which is uniformly integrable (over $r\geq 1$) by the variance asymptotics of Theorem~\ref{t:var+clt}. By the continuous mapping theorem and the distributional convergence in Theorem~\ref{t:var+clt}, $F(\widetilde{\Phi}_r)$ converges in distribution, as $r\to\infty$, to $F(\sigma Z)$ where $Z\sim\mathcal{N}(0,1)$. Hence by Vitali's convergence theorem $\E[F(\widetilde{\Phi}_r)]\to\E[F(\sigma Z)]$. By an elementary argument
    \[
        \E[S_{R}]=\frac{\int_1^Rr^{-1}\E[F(\widetilde{\Phi}_r)]\;dr}{\int_1^Rr^{-1}\;dr}\to\E[F(\sigma Z)]\qquad\text{as }R\to\infty,
    \]
    and so we conclude that $S_{R_n}\to\E[F(\sigma Z)]$ almost surely as $n\to\infty$.

    It remains to control behaviour of $S_R$ away from the subsequence $R_n$. By the triangle inequality
    \begin{equation}\label{e:asclt1}
        \sup_{R\in[R_n,R_{n+1}]}\lvert S_R-S_{R_n}\rvert\leq \sup_{R\in[R_n,R_{n+1}]}\Big\lvert S_R-\frac{\log R_n}{\log R}S_{R_n}\Big\rvert+\Big\lvert \frac{\log R_n}{\log R_{n+1}}-1\Big\rvert\lvert S_{R_n}\rvert.
    \end{equation}
    Since the sequence $\lvert S_{R_n}\rvert$ is bounded almost surely and recalling that $R_n=e^{n^2}$, we see that the second term on the right-hand side above tends to zero almost surely. The first term on the right-hand side is bounded by
    \begin{equation}\label{e:asclt2}
        \frac{1}{\log R_n}\int_{R_n}^{R_{n+1}}r^{-1}\lvert F(\widetilde{\Phi}_r)\rvert\;dr\leq \frac{1}{\log R_n}\big(\log(R_{n+1})S^\prime_{R_{n+1}}-\log(R_n) S^\prime_{R_n}\Big)
    \end{equation}
    where $S^\prime_R$ is defined as the analogue of $S_R$ with $\lvert F\rvert$ instead of $F$. Our previous subsequential arguments apply equally well to $\lvert F\rvert$ and so $S^\prime_{R_n}\to\E[\lvert F(\sigma Z)\rvert]$ almost surely as $n\to\infty$. Hence \eqref{e:asclt2} and \eqref{e:asclt1} both converge to zero almost surely as $n\to\infty$, completing the proof of the theorem.
\end{proof}

\section{Quantitative limit theorems}
In this section, we prove the quantitative limit theorems for topological functionals stated in Section~\ref{ss:quant_lim}.

\subsection{Variance asymptotics}
We begin by establishing the following rate of convergence for the normalised variance of topological functionals which we then use as an input to the quantitative CLT.
\begin{theorem}\label{t:variance_limit_quant}
    Let $f$ satisfy Assumptions~\aref{a:basic} and~\aref{a:cov_decay_pol} and $\ell\in\R$. Suppose that either
    \begin{enumerate}
        \item $\Phi$ is the Euler characteristic,
        \item $f$ satisfies Assumption~\aref{a:arm_decay} at level $\ell$, or
        \item $f$ satisfies Assumption~\aref{a:arm_decay_critical} at level $\ell$,
    \end{enumerate}
    then for $\sigma$ defined in \eqref{e:sigma}, as $R\to\infty$
    \begin{equation}\label{e:var_quant}
        \Var[\widetilde{\Phi}_R]=\sigma^2+O(R^{-\gamma})
    \end{equation}
    where in the first case $\gamma=\min\{1,\beta-d\}$, in the second case $\gamma=\min\{1,\beta-d,\beta/(2d+2)\}$ and in the third case $\gamma>0$.
\end{theorem}

\begin{proof}[Proof of Theorem~\ref{t:variance_limit_quant}]
    Assuming that $\Phi$ is the Euler characteristic, from the proof of the variance asymptotics in Section~\ref{ss:variance}, we see that for all $R$ sufficiently large
    \begin{align*}
        \big\lvert \Var[\widetilde{\Phi}_R]-\sigma^2\big\rvert&\leq C\Big(R^{-1}+\int_{\R^d}\lvert K(x)\rvert\min\{1,\lvert x\rvert/R\}\;dx\Big)\\
        &\leq C\Big(R^{-1}+\int_{\lvert x\rvert>R}\lvert x\rvert^{-\beta}\;dx+R^{-1}\int_{\lvert x\rvert\leq R}(1+\lvert x\rvert)^{-\beta}\lvert x\rvert\;dx\Big)\leq CR^{-\min\{1,\beta-d\}}
    \end{align*}
    using the bound on the covariance in Assumption~\aref{a:cov_decay_pol}.

    If $\Phi$ is a bounded excursion/level-set functional, then from the same proof we have the extra error term
    \[
        R^{-d}\int_{\Lambda_{R}^2}K(x-y)\int_0^1\mu^t_{x,y}(d_x\Phi_Rd_y^t\Phi_R-d_x\Phi_\infty d_y^t\Phi_\infty)\;dtdxdy.
    \]
    By Proposition~\ref{p:q_stabilisation}, this term is $O(R^{-\min\{1,\beta/(2d+2)\}})$ whenever Assumption~\aref{a:arm_decay} holds and is $O(R^{-\epsilon})$ for some $\epsilon>0$ whenever Assumption~\aref{a:arm_decay_critical} holds, proving the desired bound.
\end{proof}

\subsection{Higher moment bounds}
In this section we prove the bounds on centred moments of topological functionals in Theorem~\ref{t:higher_moments}. Several approaches have been developed to study the higher moments of sums of discrete (quasi)-associated random fields (see \cite[Chapter~2]{bs07} for a selection). We will make use of the following result, which was originally proven in \cite{bb97}.

We define a \emph{block} to be any set of the form $\Z^d\cap\prod_{i=1}^d[a_i,b_i]$ where $a_i<b_i$ for all $i$ and we denote by $\mathcal{U}$ the set of blocks in $\Z^d$. We let $\dist_\infty$ and $\diam_\infty$ respectively denote the distance and diameter defined by the sup-norm on $\R^d$.

\begin{theorem}[{\cite[Theorem~1.20]{bs07}}]\label{t:moment_condition}
    Let $(X_j)_{j\in\Z^d}$ be a collection of centred random variables, $n\in \N$ and $p>2n$ such that $D_p:=\sup_{j\in\Z^d}\E[\lvert X_j\rvert^p]<\infty$. Suppose that for any $i_1,\dots,i_M,j_1,\dots,j_N\in\Z^d$ (not necessarily distinct) with $M+N\leq 2n$
    \begin{equation}\label{e:moment_condition}
        \lvert\Cov[X_{i_1}\dots X_{i_M},X_{j_1}\dots X_{j_N}]\rvert\leq D_p^{\frac{N+M}{p}}F(\dist_\infty(I,J))G(\lvert I\rvert,\lvert J\rvert)
    \end{equation}
    for non-decreasing functions $F$ and $G$, where $I=\{i_1,\dots,i_M\}$ and $J=\{j_1,\dots,j_N\}$. Then for any block $U\in\mathcal{U}$
    \begin{equation}\label{e:moment_bound}
        \E\Big[\Big(\sum_{i\in U}X_i\Big)^{2n}\Big]\leq C_{n,d}\lvert U\rvert^n\Big(D_{2n}+D_p^\frac{2n}{p}\max_{a+b=2n}G(a,b)\sum_{k=1}^{\diam_\infty(U)}k^{nd-1}F(k)\Big)
    \end{equation}
    where $C_{n,d}>0$ depends only on $n$ and $d$.
\end{theorem}

The proof of this result uses a combinatorial argument and \eqref{e:moment_condition} to control the expansion of the left-hand side of \eqref{e:moment_bound}. The condition \eqref{e:moment_condition} can be verified using quasi-association and a truncation argument (which we adapt from the proof of \cite[Lemma~1.22]{bs07}):

\begin{lemma}\label{l:moment_condition}
    Let $f$ satisfy Assumptions~\aref{a:basic}, \aref{a:smooth} and \aref{a:cov_decay_pol} and let $\Phi$ be a Lipschitz topological functional. For $M,N\in\N$, let $A_1,\dots,A_M,B_1,\dots,B_N\subset\R^d$ be regular affine stratified sets and suppose that at most $m$ strata of all of these sets intersects any unit ball. Denoting $A=\cup_{i=1}^M A_i$ and $B=\cup_{i=1}^N B_i$, for any integer $p>M+N$
    \begin{align*}
        \Big\lvert\Cov\Big[\prod_{i=1}^M\overline{\Phi}(A_i),\prod_{j=1}^N\overline{\Phi}(B_j)\Big]\Big\rvert\leq C(M+N)\Big(\min\{\lvert A^{+2}\rvert,\lvert B^{+2}\rvert\}&(1+\dist(A^{+2},B^{+2}))^{-\beta+d}\Big)^{\frac{p-M-N}{p-2}}\\
        &\times\sup_{D\in\{A_1,\dots,A_M,B_1,\dots,B_N\}}\E[\lvert\overline{\Phi}(D)\rvert^p]^\frac{M+N-2}{p-2}
    \end{align*}
    where $C>0$ depends only on the distribution of $f$, the Lipschitz norm of $\Phi$ and $m$.
\end{lemma}
\begin{proof}
    If $M=N=1$, then the bound follows immediately from quasi-association (Corollary~\ref{c:quasi_association}), so we suppose that $M+N>2$. For $T>0$ to be determined later, define $H_T:\R\to\R$ by
    \[
        H_T(x)=\begin{cases}
            T &\text{if }x>T\\
            x &\text{if }-T\leq x\leq T\\
            -T &\text{if }x<-T
        \end{cases}
    \]
    and $h_T(x)=x-H_T(x)$. We also denote $Y_i=\overline{\Phi}(A_i)$ and $Y_{M+j}=\overline{\Phi}(B_j)$ for brevity. By linearity
    \begin{equation}\label{e:moment_condition1}
    \begin{aligned}
        \Cov[Y_1\dots Y_M,Y_{M+1}\dots&\, Y_{M+N}]\\
        &=\Cov[h_T(Y_1)Y_2\dots Y_M,Y_{M+1}\dots Y_{M+N}]\\
        &\quad+\Cov[H_T(Y_1)h_T(Y_2)Y_3\dots Y_M,Y_{M+1}\dots Y_{M+N}]\\
        &\quad+\dots\\
        &\quad+\Cov[H_T(Y_1)\dots H_T(Y_M),H_T(Y_{M+1})\dots H_T(Y_{M+N-1})h_T(Y_{M+N})]\\
        &\quad+\Cov[H_T(Y_1)\dots H_T(Y_M),H_T(Y_{M+1})\dots H_T(Y_{M+N})].
    \end{aligned}
    \end{equation}
    Since $H_T$ is bounded, we observe that $H_T(Y_1)\dots H_T(Y_M)$ and $H_T(Y_{M+1})\dots H_T(Y_{M+N})$ are both Lipschitz topological functionals applied respectively to $A$ and $B$ with the stratifications induced by the stratifications of $A_1,\dots, A_M$ and $B_1,\dots,B_N$ respectively. Hence by quasi-association (Corollary~\ref{c:quasi_association})
    \begin{align*}
        \lvert \Cov[H_T(Y_1)\dots H_T(Y_M),H_T(Y_{M+1})\dots &\,H_T(Y_{M+N})]\rvert\leq C m^2 T^{M+N-2}\|\Phi\|_{\mathrm{Lip}}^2\int_{A^{+2}\times B^{+2}}\widetilde{K}(x-y)\;dxdy\\
        &\leq C T^{M+N-2} \underbrace{\min\{\lvert A^{+2}\rvert,\lvert B^{+2}\rvert\}(1+\dist(A^{+2},B^{+2}))^{-\beta+d}}_{=:\mathcal{T}_1}
    \end{align*}
    where the second inequality uses Assumption~\aref{a:cov_decay_pol}.

    Using the trivial bound $\lvert h_T(x)\rvert\leq T^{-(p-M-N)}\lvert x\rvert^{1+p-M-N}$, for $i\leq M$ the $i$-th term on the right-hand side of \eqref{e:moment_condition1} is bounded in absolute value by
    \begin{align*}
        &T^{-(p-M-N)}\big(\E\big[\lvert Y_i\rvert^{p-M-N}\lvert Y_1\dots Y_{M+N}\rvert\big]+\E\big[\lvert Y_i\rvert^{p-M-N}\lvert Y_1\dots Y_{M}\rvert\big]\E[\lvert Y_{M+1}\dots Y_{M+N}\rvert]\big)\\
        &\qquad\qquad \leq 2T^{-(p-M-N)}\underbrace{\sup_{j=1,\dots,M+N}\E[\lvert Y_j\rvert^p]}_{=:\mathcal{T}_2}
    \end{align*}
    where we have used Lyapunov's inequality on the second term, followed by H\"older's inequality on both terms along with the fact that each $Y_j$ has finite $p$-th moment, courtesy of Lemmas~\ref{l:top_crit} and~\ref{l:crit_moments}. An analogous argument yields the same bound whenever $i>M$, and so we conclude that
    \[
        \lvert\Cov[Y_1\dots Y_M,Y_{M+1}\dots\, Y_{M+N}]\rvert\leq C(M+N) \big(T^{M+N-2}\mathcal{T}_1+T^{-(p-M-N)}\mathcal{T}_2\big).
    \]
    Setting $T=(\mathcal{T}_2/\mathcal{T}_1)^{\frac{1}{p-2}}$ yields the bound in the statement of the lemma.
\end{proof}

With these results in hand, it is straightforward to control the higher moments of the Euler characteristic using its exact additivity.

\begin{proof}[Proof of Theorem~\ref{t:higher_moments} (Euler characteristic)]
    We start by decomposing the Euler characteristic in an additive way. Given $a>0$, we let $\mathcal{L}_a:=(a/2,\dots,a/2)+a\Z^d$. We define an order $\prec$ on the points of $\mathcal{L}_a$ in such a way that for any $n\in\N$, if $i\in\Lambda_{2na}$ and $j\notin\Lambda_{2na}$ then $i\prec j$. (So roughly speaking, we order all of the points inside any given cube before moving on to points outside the cube.) For $i\in\mathcal{L}_a$, we set
    \[
        X_i=\overline{\Phi}(i+\Lambda_a)-\overline{\Phi}\Big((i+\partial\Lambda_a)\cap\bigcup_{j\prec i}(j+\partial\Lambda_a)\Big).
    \]
    We note that the final term above can be decomposed as a signed sum of the (centred) Euler characteristic applied to closed faces of $i+\partial\Lambda_a$. It then follows from additivity of the Euler characteristic that for any $n\in\N$,
    \begin{equation}\label{e:ec_higher1}
        \overline{\Phi}(\Lambda_{2na})=\sum_{i\in\mathcal{L}_a\cap\Lambda_{2na}}X_i.
    \end{equation}
    Given $R\geq 1$, we choose $a\in[1,2]$ such that $R/a$ is even. We apply Lemma~\ref{l:moment_condition} to conclude that for any $M,N\in\N$ and $i_1,\dots,i_M,j_1,\dots,j_N\in\mathcal{L}_a$
    \begin{align*}
        \lvert\Cov[X_{i_1}\dots X_{i_M},X_{j_1}\dots X_{j_N}]\rvert\leq C(M+N)\big(\min\{\lvert I\rvert,\lvert J\rvert\}\dist_\infty(I,J)^{-\beta+d}\big)^{\frac{p-M-N}{p-2}}D_p^\frac{M+N-2}{p-2}
    \end{align*}
    where $D_p:=\sup_{i\in\mathcal{L}_a}\E[\lvert X_i\rvert^p]$ and we recall that $I=\{i_1,\dots,i_M\}$ and $J=\{j_1,\dots,j_N\}$. Assuming now that $M+N\leq 2n$ and setting
    \[
        G(x,y)=2Cn\min\{x,y\}\qquad\text{and}\qquad F(x)=x^{-(\beta-d)\frac{p-2n}{p-2}}\times
        \begin{cases}
            D_p^{-\frac{2}{p-2}} &\text{if }D_p\geq 1\\
            D_p^{-\frac{2(p-2n)}{p(p-2)}} &\text{if }D_p<1,
        \end{cases}
    \]
    we may apply Theorem~\ref{t:moment_condition} to conclude that for any box $B=[a_1,b_1]\times\dots\times[a_d,b_d]$
    \begin{align*}
        \E\Big[\Big(\sum_{i\in B\cap\mathcal{L}_a}X_i\Big)^{2n}\Big]\leq C\lvert B\cap\mathcal{L}_a\rvert^n\Big(D_{2n}+\Big(\sum_{k=1}^{\infty}k^{nd-1-(\beta-d)\frac{p-2n}{p-2}}\Big)\times\max\Big\{D_p^{\frac{2n}{p}-\frac{2}{p-2}},D_p^{\frac{2n}{p}-\frac{2(p-2n)}{p(p-2)}}\Big\}\Big).
    \end{align*}
    Since $\beta>(n+1)d$ by assumption, taking $p$ sufficiently large ensures that the above sum is finite. The remaining terms in parentheses are then bounded uniformly over $a\in[1,2]$ since the Euler characteristic has finite moments of all orders (Lemmas~\ref{l:top_crit} and~\ref{l:crit_moments} and Assumption~\aref{a:smooth}). Setting $B=\Lambda_R$, we conclude by \eqref{e:ec_higher1}.
\end{proof}
\begin{remark}\label{r:ec_moments}
    We note that the above proof actually applies to all boxes, not only the cubes $\Lambda_R$. Specifically if $a$ is fixed, then there exists $C>0$ such that for any box $B=[a_1, b_1]\times\dots\times[a_d,b_d]$
    \[
        \E\Big[\Big(\sum_{i\in B\cap\mathcal{L}_a}X_i\Big)^{2n}\Big]\leq C\lvert B\rvert^{n}.
    \]
    We will make use of this observation later in the proof of Theorem~\ref{t:ec_lil}.
\end{remark}

The corresponding results for bounded excursion/level set functionals require more careful analysis due to the absence of (exact) additivity. To circumvent this, we use the two-scale approximate additivity from Section~\ref{s:additivity+stabilisation} as well as a novel bootstrap argument.

\begin{proof}[Proof of Theorem~\ref{t:higher_moments} (excursion/level-set functionals)]
    Let $\Phi$ be a bounded excursion/level-set functional and suppose that $a=a_R$  and $r=r_R$ satisfy conditions \eqref{e:scales}. We recall that $\chi_R$ denotes the centres of the mesoscopic boxes (of side length $r$) and $U_R$ denotes the closed complement in $\Lambda_R$ of such boxes. We recall also that $\Lambda_R$ is divided into cubes of side length $a$ by hyperplanes which we refer to as $a$-planes and that $U_{R,+a}$ denotes the union of all such cubes of side-length $a$ which intersect $U_R$ (including those which intersect only at the boundary) and are contained in $\Lambda_R$.
    
    By the decomposition in \eqref{e:additivity1}, for any $n\in\N$ there exists $C_n>0$ such that
    \begin{equation}\label{e:es_higher1}
        \E[\overline{\Phi}(\Lambda_R)^{2n}]\leq C_n\Big(\underbrace{\E\Big[\Big(\sum_{x\in\chi_R}\overline{\Phi}(x+\Lambda_r)\Big)^{2n}\Big]}_{=:\Circled{\mathrm{A}}}+\underbrace{\E\big[\overline{\Phi}^{(>a)}(\Lambda_R,U_R)^{2n}\big]}_{=:\Circled{\mathrm{B}}}+\underbrace{\E\big[\overline{\Phi}^{(<a)}(U_{R,+a},U_R)^{2n}\big]}_{=:\Circled{\mathrm{C}}}\Big).
    \end{equation}
    We will control each of these terms in turn.

    \underline{Term $\Circled{\mathrm{B}}$:} In the case $d=2$, we first observe that each excursion/level-set component which intersects parallel $a$-planes must have diameter at least $a$. For an excursion component this means that its boundary (which consists of $C^2$-smooth curves almost surely) has length at least $2a$, while for a level component this means that the component itself has length at least $2a$. By the definition of a bounded excursion/level-set functional in \eqref{e:es_functional2}, this means that
    \[
        \big\lvert\overline{\Phi}^{(>a)}(\Lambda_R,U_R)\big\rvert\leq\|\varphi\|_\infty\mathrm{Len}(\{f=\ell\}\cap\Lambda_R)/a
    \]
    where $\mathrm{Len}$ denotes length (i.e., one-dimensional volume measure). Since $f$ is assumed smooth, the latter quantity is known to have finite moments of all order \cite[Theorem~1.6]{al25}, it then follows from stationarity and the H\"older inequality (see the proof of Lemma~\ref{l:crit_moments}) that
    \begin{equation}\label{e:es_higher2}
        \E\big[\overline{\Phi}^{(>a)}(\Lambda_R,U_R)^{2n}\big]\leq C R^{4n}a^{-2n}
    \end{equation}
    for some $C>0$ depending only on $n$, $\varphi$ and the distribution of $f$.

    For arbitrary $d$, if Assumption~\aref{a:arm_decay} holds then by H\"older's inequality (or equivalently, Littlewood's $L^p$ inequality), for $p>2n$
    \begin{align*}
        \E\big[\overline{\Phi}^{(>a)}(\Lambda_R,U_R)^{2n}\big]&=\E\big[\overline{\Phi}^{(>a)}(\Lambda_R,U_R)^{2\frac{p-2n}{p-2}+p\frac{2n-2}{p-2}}\big]\\
        &\leq\Var\big[\Phi^{(>a)}(\Lambda_R,U_R)\big]^\frac{p-2n}{p-2}\E\big[\lvert\overline{\Phi}^{(>a)}(\Lambda_R,U_R)\rvert^{p}\big]^\frac{2n-2}{p-2}\\
        &\leq C\Big(R^dh(a/2)\Big)^\frac{p-2n}{p-2}(R^dar^{-1})^{p\frac{2n-2}{p-2}}
    \end{align*}
    where
    \[
        h(a)=\begin{cases}
            a^{-\frac{\beta}{2d+2}} &\text{if Assumption~\aref{a:cov_decay_pol} holds,}\\
            Ce^{-ca} &\text{if Assumption~\aref{a:cov_decay_exp} holds,}
        \end{cases}
    \]
    and the final inequality uses \eqref{e:q_additivity} to bound the variance and critical point estimates (Lemmas~\ref{l:boundary_contribution} and~\ref{l:crit_moments}) to bound the higher moment. For any $\epsilon>0$, by taking $p$ sufficiently large
    \begin{equation}\label{e:es_higher3}
        \E\big[\overline{\Phi}^{(>a)}(\Lambda_R,U_R)^{2n}\big]\leq C_\epsilon h(a)^{1-\epsilon}a^{2n-2+\epsilon}r^{-(2n-2)}R^{(2n-1)d+\epsilon}
    \end{equation}

    \underline{Term $\Circled{\mathrm{C}}$:} Turning to the right-most term in \eqref{e:es_higher1}, our goal is to decompose the contributions from small components into different domains of scale $a$ and then apply Theorem~\ref{t:moment_condition}. For $x\in\R^d$ let $\rho_a(x):=(a/2,\dots,a/2)+ax$. We define $\mathcal{A}_R$ to be the set of centres of $a$-boxes contained in $U_{R,+a}$, that is
    \[
        \mathcal{A}_R:=\rho_a(\Z^d)\cap U_{R,+a}.
    \]
    For $x\in\mathcal{A}_R$, let $\mathcal{C}_x$ denote the set of components $E$ of $\{f\geq\ell\}$ (or $\{f=\ell\}$) satisfying the following: (i) $E$ does not intersect parallel $a$-planes, (ii) $E$ intersects $U_R$, (iii) $x$ is the minimum element $y\in\mathcal{A}_R$ (with respect to the lexicographical ordering) such that $E$ intersects $y+\Lambda_a$. For $x\in\mathcal{A}_R$ we then define $D_a(x)$ to be the union of the $3^d$ $a$-boxes which are within distance $a$ of $x$ and equip it with the minimal stratification containing all open faces of these $a$-boxes. For $x\in\mathcal{A}_R$ and $i\in\Z^d$, we let
    \[
        \Psi(D_a(x))=\sum_{E\in\mathcal{C}_x}\varphi(E)\qquad\text{and}\qquad X_i=\overline{\Psi}(D_a(\rho_a(i))),
    \]
    where we recall that $\varphi$ is the excursion/level component functional from the definition of $\Phi$. Observe that by the definition of $\mathcal{C}_{x}$, $\Psi$ is a well-defined topological functional (to be consistent with our earlier definitions, we may take $\Psi(D)$ to be identically zero for any regular affine stratified set $D$ that is not of the form $D_a$ as described above). Moreover since all components contributing to $\Phi^{(<a)}(U_{R,+a},U_R)$ must be contained in a unique $\mathcal{C}_x$, we have
    \begin{equation}
        \overline{\Phi}^{(<a)}(U_{R,+a},U_R)=\sum_{x\in\mathcal{A}_R}\Psi(D_a(x))=\sum_{i\in\rho_a^{-1}(\mathcal{A}_R)}X_i.
    \end{equation}
    Since $\varphi$ is bounded, it follows from a simple Morse theory argument that $\Psi$ is a Lipschitz topological functional (see the proof of Lemma~\ref{l:exc_lipschitz}). Hence by Lemma~\ref{l:moment_condition}, for $M,N\in\N$, $p>M+N$ and $i_1,\dots,i_Mj_1,\dots,j_N\in\Z^d$
    \[
        \lvert\Cov[X_{i_1}\dots X_{i_M},X_{j_1}\dots X_{j_n}]\rvert\leq C(M+N)\big(\min\{\lvert I\rvert,\lvert J\rvert\}a^{-\beta}\dist(I,J)^{-\beta+d}\big)^{\frac{p-M-N}{p-2}}D_p^\frac{M+N-2}{p-2}
    \]
    where
    \[
        D_p:=\sup_{x\in\mathcal{A}_R}\E\Big[\lvert\overline{\Psi}(D_a(x))\rvert^p\Big].
    \]
    Setting
    \[
        G(x,y)=2Cn\min\{x,y\}^{\frac{p-2n}{p-2}}\qquad\text{and}\qquad F(x)=(a^{-\beta}x^{-\beta+d})^\frac{p-M-N}{p-2}\begin{cases}
            D_p^{-\frac{2}{p-2}} &\text{if }D_p\geq 1\\
            D_p^{-\frac{2(p-2n)}{p(p-2)}} &\text{if }D_p\leq 1
        \end{cases}
    \]
    we may then apply Theorem~\ref{t:moment_condition} to conclude that
    \begin{equation}
        \begin{aligned}
            \E\big[\overline{\Phi}^{(<a)}(U_{R,+a},U_R)^{2n}\big]\leq C \lvert \rho_a^{-1}(\mathcal{A}_R)\rvert^n\bigg(D_{2n}+a^{-\beta\frac{p-2n}{p-2}}\Big(&\sum_{k=1}^\infty k^{nd-1-(\beta-d)\frac{p-2n}{p-2}}\Big)\\
            &\times\max\Big\{D_p^{\frac{2n}{p}-\frac{2}{p-2}},D_p^{\frac{2n}{p}-\frac{2(p-2n)}{p(p-2)}}\Big\}\bigg).
        \end{aligned}
    \end{equation}
    By dividing $U_{R,+a}$ into $(R/(r+4a))^d$ regions that are translations of $\Lambda_{r+4a}\setminus\Lambda_{r-2a}$, it follows that $\lvert \rho_a^{-1}(\mathcal{A}_R)\rvert\leq Ca^{-d}(R/(r+4a))^dr^{d-1}a$.
    From our bound in terms of critical points (Lemmas~\ref{l:top_crit} and~\ref{l:crit_moments}), for $p\in\N$, $D_p\leq C_pa^{pd}$. Then using the assumption that $\beta>(n+1)d$ and taking $p$ sufficiently large,
    \begin{equation}\label{e:es_higher6}
        \E\big[\overline{\Phi}^{(<a)}(U_{R,+a},U_R)^{2n}\big]\leq C(a^{-d+1}r^{-1}R^d)^n\Big(a^{2nd}+a^{-\beta+2(n-1)d+\epsilon}\Big)\leq C a^{(d+1)n}r^{-n}R^{nd}.
    \end{equation}

    \underline{Term $\Circled{\mathrm{A}}$:} Finally we turn to the first term on the right-hand side of \eqref{e:es_higher1}. For $i\in\Z^d$, we now define $X_i=\overline{\Phi}(\rho_{r+4a}(i)+\Lambda_r)$. Then applying Lemma~\ref{l:moment_condition}, for $M,N\in\N$, $p>M+N$ and $i_1,\dots,i_M,j_1,\dots,j_N\in\Z^d$ we have
    \[
        \lvert\Cov[X_{i_1}\dots X_{i_M},X_{j_1}\dots X_{j_N}]\rvert\leq C(M+N)\Big(\min\{\lvert I\rvert,\lvert J\rvert\}r^{-\beta}\dist(I,J)^{-\beta+d}\Big)^{\frac{p-M-N}{p-2}}D_p^\frac{M+N-2}{p-2}
    \]
    where now $D_p:=\E[\lvert\overline{\Phi}_r\rvert^p]$. We can then define $G$ and $F$ as before with $a$ replaced by $r$ and apply Theorem~\ref{t:moment_condition} to deduce that
    \begin{align*}
        \E\Big[\Big(\sum_{x\in\chi_R}\overline{\Phi}(x+\Lambda_r)\Big)^{2n}\Big]&=\E\Big[\Big(\sum_{i\in\rho_{r+4a}^{-1}(\chi_R)}X_i\Big)^{2n}\Big]\\
        &\leq C(r^{-d}R^d)^n\bigg(D_{2n}+r^{-\beta\frac{p-2n}{p-2}}\max\Big\{D_p^{\frac{2n}{p}-\frac{2}{p-2}},D_p^{\frac{2n}{p}-\frac{2(p-2n)}{p(p-2)}}\Big\}\bigg).
    \end{align*}
    Using our bound based on critical points, for any $p\in\N$, $D_p\leq C_pr^{pd}$. Hence by taking $p$ sufficiently large and using the assumption that $\beta>(n+1)d$,
    \begin{equation}\label{e:es_higher7}
        \E\Big[\Big(\sum_{x\in\chi_R}\overline{\Phi}(x+\Lambda_r)\Big)^{2n}\Big]\leq CR^{nd}(r^{-nd}\E[\overline{\Phi}_r^{2n}]+r^{-\beta+(n-2)d+1}\big)\leq CR^{nd}\big(1+r^{-nd}\E[\overline{\Phi}_r^{2n}]\big).
    \end{equation}

    \underline{Conclusion:} Drawing together \eqref{e:es_higher1}, \eqref{e:es_higher3}, \eqref{e:es_higher6} and \eqref{e:es_higher7}, under Assumption~\aref{a:arm_decay} we have
    \begin{equation}\label{e:es_higher8}
        \E[\overline{\Phi}_R^{2n}]\leq CR^{nd}\Big(h(a)^{1-\epsilon}a^{2n-2+\epsilon}r^{-(2n-2)}R^{(n-1)d+\epsilon}+a^{(d+1)n}r^{-n}+1+r^{-nd}\E[\overline{\Phi}_r^{2n}]\Big).
    \end{equation}
    Under Assumption~\aref{a:cov_decay_pol}, $h(a)=a^{-\beta/(2d+2)}$. We then choose $a$ to be of order $r^{1/(d+1)}$ (that is, we choose the sequence $a_R$ so that there exist absolute constants $C_2>C_1>0$ such that for all $R$, $C_1r^{1/(d+1)}<a<C_2r^{1/(d+1)}$). Then for any $\epsilon>0$, we have
    \begin{equation}
        \E[\overline{\Phi}_R^{2n}]\leq CR^{nd}\Big(r^{-\frac{\beta+4d(d+1)(n-1)}{2(d+1)^2}+\epsilon}R^{(n-1)d+\epsilon}+1+r^{-nd}\E[\overline{\Phi}_r^{2n}]\Big).
    \end{equation}
    We then choose some $\alpha$ such that
    \[
        \frac{2d(d+1)^2(n-1)}{\beta+4d(d+1)(n-1)}<\alpha<1,
    \]
    which is possible by the assumption that $\beta>2d(d^2-1)(n-1)$. We then set $r$ to be of order $R^\alpha$ which ensures, for $\epsilon>0$ sufficiently small, that the first term in parentheses above tends to zero. If we denote $F(R)=r_R$ and $H(R)=\max\{1,R^{-nd}\E[\overline{\Phi}_R^{2n}]\}$ then we conclude that for a constant $C$ independent of $R$ and for all $R$ sufficiently large
    \[
        H(R)\leq C H(F(R))\qquad\text{where }F(R)=O(R^\alpha)\text{ as }R\to\infty.
    \]
    For any $\delta>0$, we choose $k$ such that $\alpha^knd<\delta$, then by iterating the above inequality $k$ times and using our bound in terms of critical points (Lemmas~\ref{l:top_crit} and~\ref{l:crit_moments}),
    \[
        R^{-nd}\E[\overline{\Phi}_R^{2n}]\leq C^k\max\Big\{1,\frac{\E[\overline{\Phi}_{c_kR^{\alpha^k}}^{2n}]}{c_kR^{nd\alpha^k}}\Big\}\leq C R^\delta
    \]
    for all $R$ sufficiently large. This completes the proof of the third statement of the theorem.

    In the case that Assumption~\aref{a:cov_decay_exp} holds, $h(a)=Ce^{-ca}$ and so if we choose $a_R$ to be of order $\log R$ and at least $\overline{C}\log R$ for a sufficiently large constant $\overline{C}$, then the first term in parentheses in \eqref{e:es_higher8} tends to zero as $R$ to $\infty$. If we choose $r\sim(\log R)^{d+2}$ then the same will be true of the second term in parentheses. Defining $H(R)$ as before, we conclude that
    \[
        H(R)\leq CH(F(R))\qquad\text{where }F(R)=O(\log^{d+2} (R))\text{ as }R\to\infty.
    \]
    Iterating this inequality $k+1$ times, and using our bound in terms of critical points, we find that
    \[
        R^{-nd}\E[\overline{\Phi}_R^{2n}]=O\big(\log^{(k)}(R)\big),
    \]
    as $R\to\infty$, proving the final statement of the theorem.

    In the case $d=2$ (without Assumption~\aref{a:arm_decay}), we replace \eqref{e:es_higher3} with \eqref{e:es_higher2} and so \eqref{e:es_higher8} becomes
    \begin{equation}\label{e:es_higher9}
        H(R)\leq C\big(a^{-2n}R^{2n}+a^{3n}r^{-n}+H(r)\big).
    \end{equation}
    where we define $H(R)$ as before. If we set $a\sim R^{1/3}$, $r\sim R^{1/2}$ and use our critical point bound on the final term, we have $H(R)=O(R^{(4/3)n})$ as $R\to\infty$. Suppose now that we know $H(R)=O(R^{\alpha n})$ for some $\alpha\in(4/5,2)$, then by setting $r\sim R^{6/(5\alpha+2)}$ and $a\sim r^{(\alpha+1)/3}$ (which is valid by our restrictions on $\alpha$) \eqref{e:es_higher9} yields $H(R)=O(R^{6\alpha/(5\alpha+2)})$. Noting that $4/5$ is the fixed point of $\alpha\mapsto\frac{6\alpha}{5\alpha+2}$, for any $\epsilon>0$ we can iterate this observation finitely many times to obtain that $H(R)=O(R^{(4/5)n+\epsilon})$, completing the proof of the second statement of the theorem.
\end{proof}

\subsection{Quantitative CLT}
We now turn to the quantitative central limit theorem for topological functionals (Theorem~\ref{t:qclt}). Our argument is based on the classical Berry-Esseen inequality. We recall that for $t\in\R$ and a random variable $X$, we denote $\varphi_t(X):=\E[e^{itX}]$.
\begin{theorem}[Berry-Esseen inequality {\cite[Theorem~2a]{ess45}}]\label{t:berry_esseen}
    Let $X$ and $Y$ be random variables such that $Y$ has a continuous probability density function $\rho_Y$. Then for any $T>0$
    \[
        d_\mathrm{Kol}(X,Y)\leq C\int_{-T}^T\Big\lvert\frac{\varphi_t(X)-\varphi_t(Y)}{t}\Big\rvert\;dt+\frac{C}{T}\sup_{s\in\R}\rho_Y(s),
    \]
    for an absolute constant $C>0$.
\end{theorem}

Our proof of the distributional CLT in Theorem~\ref{t:var+clt} showed convergence of the characteristic function of $\widetilde{\Phi}_R$ to that of a Gaussian variable. To use the Berry-Esseen inequality, we must quantify this convergence, which requires two more inputs. The first is an elementary estimate for i.i.d.\ sums:

\begin{lemma}\label{l:char_iid}
    Let $X_1,\dots,X_n$ be independent, identically distributed, centred random variables with finite third moments and let $S_n=\sum_{i=1}^nX_i$. Then for $\lvert t\rvert\leq\sqrt{n}\Var[X_1]/(4\E[\lvert X_1\rvert^3])$ we have
    \[
        \Big\lvert\varphi_t\Big(\frac{S_n}{\sqrt{n}}\Big)-\exp\Big(-\frac{t^2\Var[X_1]}{2}\Big)\Big\rvert\leq \frac{16\lvert t\rvert^3}{\sqrt{n}}\E[\lvert X_1\rvert^3]\exp\Big(-\frac{t^2\Var[X_1]}{3}\Big).
    \]
\end{lemma}
\begin{proof}
    Let $v:=\Var[X_1]$, $\tau:=\E[\lvert X_1\rvert^3]$ and let $Y$ be independent of $X_1$ and have the same distribution. We first suppose that 
    \begin{equation}\label{e:char_lemma1}
        \frac{n^{1/6}}{2\tau^{1/3}}\leq\lvert t\rvert\leq\frac{v\sqrt{n}}{4\tau}.
    \end{equation}
    By Taylor's theorem, for some $\theta\in(-1,1)$,
    \[
        \Big\lvert\varphi_t\Big(\frac{X_1}{\sqrt{n}}\Big)\Big\rvert^2=\varphi_t\Big(\frac{X_1-Y}{\sqrt{n}}\Big)=1-\frac{t^2}{2!n}\E[(X_1-Y)^2]+\theta\frac{t^3}{3!n^{3/2}}\E[\lvert X_1-Y\rvert^3]\leq \exp\Big(\frac{-vt^2}{n}+\frac{4\lvert t\rvert^3\tau}{3n^{3/2}}\Big)
    \]
    where we have used the elementary inequality $1+x\leq e^x$. Then using independence of $X_1,\dots,X_n$, the triangle inequality and the second inequality of \eqref{e:char_lemma1}
    \[
        \lvert\varphi_t(S_n/\sqrt{n})-\exp(-vt^2/2)\rvert\leq \exp\Big(-vt^2+\frac{4\lvert t\rvert^3\tau}{3\sqrt{n}}\Big)+\exp(-vt^2/2)\leq 2\exp(-vt^2/2).
    \]
    Given the first inequality in \eqref{e:char_lemma1}, this proves the required bound.

    We now suppose that $\lvert t\rvert<n^{1/6}/(2\tau^{1/3})$. By Taylor's theorem, there exists $\theta\in\mathbb{C}$ with $\lvert\theta\rvert\leq 1$ such that
    \[
        \varphi_t(X_1/\sqrt{n})-1=-\frac{vt^2}{2n}+\theta\frac{\lvert t\rvert^3\tau}{6n^{3/2}}.
    \]
    Making use of our upper bound on $\lvert t\rvert$ and Lyapunov's inequality, we have
    \begin{align*}
        \lvert\varphi_t(X_1/\sqrt{n})-1\rvert^2\leq 2\frac{v^2t^4}{4n^2}+2\frac{t^6\tau^2}{36n^3}\leq\lvert t\rvert^3\tau\Big(\frac{v^2n^{1/6}}{4\tau^{4/3}n^2}+\frac{n^{1/2}}{144n^3}\Big)\leq\frac{37\tau\lvert t\rvert^3}{144n^{3/2}}.
    \end{align*}
    Applying the upper bound on $\lvert t\rvert$ once more shows that the latter quantity is uniformly bounded by $37/1152$. We may therefore invoke the elementary inequality $\lvert\log(1+z)-z\rvert\leq\lvert z\rvert^2$ which is valid for $z\in\mathbb{C}$ with $\lvert z\rvert\leq 1/2$ (using the principal logarithm), to see that
    \begin{align*}
        \log\varphi_t(S_n/\sqrt{n})=n\log \varphi_t(X_1/\sqrt{n})=-\frac{vt^2}{2}+\theta\frac{\lvert t\rvert^3\tau}{6\sqrt{n}}+\theta^\prime\frac{37\tau\lvert t\rvert^3}{144\sqrt{n}}
    \end{align*}
    for some $\theta^\prime\in\mathbb{C}$ satisfying $\lvert\theta^\prime\rvert\leq 1$. Setting $\overline{\theta}=(\theta/3)+37\theta^\prime/72$, we use the mean value theorem to conclude that
    \[
        \lvert\varphi_t(S_n/\sqrt{n})-\exp(-vt^2/2)\rvert=\exp(-vt^2/2)\Big\lvert\exp\Big(\overline{\theta}\frac{\lvert t\rvert^3\tau}{2\sqrt{n}}\Big)-1\Big\rvert\leq \exp(-vt^2/2)\frac{\lvert t\rvert^3\tau}{2\sqrt{n}}\exp\Big(\frac{\lvert t\rvert^3\tau}{2\sqrt{n}}\Big).
    \]
    By our assumption on $\lvert t\rvert$, the final exponential here is at most $\exp(1/16)$, which completes the proof in this case.
\end{proof}

The second input is a bound on the third moment of excursion/level-set functionals which follows from interpolating our previous results. Specifically we claim that, in the setting of Theorem~\ref{t:qclt}, as $R\to\infty$
\begin{equation}\label{e:third_moment}
    \E[\lvert\overline{\Phi}_R\rvert^3]=O(H(R))
\end{equation}
where
\[
    H(R)=\begin{cases}
        R^{\frac{3d}{2}} &\text{if $\Phi$ is the Euler characteristic and Assumption~\aref{a:smooth} holds,}\\
        R^{\frac{3d}{2}+\epsilon} &\parbox[t]{0.6\textwidth}{for any $\epsilon>0$ if $\beta>2d(d^2-1)$ and Assumptions~\aref{a:smooth} and~\aref{a:arm_decay} hold,}\\
        R^{\frac{3d}{2}}\log(R) &\parbox[t]{0.6\textwidth}{if Assumptions~\aref{a:smooth}, ~\aref{a:arm_decay} and~\aref{a:cov_decay_exp} hold,}\\
        3d &\text{otherwise.}
    \end{cases}
\]
The final case here just follows from our bound in terms of critical points (Lemmas~\ref{l:top_crit} and~\ref{l:crit_moments}) while the other cases follow from interpolation (Littlewood's $L^p$-inequality) applied to our bounds on the second and fourth moments of $\lvert\overline{\Phi}_R\rvert$ in Theorems~\ref{t:var+clt} and~\ref{t:higher_moments} respectively.

\begin{proof}[Proof of Theorem~\ref{t:qclt}]
    Let $\sigma$ be as specified in Theorem~\ref{t:var+clt} and $Z\sim\mathcal{N}(0,1)$ and define $\sigma_R^2=\Var[\Phi_R]$. We let $1\ll a\ll r\ll R$ satisfy \eqref{e:scales}. By the triangle inequality and the Berry-Esseen inequality (Theorem~\ref{t:berry_esseen})
    \[
    \begin{alignedat}{2}
        d_\mathrm{Kol}\big(\widetilde{\Phi}_R,\sigma Z\big)&\leq d_\mathrm{Kol}\Big(\frac{\sigma_r}{(r+4a)^{d/2}} Z,\sigma Z\Big)&\Big\}\Circled{\mathrm{A}}\\
        &+C\int_{-T}^T\lvert t\rvert^{-1}\Big\lvert\varphi_t\big(\widetilde{\Phi}_R\big)-\varphi_t\Big(R^{-d/2}\sum_{x\in\chi_R}\overline{\Phi}(x+\Lambda_r)\Big)\Big\rvert\;dt&\Big\}\Circled{\mathrm{B}}\\
        &+C\int_{-T}^T\lvert t\rvert^{-1}\Big\lvert\varphi_t\Big(R^{-d/2}\sum_{x\in\chi_R}\overline{\Phi}(x+\Lambda_r)\Big)-\varphi_t(R^{-d/2}\overline{\Phi}_r)^{\lvert\chi_R\rvert}\Big\rvert\;dt&\Big\}\Circled{\mathrm{C}}\\
        &+C\int_{-T}^T\lvert t\rvert^{-1}\Big\lvert\varphi_t(R^{-d/2}\overline{\Phi}_r)^{\lvert\chi_R\rvert}-\varphi_t\Big(\frac{\sigma_r}{(r+4a)^{d/2}}Z\Big)\Big\rvert\;dt&\Big\}\Circled{\mathrm{D}}\\
        &+\frac{C}{T}\frac{1}{\sqrt{2\pi}}\frac{(r+4a)^{d/2}}{\sigma_r}.
    \end{alignedat}
    \]
    We consider each of these terms in turn.

    Letting $P$ denote the standard normal cumulative distribution function, by the mean value theorem we have
    \[
        \Circled{\mathrm{A}}=\sup_{x\in\R}\Big\lvert P(x/\sigma)-P\big((r+4a)^{d/2}x/\sigma_r\big)\Big\rvert\leq \sup_{x\in\R}\lvert P^\prime(x)\rvert\sup_{y}y^{-2}\Big\lvert\sigma-\frac{\sigma_r}{(r+4a)^{d/2}}\Big\rvert
    \]
    where the middle supremum is taken over all $y$ between $\sigma$ and $\sigma_r/(r+4a)^{d/2}$. For $R$ sufficiently large, $\sigma_r/(r+4a)^{d/2}$ is uniformly bounded away from zero (since it converges to $\sigma>0$ by Theorem~\ref{t:var+clt}) and so
    \begin{equation}\label{e:qclt1}
        \Circled{\mathrm{A}}\leq C \Big(\Big\lvert\sigma^2-\frac{\sigma_r^2}{r^d}\Big\rvert+\Big\lvert\frac{\sigma_r^2}{r^d}-\frac{\sigma_r^2}{(r+4a)^d}\Big\rvert\Big)\leq C(r^{-\gamma}+a/r)
    \end{equation}
    where $\gamma$ is specified in Theorem~\ref{t:variance_limit_quant}.

    Using the inequality $\lvert e^{ix}-e^{iy}\rvert\leq\lvert x-y\rvert$ for $x,y\in\R$ and our additivity bounds (Proposition~\ref{p:additivity} for the Euler characteristic and Proposition~\ref{p:q_additivity} otherwise)
    \begin{equation}\label{e:qclt2}
        \Circled{\mathrm{B}}\leq C\int_{-T}^TR^{-d/2}\E\Big[\Big\lvert\overline{\Phi}_R-\sum_{x\in\chi_R}\overline{\Phi}(x+\Lambda_r)\Big\rvert\Big]\;dt\leq C^\prime T \big(a^{1/2}r^{-1/2}+F(R)\big),
    \end{equation}
    where
    \[
        F(R):=
        \begin{cases}
            0 &\text{if $\Phi$ is the Euler characteristic,}\\
            R^{-\epsilon}&\text{if Assumption~\aref{a:arm_decay_critical} holds,}\\
            a^{-\beta/(4d+4)}&\text{if Assumptions~\aref{a:arm_decay} and~\aref{a:cov_decay_pol} hold,}\\
            e^{-ca}&\text{if Assumptions~\aref{a:arm_decay} and~\aref{a:cov_decay_exp} hold.}
        \end{cases}
    \]

    By \eqref{e:clt2} and Assumption~\aref{a:cov_decay_pol} or~\aref{a:cov_decay_exp}
    \begin{equation}\label{e:qclt3}
        \Circled{\mathrm{C}}\leq C\int_{-T}^T\lvert t\rvert\;dt \int_{\R^d\setminus\Lambda_a}\widetilde{K}(y)\;dy\leq C T^2 G(a),
    \end{equation}
    where
    \[
        G(a):=\begin{cases}
            e^{-ca}&\text{if Assumption~\aref{a:cov_decay_exp} holds,}\\
            a^{-\beta+d} &\text{otherwise.}
        \end{cases}
    \]

    Turning to $\Circled{\mathrm{D}}$, we may apply the bound on characteristic functions of a sum of i.i.d.\ variables (Lemma~\ref{l:char_iid} with $X_i\overset{d}{=}(r+4a)^{-d/2}\overline{\Phi}_r$) to say that for $T\leq\sqrt{\lvert\chi_R\rvert}\sigma_r^2(r+4a)^{d/2}/(4\E[\lvert\overline{\Phi}_r\rvert^3])$
    \begin{equation}\label{e:qclt4}
        \Circled{\mathrm{D}}\leq C\int_{-T}^T\frac{16 t^2}{\sqrt{\lvert\chi_R\rvert}}\frac{\E[\lvert\overline{\Phi}_r\rvert^3]}{(r+4a)^{3d/2}}\exp\Big(-t^2\frac{\sigma_r^2}{3(r+4a)^{d}}\Big)\;dt\leq C\frac{\E[\lvert\overline{\Phi}_r\rvert^3]}{r^dR^{d/2}}\leq C H(r)r^{-d}R^{-d/2}
    \end{equation}
    for some $C,C^\prime>0$ and all $R\geq 1$, where the final inequality uses \eqref{e:third_moment}.

    Combining \eqref{e:qclt1}-\eqref{e:qclt4}, we see that for all $R$ sufficiently large and $0<T\leq R^{d/2}\sigma_r^2/(4\E[\lvert\overline{\Phi}_r\rvert^3])$
    \[
        d_\mathrm{Kol}(\widetilde{\Phi}_R,\sigma Z)\leq C\Big(r^{-\gamma}+Ta^{1/2}r^{-1/2}+TF(R)+T^2G(a)+H(r)r^{-d}R^{-d/2}+T^{-1}\Big)
    \]
    For the first statement of the theorem, we may choose $r\ll R$ and then $a,T\ll r$ to ensure that this expression decays polynomially in $R$ (with some small exponent).

    We now suppose that Assumption~\aref{a:smooth} holds, $\beta>2d(d^2-1)$ and that either $\Phi$ is the Euler characteristic or Assumption~\aref{a:arm_decay} holds. We set $T=R^{d/2}\sigma_r^2/(4\E[\lvert\overline{\Phi}_r\rvert^3])$. By Lyapunov's inequality and Theorem~\ref{t:var+clt}, since $\sigma>0$ there exists $c>0$ such that $\E[\lvert\overline{\Phi}_r\rvert^3]\geq \Var[\Phi_r]^{3/2}\geq cr^{3d/2}$ for all $R$ sufficiently large. Combining this with the bound in \eqref{e:third_moment} we have $C^{-1}R^{d/2}r^{d}/H(r)\leq T\leq CR^{d/2}r^{-d/2}$. Since $\gamma=1$, the term $r^{-\gamma}$ will be negligible above. Combining these observations we have
    \[
        d_\mathrm{Kol}(\widetilde{\Phi}_R,\sigma Z)\leq C\Big(a^{1/2}r^{-(d+1)/2}R^{d/2}+F(R)r^{-d/2}R^{d/2}+G(a)r^{-d}R^{d}+H(r)r^{-d}R^{-d/2}\Big).
    \]
    If $\Phi$ is the Euler characteristic, then we set $a\sim R^{\alpha_1}$ and $r\sim R^{\alpha_2}$ where
    \[
        \alpha_1=\frac{3d}{2(2d+1)\beta-d(4d-1)}\qquad\text{and}\qquad\alpha_2=\frac{d(4\beta-4d+3)}{2(2d+1)\beta-d(4d-1)}.
    \]
    It is easily verified that $0<\alpha_1<\alpha_2<1$ so that these choices are consistent with the restrictions $1\ll a\ll r\ll R$. Substituting these in above, it follows that
    \[
        d_\mathrm{Kol}(\widetilde{\Phi}_R,\sigma Z)=O\Big(R^{-\frac{d(\beta-d)}{2(2d+1)\beta-d(4d-1)}}\Big)
    \]
    as required. If instead $\Phi$ is a bounded excursion/level-set functional (and Assumption~\aref{a:arm_decay} holds), then we set
    \[
        \alpha_1=\frac{4d(d+1)}{(2d+1)\beta+4d(d+1)}\qquad\text{and}\qquad\alpha_2=\frac{2d(\beta+2d+2)}{(2d+1)\beta+4d(d+1)}
    \]
    and substitute this into our bound on the Kolmogorov distance, which yields
    \[
        d_\mathrm{Kol}(\widetilde{\Phi}_R,\sigma Z)=O\Big(R^{-\frac{d}{2}\frac{\beta}{(2d+1)\beta+4d(d+1)}+\epsilon}\Big).
    \]
    Finally if Assumption~\aref{a:cov_decay_exp} holds then we choose $a\sim\log R$ with $a\geq \overline{C}\log R$ for a large constant $\overline{C}>0$ and $r\sim (R\sqrt{\log R})^\frac{2d}{2d+1}$. Taking $\overline{C}$ sufficiently large then yields the final statement of the theorem.
\end{proof}

\subsection{Law of the iterated logarithm}
We complete this section with a proof of the law of the iterated logarithm (LIL) for the Euler characteristic. We generally follow the method of \cite{lw08}, which establishes the LIL for positively dependent sums using a quantitative CLT to control tail behaviour. The main differences are that we work with a multidimensional sum and that some complications arise from the contribution of the boundary of a domain to the Euler characteristic.

Recall from the proof of Theorem~\ref{t:higher_moments} the lattice $\mathcal{L}_a=(a/2,\dots,a/2)+a\Z^d$ defined for $a>0$ and the random variables $X_i$ for $i\in\mathcal{L}_a$ defined so that for any $n\in\N$
\[
    \overline{\Phi}(\Lambda_{2na})=\sum_{i\in\mathcal{L}_a\cap\Lambda_{2na}}X_i.
\]
Throughout this subsection we fix $a=1/2$ and so abbreviate $\mathcal{L}:=\mathcal{L}_{1/2}$. Given $D\subset\R^d$ we set $S(D)=\sum_{i\in\mathcal{L}\cap D}X_i$. We recall that $\mathcal{U}$ denotes the set of blocks in $\Z^d$ of the form $[a_1,b_1]\times\dots[a_d,b_d]\cap\Z^d$.

Our first input is a version of a theorem due to M\'oricz which controls the moments of the supremum of a discrete random field:
\begin{theorem}[{\cite[Theorem~2.1.2]{bs07}}]\label{t:Moricz}
    Let $(X_j)_{j\in\Z^d}$ be a collection of random variables and for $U\in\mathcal{U}$ define $S(U)=\sum_{j\in U}X_j$. Suppose that there exists $C>0$, $p\geq 1$ and $\alpha>1$ such that
    \[
        \forall U\in\mathcal{U},\qquad\qquad \E[\lvert S(U)\rvert^p]\leq C\lvert U\rvert^\alpha,
    \]
    then there exists $C^\prime>0$ such that
    \[
        \forall U\in\mathcal{U},\qquad\qquad \E[M(U)^p]\leq C^\prime\lvert U\rvert^\alpha,
    \]
    where $M(U):=\sup_{W\subset U, W\in\mathcal{U}}\lvert S(W)\rvert$.
\end{theorem}

We next state an elementary lemma for controlling tail behaviour of a sequence using bounds in Kolmogorov distance:

\begin{lemma}\label{l:gaussian_tail}
    Let $Y_k$ be a sequence of random variables such that
    \[
        \sum_{k=1}^\infty d_\mathrm{Kol}(Y_k,\sigma Z)<\infty
    \]
    for some $\sigma>0$ and $Z\sim\mathcal{N}(0,1)$. Let $(a_k)_{k\in\N}$ be non-decreasing and positive, then
    \[
        \sum_{k=1}^\infty\P(Y_{k}>\sigma a_k)<\infty\qquad\text{if and only if}\qquad\sum_{k=1}^\infty\frac{1}{a_k}\exp\Big(-\frac{1}{2}a_k^2\Big)<\infty
    \]
\end{lemma}
\begin{proof}
    By the definition of Kolmogorov distance,
    \[
        \sum_{k=1}^\infty\big\lvert\P(Y_{k}>\sigma a_k)-\P(Z>a_k)\big\rvert\leq \sum_{k=1}^\infty d_\mathrm{Kol}(Y_k,\sigma Z)<\infty.
    \]
    A standard Gaussian tail inequality states that for any $x>0$
    \[
        \frac{1}{\sqrt{2\pi}}\frac{1}{x}\Big(1-\frac{1}{x^2}\Big)\exp(-x^2/2)\leq \P(Z>x)\leq\frac{1}{\sqrt{2\pi}}\frac{1}{x}\exp(-x^2/2).
    \]
    The equivalence follows easily from these two estimates.
\end{proof}

We will apply the above estimate to the renormalised Euler characteristic on suitable domains. For cube domains, the bounds on Kolmogorov distance will be taken from the quantitative CLT, but we also need such bounds for annular domains:

\begin{theorem}\label{t:qclt_annulus}
    Let $f$ satisfy the conditions of Theorem~\ref{t:ec_lil} and $\Phi$ be the Euler characteristic. Fixing an integer $N>2$, for $k\in\N$ we define $A_k=\Lambda_{N^k}\setminus\Lambda_{N^{k-1}+\lfloor N^{k/2}\rfloor}$, (where $\lfloor x\rfloor$ denotes the integer part of $x$) then there exists $C,\epsilon>0$ such that for all $k$
    \[
        d_\mathrm{Kol}\big(S(A_k),\sigma Z\big)\leq CN^{-\epsilon k},
    \]
    where $\sigma$ is defined in Theorem~\ref{t:var+clt} and $Z\sim\mathcal{N}(0,1)$.
\end{theorem}
\begin{proof}
    This result follows from essentially the same argument given to prove Theorem~\ref{t:qclt}, so we just highlight the points of difference. By additivity of the Euler characteristic (Lemma~\ref{l:ec_additivity}) it follows that
    \begin{equation}\label{e:ec_lil_additivity1}
        S(A_k)=\Phi(\overline{A}_k)-\Phi(\partial\Lambda_{N^{k-1}+\lfloor N^{k/2}\rfloor}).
    \end{equation}
    Working towards an additivity estimate, we can tile $A_k$ by cubes of side length $r\ll N^k$ separated at scale $a\ll r$. Denoting the centres of these mesoscopic cubes by $\chi^\prime_k$, the proof of Proposition~\ref{p:additivity} can be adjusted to yield the estimate
    \[
    \Var\bigg[\Phi(\overline{A}_k)-\sum_{x\in\chi^\prime_k}\Phi(x+\Lambda_r)\bigg]\leq Car^{-1}N^{kd}.
    \]
    By quasi-association, the variance of the last term in \eqref{e:ec_lil_additivity1} is at most $CN^{k(d-1)}$, so we deduce that
    \begin{equation}\label{e:ec_lil_additivity2}
        \Var\bigg[S(A_k)-\sum_{x\in\chi_{N,k}}\Phi(x+\Lambda_r)\bigg]\leq Car^{-1}N^{kd}.
    \end{equation}
    The result then follows from the proof of Theorem~\ref{t:qclt} after replacing $\lvert\Lambda_R\rvert$ with $\lvert A_k\rvert$ and $\chi_R$ with $\chi^\prime_k$ and using \eqref{e:ec_lil_additivity2} in place of Proposition~\ref{p:additivity}.
\end{proof}

\begin{proof}[Proof of Theorem~\ref{t:ec_lil}]
     We prove only the statement for the limit superior; the statement for the limit inferior follows from the same argument applied to $-\Phi$. It is sufficient to show that for any $\epsilon>0$, with probability one
     \begin{equation}\label{e:ec_lil1}
         \limsup_{n\to\infty}\frac{S(\Lambda_n)}{\sqrt{2n^d\log\log n}}\leq (1+\epsilon)\sigma\qquad\text{and}\qquad\limsup_{n\to\infty}\frac{S(\Lambda_n)}{\sqrt{2n^d\log\log n}}\geq (1-4\epsilon)\sigma.
     \end{equation}
     
     \underline{Upper bound:} Turning to the first of these statements, our strategy is to use the quantitative CLT to control behaviour on a (stretched-exponential) subsequence and then use M\'oricz' theorem to bound deviations on the remaining scales. Given $\epsilon\in(0,1)$ we choose $\alpha\in(0,1)$ such that $\alpha(1+\epsilon)^2>1$. We then set
    \[
        n_k=\lfloor e^{k^\alpha}\rfloor,\qquad Y_k=n_k^{-d/2}S(\Lambda_{n_k}),\qquad a_k=\sqrt{2(1+\epsilon)^2\log\log n_k}\qquad\text{and}\qquad L(n)=\sqrt{2n^d\log\log n},
    \]
    and observe that, by our choice of parameters and Theorem~\ref{t:qclt}
    \[
        \sum_{k=2}^\infty\frac{1}{a_k}\exp\Big(-\frac{1}{2}a_k^2\Big)\leq C\sum_{k=2}^\infty k^{-\alpha(1+\epsilon)^2}<\infty,\qquad\text{and}\qquad\sum_{k=1}^\infty d_\mathrm{Kol}(Y_k,\sigma Z)\leq C\sum_{k=1}^\infty n_k^\eta<\infty.
    \]
    Lemma~\ref{l:gaussian_tail} then yields
    \[
        \sum_{k=2}^\infty\P(S(\Lambda_{n_k})>(1+\epsilon)\sigma L(n_k))<\infty,
    \]
    and so by Borel-Cantelli, with probability one,
    \begin{equation}\label{e:ec_lil2}
        \limsup_{k\to\infty}\frac{S(\Lambda_{n_k})}{L(n_k)}\leq(1+\epsilon)\sigma.
    \end{equation}
    To control behaviour on the remaining subsequence, we consider $M_k:=\sup_{n_k\leq n<n_{k+1}}\lvert S(\Lambda_n)-S(\Lambda_{n_k})\rvert$. We partition $\mathcal{L}\cap\Lambda_{n_{k+1}}\setminus\Lambda_{n_k}$ into blocks $B_1,\dots,B_J\in\mathcal{U}$ such that $J$ depends only on the dimension $d$ and for each $i$, $\lvert B_i\rvert\leq Cn_{k+1}^{d-1}(n_{k+1}-n_k)$. (One way of constructing such a partition is to use the hyperplanes passing through the faces of $\Lambda_{n_k}$ to separate $\Lambda_{n_{k+1}}\setminus\Lambda_{n_k}$.) It then follows that
    \[
        M_k\leq\sum_{i=1}^JM(B_i)
    \]
    where we recall that $M(B_i)=\sup_{W\subset B_i,W\in\mathcal{U
    }}\lvert S(W)\rvert$. Then applying Theorem~\ref{t:Moricz} to this expression (which is justified by Remark~\ref{r:ec_moments}) we conclude that for any even integer $p>2$
    \[
        \E[M_k^p]\leq Cn_{k+1}^\frac{p(d-1)}{2}(n_{k+1}-n_k)^{\frac{p}{2}}.
    \]
    Choosing $p>2/(1-\alpha)$, we have
    \[
        \sum_{k=2}^\infty\frac{\E[M_k^p]}{L(n_k)^p}\leq C\sum_{k=2}^\infty\Big(\frac{n_{k+1}}{n_k}\Big)^\frac{p(d-1)}{2}\Big(\frac{n_{k+1}-n_k}{n_k}\Big)^\frac{p}{2}\frac{1}{(\log\log n_k)^{p/2}}\leq C\sum_{k=2}^\infty k^{-\frac{p(1-\alpha)}{2}}\frac{1}{(\log k)^{p/2}}<\infty,
    \]
    where we have used the facts that
    \[
        \sup_{k\in\N}\frac{n_{k+1}}{n_k}<\infty\qquad\text{and}\qquad\frac{e^{(k+1)^\alpha}-e^{k^\alpha}}{e^{k^\alpha}}\leq Ck^{-(1-\alpha)}
    \]
    with the latter following from the mean-value theorem. By Borel-Cantelli, $M_k/L(n_k)\to 0$ almost surely as $k\to\infty$. Combined with \eqref{e:ec_lil2} we deduce the first inequality in \eqref{e:ec_lil1}.

    \underline{Lower bound:} For some integer $N>2$, we set $m_k=N^k-N^{k-1}-\lfloor N^{k/2}\rfloor$,
    \[
        E_k(\epsilon)=\big\{S(A_k)>(1-\epsilon)\sigma L(m_k)\big\}\qquad\text{and}\qquad E_k^\prime=\Big\{S\big(\Lambda_{N^{k-1}+\lfloor N^{k/2}\rfloor}\big)>-2\sigma L\big(N^{k-1}+\lfloor N^{k/2}\rfloor\big)\Big\}.
    \]
    Given $\epsilon>0$, if $N$ is fixed sufficiently large, then for all $k\in\N$, the event $E_k(3\epsilon)\cap E_k^\prime$ implies
    \[
        S(\Lambda_{N^k})>(1-4\epsilon)\sigma L(N^k).
    \]
    From the proof of the first part of the theorem (applied to $-\Phi$), there exists almost surely a random $k_0$ such that $E_k^\prime$ holds for all $k>k_0$. Hence the second part of \eqref{e:ec_lil1} will be proven if we can show that
    \[
        \P(E_k(3\epsilon)\;\mathrm{i.o.})=1.
    \]
    To this end, for each $k$ we choose a Lipschitz function $F_k:\R\to[0,1]$ such that
    \[
        F_k(x)=\begin{cases}
            0 &\text{if }x<(1-3\epsilon)\sigma L(m_k)\\
            1 &\text{if }x>(1-2\epsilon)\sigma L(m_k),
        \end{cases}
        \qquad\text{and}\qquad\sup_{k}\lVert F_k\rVert_\mathrm{Lip}<\infty.
    \]
    It is therefore sufficient to show that $\sum_{k=1}^\infty F_k(S(A_k))=\infty$ almost surely. Letting $W_k=F_k(S(A_k))$, by the triangle inequality, Markov's inequality and the bounds $0\leq F_k\leq 1$, for any $n\in\N$
    \begin{align*}
        \P\bigg(\sum_{k=1}^\infty W_k\leq\frac{1}{2}\sum_{k=1}^n\E[W_k]\bigg)&\leq\P\bigg(\bigg\lvert\sum_{k=1}^nW_k-\E[W_k]\bigg\rvert\geq\frac{1}{2}\sum_{k=1}^n\E[W_k]\bigg)\\
        &\leq \frac{4\Var\big[\sum_{k=1}^nW_k\big]}{\big(\sum_{k=1}^n\E[W_k]\big)^2}\leq \frac{4}{\sum_{k=1}^n\E[W_k]}+\frac{8\sum_{k=1}^n\sum_{j=k+1}^n\lvert\Cov[W_k,W_j]\rvert}{\big(\sum_{k=1}^n\E[W_k]\big)^2}.
    \end{align*}
    From this bound, the theorem will follow if we can show that
    \begin{equation}\label{e:ec_lil3}
        \sum_{k=1}^\infty\sum_{j=k+1}^\infty\lvert\Cov[W_k,W_j]\rvert<\infty\qquad\text{and}\qquad\sum_{k=1}^\infty\E[W_k]=\infty.
    \end{equation}
    By quasi-association (Corollary~\ref{c:quasi_association}), the fact that the $F_k$ are uniformly Lipschitz and our bound on the covariance function,
    \[
        \sum_{j=k+1}^\infty\lvert\Cov[W_k,W_j]\rvert\leq C\int_{\Lambda_{N^k+2}\times\R^d\setminus\Lambda_{N^k+\lfloor N^{k/2}\rfloor-2}}\widetilde{K}(x-y)\;dxdy\leq CN^{kd-(\beta-d)k/2}.
    \]
    By assumption we may choose $\beta$ sufficiently large to ensure that this is summable over $k$, verifying the first part of \eqref{e:ec_lil3}.

    Finally, setting
    \[
        a_k=\sqrt{2(1-2\epsilon)^2\log\log m_k}\qquad\text{and}\qquad Y_k=\lvert A_k\rvert^{-1/2}S(A_k),
    \]
    for $k_0$ sufficiently large, we have
    \begin{align*}
        \sum_{k=k_0}^\infty a_k^{-1}\exp\big(-(1/2)a_k^2\big)\geq \sum_{k=k_0}^\infty\exp\big(-(1-\epsilon)^2\log\log m_k\big)\geq\sum_{k=k_0}^\infty (k\log N)^{-(1-\epsilon)^2}=\infty.
    \end{align*}
    By Theorem~\ref{t:qclt_annulus}, $\sum_kd_\mathrm{Kol}(Y_k,\sigma Z)<\infty$ and so we may apply Lemma~\ref{l:gaussian_tail} to deduce that
    \[
        \sum_{k=1}^\infty\P(E_k(2\epsilon))=\sum_{k=1}^\infty\P(Y_k>\sigma a_k)=\infty.
    \]
    From our definition of $F_k$, we have $F_k(S(A_k))\geq\ind_{E_k(2\epsilon)}$. Combining these final two observations verifies the second part of \eqref{e:ec_lil3}, completing the proof of the theorem.
\end{proof}

\section{Volume of the unbounded component}
In this section we turn our attention to the final non-local functional that we consider: the volume of the unbounded component of the excursion set. We will prove the limiting results in Theorem~\ref{t:vol} by showing that the functional is associated and appealing to established results for (discrete) associated fields.

We recall that for a field $f$, level $\ell\in\R$ and bounded Borel set $A\subset\R^d$, $\{f\geq\ell\}_\infty$ denotes the union of all unbounded components of $\{f\geq\ell\}$ and $\Vol_\infty[A]$ denotes the $d$-dimensional Lebesgue measure of $A\cap\{f\geq\ell\}_\infty$. For $v\in\Z^d$, we then define
\[
    V_v=\Vol_\infty(v/2+[0,1/2]^d).
\]
We note that this slightly unusual indexing allows us to tile $\Lambda_n=[-n/2,n/2]^d$ using cubes of side length $1/2$ while also appealing to textbook results stated for processes indexed by $\Z^d$.

\begin{lemma}\label{l:assoc}
    Let $f$ be a $C^2$ Gaussian field with non-negative covariance function and $\ell\in\R$, then $(V_v)_{v\in\Z^d}$ is associated (as defined at the start of Section~\ref{ss:proof_outline}).
\end{lemma}
\begin{proof}
    It was shown by Pitt \cite{pit82} that a collection of Gaussian variables is associated if (and only if) all of their correlations are non-negative. Hence, by assumption, $(f(x))_{x\in\R^d}$ is associated. It follows from the definition of association that any collection of measurable, coordinate-wise non-decreasing functions that depend on the field at finitely many points is also associated. Therefore we need only show that each $V_v$ can be approximated (say, in $L^1$) by a coordinate-wise non-decreasing function of $f$ which depends on only finitely many points.
    
    To see such an approximation, first note that $V_v$ is the limit (in $L^1$) as $R\to\infty$ of the volume of all excursion components in $v/2+[0,1/2]^d$ that are connected to $\partial\Lambda_R$ in $\{f\geq\ell\}$. For any fixed $R$, this volume can be approximated by considering the analogous quantity for a discretisation of $f$. That is, one restricts $f$ to $\epsilon\Z^d\cap\Lambda_{R+1}$, says that if nearest neighbours are both above/below $\ell$ then they belong to the same component of the excursion set and approximates the volume of a component by $\epsilon^{-d}$ times the number of points in that component. Since the level sets of $f$ are $C^2$-smooth almost surely, this discretised estimate will converge to the requisite volume as $\epsilon\searrow 0$.
\end{proof}

Many results follow for discrete associated fields provided that correlations decay sufficiently quickly. In our case, this can be proven using the covariance formula for topological events. We recall that for $x\in\R^d$, $E_x:=\{x\in\{f\geq\ell\}_\infty\}$.

\begin{proposition}\label{p:vol_cov}
    Let $f$ and $\ell<\ell_c$ satisfy Assumptions~\aref{a:basic}, \aref{a:cov_pos}, \aref{a:trunc_arm_decay} and either (1)~\aref{a:cov_decay_pol}, or (2)~\aref{a:cov_decay_exp}. Then there exists $C,c>0$ such that for all $x,y\in\Z^d$
    \[
        0\leq\Cov[E_x,E_y]\leq
        \begin{cases}
            C(1+\lvert x-y\rvert)^{-\beta}(\log(1+\lvert x-y\rvert))^{2d} &\text{in case (1),}\\
            C\exp(-c\lvert x-y\rvert) &\text{in case (2).}
        \end{cases}
    \]
    Moreover the same bounds apply to $\Cov[V_x,V_y]$ whenever $x,y\in\Z^d$.
\end{proposition}

\begin{proof}
    Given $x\in\R^d$ and $R\geq 1$, we define
    \[
        E_x(R):=\Big\{x\overset{\{f\geq\ell\}}{\longleftrightarrow}x+\partial \Lambda_R\Big\}.
    \]
    Since $E_x$ is an increasing event with respect to $f$, which is associated, and $E_x\subset E_x(R)$, for any $R>0$ we have
    \begin{equation}\label{e:vol_cov1}
    \begin{aligned}
        0\leq\Cov[E_x,E_y]&=\Cov[E_x(R),E_y(R)]-\Cov[E_x(R)\setminus E_x,E_y(R)]-\Cov[E_x,E_y(R)\setminus E_y]\\
        &\leq \Cov[E_x(R),E_y(R)]+2\P(E_0(R)\setminus E_0).
    \end{aligned}
    \end{equation}
    Applying Assumption~\aref{a:trunc_arm_decay}, we have
    \begin{equation}\label{e:vol_cov2}
        \P(E_0(R)\setminus E_0)\leq Ce^{-cR}.
    \end{equation}

    We now assume $R<\lvert x-y\rvert/3$ and define $\mathcal{F}_R(x)$ to be the stratification of $x+\Lambda_R$ consisting of $\{x\}$, the interior of $\Lambda_R\setminus\{x\}$ and the open boundary faces of $\Lambda_R$ of all dimensions. The event $E_x(R)$ is a topological event with respect to this stratification (i.e., it is a measurable function of the random stratified isotopy class of $\{f\geq\ell\}$ with respect to this stratification). Therefore by the covariance formula (Theorem~\ref{t:cov_formula_events})
    \[
        \Cov[E_x(R),E_y(R)]=\int_{\mathcal{F}_R(x)\times\mathcal{F}_R(y)}K(w-z)\int_0^1\mu^t_{w,z}(d_w\ind_{E_x(R)}d_z^t\ind_{E_y(R)})\;dtdwdz.
    \]
    By definition, for any event $A$ and $u\in\R^d$ $\lvert d_u\ind_A\rvert\leq 1$ and so using the bound on pivotal measures (Proposition~\ref{p:intensity_integrable})
    \begin{equation}\label{e:vol_cov3}
        \Cov[E_x(R),E_y(R)]\leq C\int_{\mathcal{F}_R(x)\times\mathcal{F}_R(y)}\lvert K(w-z)\rvert \;dwdz\leq C R^{2d}\sup_{\lvert z\rvert>\lvert x-y\rvert-2R}\lvert K(z)\rvert.
    \end{equation}
    In the case that Assumption~\aref{a:cov_decay_pol} holds, we set $R=\overline{C}\log\lvert x-y\rvert$ for a sufficiently large constant $\overline{C}$, while in the case that Assumption~\aref{a:cov_decay_exp} holds, we set $R=\lvert x-y\rvert/4$. Then combining \eqref{e:vol_cov1}-\eqref{e:vol_cov3} we see that for $\lvert x-y\rvert$ sufficiently large
    \[
        \Cov[E_x,E_y]\leq \begin{cases}
            C\lvert x-y\rvert^{-\beta}(\log\lvert x-y\rvert)^{2d} &\text{in case (1),}\\
            C\exp(-c\lvert x-y\rvert) &\text{in case (2).}
        \end{cases}
    \]
    Increasing $C$, if necessary, proves the first statement of the proposition for all $x,y\in\R^d$.

    Now suppose that $v,w\in\Z^d$. By linearity
    \[
        \Cov[V_v,V_w]=\Cov\Big[\int_{v/2+[0,1/2]^d}\ind_{E_x}\;dx,\int_{w/2+[0,1/2]^d}\ind_{E_y}\;dy\Big]=\int_{v/2+[0,1/2]^d}\int_{w/2+[0,1/2]^d}\Cov[E_x,E_y]\;dxdy.
    \]
    Combining this with the first statement completes the proof of the proposition.
\end{proof}

We may now prove our limit theorems for the volume of the unbounded component:

\begin{proof}[Proof of Theorem~\ref{t:vol}]
    We first consider the asymptotic behaviour of the variance. By the definition of $\sigma^2$ in \eqref{e:vol_sigma}, stationarity of $f$ and Proposition~\ref{p:vol_cov}, for any $\epsilon>0$
    \[
        0\leq R^d\sigma^2-\Var[\Vol_\infty(\Lambda_R,\ell)]=\int_{\Lambda_R\times\R^d\setminus\Lambda_R}\Cov[E_x,E_y]\;dxdy\leq \int_{\Lambda_R\times\R^d\setminus\Lambda_R}C(1+\lvert x-y\rvert)^{-\beta+\epsilon}\;dxdy.
    \]
    Since $\beta>d$, choosing $\epsilon>0$ sufficiently small, this is bounded by
    \begin{align*}
        \int_{\Lambda_R\setminus\Lambda_{R-1}}\int_{\R^d}C(1+\lvert x-y\rvert^{-\beta+\epsilon}\;dxdy+\int_{\Lambda_{R-1}}C(R-\|x\|_\infty)^{-\beta+d+\epsilon}\;dx\leq C\big(R^{d-1}+R^{-\beta+2d+\epsilon}\big)
    \end{align*}
    where the bound for the second integral above can be derived as in the proof of Proposition~\ref{p:q_stabilisation}. This shows that for any $\epsilon>0$, as $R\to\infty$
    \[
        R^{-d}\Var[\Vol_\infty(\Lambda_R,\ell)]=\sigma^2+O\big(R^{-\min\{1,\beta-d-\epsilon\}}\big)
    \]
    establishing (a slightly stronger version of) the second statement of Theorem~\ref{t:vol}.
    
    We now show positivity of $\sigma$. By Proposition~\ref{p:vol_cov} and the assumption that $\ell<\ell_c$, for any $x\in\R^d$
    \[
        \Cov[E_x,E_0]\geq 0\qquad\text{and}\qquad\Cov[E_0,E_0]=\theta(\ell)(1-\theta(\ell))>0.
    \]
    Hence by the definition of $\sigma$, positivity will follow if we can show that $x\mapsto\Cov[E_x,E_0]$ is continuous. For $x,y\in\R^d$,
    \[
        \lvert\Cov[E_x,E_0]-\Cov[E_y,E_0]\rvert\leq\lvert\Cov[E_x\Delta E_y,E_0]\rvert\leq 2\P(E_x\Delta E_y)
    \]
    where $\Delta$ denotes the symmetric difference operator. On $E_x\Delta E_y$, either $x$ is contained in an unbounded component of $\{f\geq\ell\}$ and $y$ is not, or vice versa. In either case, the line segment from $x$ to $y$ must contain some value $z$ for which $f(z)=\ell$. Hence $\P(E_x\Delta E_y)$ is bounded above by the expected number of zeroes of $f-\ell$ on the line segment from $x$ to $y$. By stationarity of $f$ and the Kac-Rice formula, the latter quantity is at most $C\lvert x-y\rvert$ for some $C>0$. Combining these observations with the previous equation establishes continuity of $x\mapsto\Cov[E_x,E_0]$ and hence positivity of $\sigma$.

    Turning to the limiting distribution; since the random field $(V_w)_{w\in\Z^d}$ is stationary, associated (Lemma~\ref{l:assoc}) and satisfies the `finite susceptibility' condition
    \[
        \sum_{w\in\Z^d}\Cov[V_w,V_0]\leq\sum_{w\in\Z^d}C(1+\lvert w\rvert)^{-\beta+\epsilon}<\infty,
    \]
    Newman's CLT (Theorem~2 of \cite{new80}) states that as $n\to\infty$
    \begin{equation}\label{e:vol_clt1}
        \frac{\Vol_\infty(\Lambda_n,\ell)-\theta(\ell)n^d}{n^{d/2}}=\frac{\sum_{w\in\Z^d\cap[-n,n)^d}V_w-\E[V_w]}{n^{d/2}}\overset{d}{\longrightarrow}\sigma Z
    \end{equation}
    where $Z\sim\mathcal{N}(0,1)$. It remains to show that the convergence also holds for non-integer valued box-lengths $R$. Letting $\widetilde{\Vol}_\infty(\Lambda_R):=R^{-d/2}\overline{\Vol}_\infty(\Lambda_R)$ and recalling that $\overline{X}:=X-\E[X]$, we observe that
    \begin{equation}\label{e:vol_clt2}
        \widetilde{\Vol}_\infty(\Lambda_R)-\widetilde{\Vol}_\infty(\Lambda_{\lfloor R\rfloor})=\frac{\overline{\Vol}_\infty(\Lambda_R\setminus\Lambda_{\lfloor R\rfloor})}{\lfloor R\rfloor^{d/2}}+(1-(R/\lfloor R\rfloor)^{d/2})\frac{\overline{\Vol}_\infty(\Lambda_{R})}{R^{d/2}}.
    \end{equation}
    The variance convergence result previously established implies that the second term here converges to zero in $L^2$ as $R\to\infty$. Then by linearity and Proposition~\ref{p:vol_cov}
    \begin{align*}
        \Var\big[\lfloor R\rfloor^{-d/2}\overline{\Vol}_\infty(\Lambda_R\setminus\Lambda_{\lfloor R\rfloor})\big]&=\lfloor R\rfloor^{-d}\int_{(\Lambda_R\setminus\Lambda_{\lfloor R\rfloor})^2}\Cov[E_x,E_y]\;dxdy\\
        &\leq \lfloor R\rfloor^{-d}\lvert \Lambda_R\setminus\Lambda_{\lfloor R\rfloor}\rvert\int_{\R^d}C(1+\lvert x\rvert)^{-\beta+\epsilon}\;dx\to 0.
    \end{align*}
    Hence \eqref{e:vol_clt2} converges to zero in probability, and so applying the continuous mapping theorem to \eqref{e:vol_clt1} verifies the central limit theorem on non-integer scales.

    By stationarity and Proposition~\ref{p:vol_cov}, for any $\epsilon>0$ and $k\geq 1$
    \begin{equation}\label{e:cox_grimmett}
        \sup_{w\in\Z^d}\sum_{v\in\Z^d\;:\;\lvert v-w\rvert\geq k }\lvert\Cov[V_v,V_w]\rvert\leq C_\epsilon k^{-\beta+d+\epsilon}.
    \end{equation}
    (The left-hand expression here is known as the `Cox-Grimmett' coefficient.) Theorem~3 of \cite{bul96} bounds the rate of convergence in the CLT for discrete, uniformly bounded, associated fields in terms of the Cox-Grimmett coefficient. Applying this to $\Vol_\infty(\Lambda_n)=\sum_{v\in\Z^d\cap[-n,n)^d}V_v$, it yields for any $\delta>0$
    \[
        d_\mathrm{Kol}\Big(\frac{\overline{\Vol}_\infty(\Lambda_n,\ell)}{\sqrt{\Var[\Vol_\infty(\Lambda_n,\ell)]}},Z\Big)\leq C_\delta n^d\Var[\Vol_\infty(\Lambda_n)]^{-3/2+\frac{d}{\lambda+2d}+\delta}
    \]
    where $\lambda$ is the rate of decay of the Cox-Grimmett coefficient. By the previously established variance convergence, for $n$ sufficiently large $\Vol_\infty(\Lambda_n)\geq (\sigma/2)n^d$. Then setting $\lambda=\beta-d-\epsilon$ for sufficiently small $\epsilon$ proves the first statement of the quantitative CLT in Theorem~\ref{t:vol}. If we impose the stronger Assumption~\aref{a:cov_decay_exp}, then by Proposition~\ref{p:vol_cov} the Cox-Grimmett coefficient decays exponentially. In this case, Theorem~4 of \cite{bul96} yields the improved bound in the second statement of the quantitative CLT.
\end{proof}

\bigskip
\printbibliography

\end{document}